\documentclass[10pt,a4paper]{article}
\usepackage{ae,lmodern}
\usepackage{babel}
\usepackage[utf8x]{inputenc}
\usepackage[T1]{fontenc}
\usepackage{amsmath}
\usepackage{amsfonts}
\usepackage{amssymb}
\usepackage{amsthm}
\usepackage{mathtools}
\usepackage{relsize,scalerel}
\usepackage{graphicx}
\usepackage[top=2cm,bottom=2cm,left=2cm,right=2cm]{geometry}
\usepackage{fancyhdr}
\usepackage{enumitem}
\usepackage{stmaryrd}
\usepackage[dvipsnames]{xcolor}
\usepackage[nottoc, notlof, notlot]{tocbibind}
\usepackage{dsfont}
\usepackage{enumitem}
\usepackage{euscript}
\usepackage{mathrsfs}

\usepackage{soul}

\usepackage{xcolor}
\definecolor{unbleu}{rgb}{0.03, 0.15, 0.4}
\definecolor{monvert}{rgb}{0.0,.5,0.0}
\definecolor{britishracinggreen}{rgb}{0.0, 0.26, 0.15}
\definecolor{monbleu}{rgb}{0,.2,.8}
\definecolor{monautrebleu}{rgb}{0,0.4,.75}
\definecolor{applegreen}{rgb}{0.55, 0.71, 0.0}
\definecolor{monrouge}{rgb}{0.8, 0.0, 0.0} 
\definecolor{cadmiumgreen}{rgb}{0.0, 0.42, 0.24}
\definecolor{royalblue(traditional)}{rgb}{0.0, 0.14, 0.4}
\definecolor{black}{rgb}{0.0, 0.0, 0.0}
\definecolor{sepia}{rgb}{0.44, 0.26, 0.08}
\definecolor{teagreen}{rgb}{0.82, 0.94, 0.75}
\definecolor{yellow-green}{rgb}{0.6, 0.8, 0.2}
\definecolor{azure(colorwheel)}{rgb}{0.0, 0.5, 1.0}
\definecolor{awesome}{rgb}{1.0, 0.13, 0.32}
\definecolor{cadmiumyellow}{rgb}{1.0, 0.96, 0.0}
\definecolor{carrotorange}{rgb}{0.93, 0.57, 0.13}
\definecolor{green-yellow}{rgb}{0.68, 1.0, 0.18}
\definecolor{huntergreen}{rgb}{0.21, 0.37, 0.23}

\usepackage{hyperref}
\hypersetup{linkcolor=royalblue(traditional),citecolor=cadmiumgreen, urlcolor=azure(colorwheel), colorlinks=true}

\allowdisplaybreaks

\numberwithin{equation}{section}
\input cyracc.def

\newcommand{\widesim}{\scaleobj{1.3}{\sim}}

\makeatletter
\newcommand{\uset}[3][0ex]{%
  \mathrel{\mathop{#2}\limits_{
    \vbox to#1{\kern-6\ex@
    \hbox{$\scriptstyle#3$}\vss}}}}
\makeatother

\newcommand{\lip}{\operatorname{Lip}}

\newcommand{\dd}{\mathrm{d}}

\newcommand{\leb}{\operatorname{Leb}}

\newcommand{\diam}{\operatorname{\mathrm{diam}}}
\newcommand{\fPp}{\operatorname{FPP}}
\newcommand{\cfPp}{\operatorname{CFPP}}
\newcommand{\geo}{\operatorname{Geo}}
\newcommand{\RPP}{\operatorname{RPP}}
\newcommand{\DRPP}{\operatorname{DRPP}}

\newcommand{\lr}{\scaleto{r}{5pt}}
\newcommand{\mr}{\scaleto{r}{4pt}} 

\newcommand{\RV}{\mathrm{RV}}
\newcommand\eqlaw{\stackrel{\mathrm{{{\scriptscriptstyle law}}}}{=}}
 
	\theoremstyle{definition}
	\newtheorem{defn}{Definition}[section]
	
    \newtheorem*{notations}{Notations}
    
	\theoremstyle{plain}
	\newtheorem{thm}{Theorem}[section]
        \newtheorem{prop}{Proposition}[section]
	\newtheorem{cor}[prop]{Corollary}
        \newtheorem{lem}[prop]{Lemma}
        
	\theoremstyle{plain}
	
	\newtheorem{rem}{Remark}[section]

\begin{document}

\title{The Fractional Poisson Process and Other Limit Point Processes for Rare Events in Infinite Ergodic Theory}


\author{Dylan Bansard-Tresse \vspace{0.2cm}\\ 
CPHT, CNRS, École polytechnique, \\ Institut Polytechnique de Paris, 91120 Palaiseau, France\\
\texttt{dylan.bansard-tresse@polytechnique.edu}\\
\small{\url{https://sites.google.com/view/bansard-tresse}}}


\maketitle
\begin{abstract}
We study the process of suitably normalized successive return times to rare events in the setting of 
infinite-measure preserving dynamical systems. Specifically, we consider small neighborhoods of points whose 
measure tends to zero. We obtain two types of results. First, we conduct a detailed study of a class of interval maps with a neutral fixed point and we fully characterize the limit processes for all points, highlighting a trichotomy and the emergence of the fractional (possibly compound) Poisson process. 
This is the first time that these processes have been explicitly identified in this context.
Second, we prove an abstract result that offers an explanation for the
emergence of the fractional Poisson process, as the unique fixed point of a functional equation, drawing a 
parallel with the well-established behavior of the Poisson process in finite-measure preserving dynamical systems.
\end{abstract}

{\small
\tableofcontents
}

\section{Introduction}

We consider a measure-theoretic dynamical system $(X, \mathscr{B}, \mu, T)$ where $T$ acts on the phase 
space $X$ and preserves the measure $\mu$ which can be finite or infinite. We are interested
in ``asymptotically rare events'', meaning that we consider sequences
$(B_n)_{n\geq 0}$ of measurable sets such that $\mu(B_n)\to 0$ as
$n\to\infty$. Denoting $\lr_{B_n}^{(k)}$ the $k$-th return to $B_n$, we want to find the 
asymptotic behavior of the point process
\[
N_{B_n}^{\gamma} := \sum_{k\geq 1} \delta_{\gamma(\mu(B_n))\,\mr_{B_n}^{(k)}}
\]
where $\gamma$ is an appropriate scaling function. Such point processes are often 
called ``Rare Event Point Processes'' (\text{REPP} for short). When $\mu$ is a probability 
measure, extensive research has focused on the limiting distribution of the first 
hitting or return time, with contributions from many authors (see, {\em e.g.}, 
\cite{Sau09_Survey_AnIntroductionToQuantitativeRecurrenceInDynamicalSystems,
Haydn13_EntryAndReturnTimesDistribution} and references therein), or the whole sequence of return times seen as a process (see \cite{LFFF16} and  references therein). 
For differentiable dynamical systems, the measure $\mu$ is either 
absolutely continuous with respect to the Lebesgue measure or, more broadly, is an SRB 
(Sinai-Ruelle-Bowen) measure; see, {\em e.g.}, \cite{ChazottesCollet13, SuBunimovitch22}.
Due to Ka\v{c}'s theorem, the right scaling consists in taking $\gamma = \text{id}$. In 
this context, if mixing is sufficiently strong or correlations decay sufficiently fast, when the sets $B_n$ are either balls or 
cylinder sets shrinking to a point $x$ in the phase space and $x$ is $\mu$-generic, 
then $N^{\text{id}}_{B_n}$ converges towards the Poisson point process ($\text{PPP}$ 
for short), with this convergence understood when $N^{id}_{B_n}$
is considered as a random variable on the probability space $(X, \mathscr{B}, \mu)$
or on the restricted probability spaces $(B_n,\mathscr{B}\cap B_n, \mu_{B_n})$. 
These cases are commonly referred to as ``hitting REPP'' and ``return REPP'', respectively.

Hitting and return REPPs are intrinsically tied, converging to the same limit (see \cite{HLV05} for the first 
hitting/return time relationship and \cite{Mar17,Zwe16} for the point process version). 
The Poisson point process is the unique fixed point of the equation connecting the two 
limits, thereby confirming its prominent role as the anticipated limit point process 
for sufficiently mixing systems.

Yet, not all dynamical systems of interest preserve a probability measure, and some indeed preserve a $\sigma$-finite measure with infinite mass.
Examples include null-recurrent Markov chains and Markov shifts 
\cite{Sar01_NullRecurrentPotentials}, interval maps with indifferent fixed points 
\cite{Aar97, Tha80_EstimatesInvariantDensities}, $\mathbb{Z}^d$-extensions of probability
preserving systems or billiards with cusps. For conservative ergodic systems (see Section \ref{section:preliminaries} for 
precise definitions), every set of positive measure is visited infinitely often by almost 
every orbit. Consequently, $N_{B_n}^{\gamma}$ is well defined, and the question of its 
limiting behavior arises naturally. 
Research on infinite measure systems has been more limited but has seen significant growth in recent years.
 According to Ka\v{c}'s formula, the 
mean return time to any finite measure set is infinite, making the identity scaling 
$\gamma = \text{id}$ inappropriate. Nevertheless, suitable scalings can still be 
identified \cite{RZ20}. The convergence of first hitting and return times has been 
investigated, revealing non-exponential limiting laws for various infinite measure 
preserving dynamical systems and natural targets \cite{BZ01, PeneSaussol10_BackToBallsInBilliards, RZ20, Yas18, 
Yas24_quantitativerecurrencezextensionthreedimensional}, with fewer studies considering 
the whole sequence of successive returns \cite{PSZ13}. 
Recent work \cite{PS23} has shed light on the limiting behavior of point processes for
$\mathbb{Z}$-extensions of subshifts of finite type.

In this paper, we encompass both an in-depth analysis of a paradigmatic interval map with a 
neutral fixed point and a general theoretical framework.  We will now present our results in a 
fairly informal way, deferring the precise statements until afterwards.

We choose to focus on the following class of interval maps with a neutral fixed point as it 
allows for a clear illustration of our main results without the need for overly technical 
considerations. Namely, we consider $X=[0,1]$ and
\begin{equation}\label{def:LSVmap}
T_p(x) =
\begin{cases}
x + 2^px^{p+1}, & 0\leq x < 1/2,\\
2x - 1, & 1/2\leq x \leq 1,    
\end{cases}    
\end{equation}
where $0$ is the neutral fixed point, and $p$ is a nonnegative parameter. This map preserves an absolutely continuous measure which is finite for $p<1$ and infinite for $p\geq 1$ (see more details below).
This map is commonly referred to as Manneville-Pomeau map or LSV map in the literature, following Liverani, Saussol and Vaienti \cite{LSV97}.
This paradigmatic one-parameter family of maps displays a rich spectrum of statistical behaviors in both finite (see {\em e.g.} \cite{Gouezel04_PhD}) and infinite (see {\em e.g.} \cite{Tha80_EstimatesInvariantDensities, Tha83}) contexts. 

Our first result establishes the \emph{fractional Poisson process} as the limiting point 
process for cylinders shrinking to \emph{generic points} of the invariant measure in the 
infinite measure case (see Theorem \ref {thm:REPP_LSV_nonperiodic_points}, p.\pageref{thm:REPP_LSV_nonperiodic_points} for the precise statement).\newline

\noindent \textbf{Theorem A.} \textit{For the map \eqref{def:LSVmap} with $p>1$, and for
cylinders $B_n$ shrinking to a point $x$ which is generic for the absolutely continuous invariant measure, 
the point process $N_{B_n}^{\gamma}$ converges in law towards a fractional Poisson process
(parametrized by $p$), both when one starts in $B_n$ or off $B_n$ (return and hitting REPPs, respectively).}\newline

The fractional Poisson process, which we will define below, was developed to extend the framework of the 
Poisson process to systems exhibiting long-term temporal correlations, which prevent exponential behavior 
\cite{Las03} (see also \cite{MS19_Stochastic_models_for_fractional_calculus} for related applications). 
The following result offers an abstract explanation for the emergence of the fractional Poisson process, 
drawing a parallel with the well-established behavior of the Poisson process in finite-measure preserving 
dynamical systems, as previously mentioned. It generalizes the results of \cite{Zwe16} to the infinite 
measure situation and builds up on \cite{RZ20} where only the first return is considered.\newline

\noindent \textbf{Theorem B} (Abstract result). \textit{Consider a dynamical system
$(X, \mathscr{B}, \mu, T)$ preserving an infinite measure, and let $(B_n)_{n\geq 0}$ be a sequence of asymptotically rare events lying in a ``good'' subset of $X$. Then, 
$N_{B_n}^{\gamma}$ when one starts off $B_n$ (hitting REPP) converges in law if and only if
$N_{B_n}^{\gamma}$ when one starts in $B_n$ (return REPP) converges in law.
If convergence takes place, both limits determine one another through a functional equation whose unique fixed point is the fractional Poisson process (with parameter related to $\gamma$).} \newline

For the precise statement, see Theorem \ref{thm:HTS-REPP_vs_RTS-REPP_infinite_measure_renormalized_measure}, p.\pageref{thm:HTS-REPP_vs_RTS-REPP_infinite_measure_renormalized_measure}. It is worth noting that restricting to  ``good'' subsets is an intrinsic limitation arising in infinite ergodic theory, and that the scaling function $\gamma$ is directly linked to the regularly varying function defining the correct normalizing sequence in the analogue of the ergodic theorem for the infinite measure case,  see Section \ref{section:preliminaries} for further details.


In the context of finite-measure preserving dynamical systems, the study of non-generic points, 
especially periodic points, has gained significant attention in recent years, particularly in the context of 
extreme value theory. It is not surprising that these periodic points give rise to clustering phenomena 
leading to a compound Poisson limit process for hitting and return times. 
The intensity and multiplicity of Poisson compound processes are characterized by a parameter $\theta$, lying between 0 and 1, known as the ``extremal index'' (see \cite{LFFF16} and references therein).
Besides the generic points of the invariant measure under consideration and periodic points, there also 
exist generic points for other invariant measures, as well as points that are not generic for any measure.
For several classes of probability-preserving dynamical systems with sufficiently rapid decay of correlations, a 
\emph{dichotomy} arises in the asymptotic behavior of points: periodic points lead to a compound Poisson limit 
process, while \emph{all} other points give rise to a Poisson limit process; see {\em e.g.}, \cite{LFFF16}
(Rychlik and Gibbs-Markov maps), \cite{FFTV16} (for the map \eqref{def:LSVmap} with $p<1$) and \cite{BF23, DT23}
(quadratic maps with Misiurewicz parameters).

In the context of infinite measure preserving dynamical systems, the only existing study of certain non-
generic points is done for hyperbolic periodic points of prototypical null-recurrent interval maps 
\cite{RZ20}.
In this paper, we conduct a thorough investigation of the application of the interval \eqref{def:LSVmap} and demonstrate 
that the asymptotic behavior of successive return times exhibits not a dichotomy but a trichotomy, involving fractional 
Poisson processes, compound fractional Poisson processes, and a third type of process (for which no standard nomenclature 
has been established). This result represents a novel contribution to the theory of 
infinite-measure dynamical systems.\newline

\noindent \textbf{Theorem C.} \textit{
Consider the map \eqref{def:LSVmap} with $p =: 1/\alpha > 1$ with its (infinite) absolutely continuous 
invariant measure. Let $x \in (0,1]$ and $(B_n)$ be the cylinders shrinking to $x$. Then we have the 
following trichotomy: 
\begin{itemize}
\item If $x$ is a periodic point, then $N_{B_n}^{\gamma}$ converges in law towards a 
compound fractional 
Poisson process $\cfPp_{\alpha}$ with an extremal index $\theta \in (0,1)$ depending 
only on $x$ (see Theorem \ref{thm:REPP_LSV_periodic_points}, 
p.\pageref{thm:REPP_LSV_periodic_points}).
\item If $x$ is a preimage of $0$,  then $N_{B_n}^{\gamma}$ converges in law towards 
a point process $N_x$ 
depending on the point $x$ (see Theorem \ref{thm:REPP_LSV_preimages_of_0}, 
p.\pageref{thm:REPP_LSV_preimages_of_0}).
\item Otherwise, $N_{B_n}^{\gamma}$ converges in law towards the fractional Poisson 
process $\fPp_{\alpha}$ (see Theorem \ref{thm:REPP_LSV_nonperiodic_points}, 
p.\pageref{thm:REPP_LSV_nonperiodic_points}).
\end{itemize}}

This theorem generalizes Theorem A. As in that theorem, the parameter $\gamma$ is determined by $\alpha$, which in turn is determined by $p$.
In the case of finite measure, {\em i.e.}, $p < 1$, prior results \cite{FFTV16, Zwe18} have shown that two 
types of limiting behavior emerge: a compound Poisson process appears around periodic points, while a 
standard Poisson process describes almost all other points except for the fixed point at $0$. This point, 
being indifferent, is distinguished by an extremal index $\theta = 0$, producing an infinite cluster. 
Interestingly, for shrinking neighborhoods around $0$, a different scaling allows the hitting time 
distribution to converge to the exponential distribution. However, for $p > 1$, the limiting distribution 
deviates from the exponential law \cite{Zwe08}. Notably, unlike the finite measure case, as established in 
Theorem C, preimages of $0$ give rise to limit point processes that do not align with the expected 
fractional Poisson process $\fPp_{\alpha}$. Instead, these processes are obtained through thinning and 
rescaling of a specific renewal process.

Finally, in the ``barely infinite case'' ($p=1$), we prove a similar dichotomy to that found in the finite-measure 
case. Our approach parallels the method for $p>1$, with the essential distinction that the preimages of 
$0$ demand specific handling, differing from that of other points. Here, however, we benefit from the 
persistence of the exponential distribution as the limiting law, which continues to apply for neighborhoods 
around $0$ that are scaled down \cite{CG93_Statisticsofclosevisitstotheindifferentfixedpointofanintervalmap}.\newline

\noindent \textbf{Theorem D.} \textit{Consider the map \eqref{def:LSVmap} with $p =1$ with its (infinite) 
Let $x\in (0,1]$ and $(B_n)$ be the cylinders shrinking to $x$. Then we have the following dichotomy:
\begin{itemize}
\item If $x$ is a periodic point, then $N_{B_n}^{\gamma}$ converges in law towards a 
compound Poisson point process with an extremal index $\theta \in (0,1)$ depending 
only on $x$ (see Theorem \ref{thm:REPP_LSV_periodic_points}, 
p.\pageref{thm:REPP_LSV_periodic_points}).
\item If $x$ is not a periodic point, then $N_{B_n}^{\gamma}$ converges in law 
towards the fractional Poisson process (see Theorems 
\ref{thm:REPP_LSV_nonperiodic_points} p.\pageref{thm:REPP_LSV_nonperiodic_points} and 
\ref{thm:REPP_preimages_of_0_barely_infinite_case} 
p.\pageref{thm:REPP_preimages_of_0_barely_infinite_case}).
\end{itemize}}

This paper is structured as follows: Section \ref{section:statements_of_results} states our main results, 
Section \ref{section:proofs_abstract_results} provides proofs of the abstract theorems relating hitting and 
return REPPs and establishes convergence to fractional and compound fractional Poisson processes, Section 
\ref{section:proofs_LSV_map} focuses on convergence for neighborhoods of all points for the map 
\eqref{def:LSVmap} and Section \ref{section:discussion_Pene_Saussol} discusses related work and future 
research directions.

\section{Statement of results}
\label{section:statements_of_results}

\subsection{Preliminaries}\label{section:preliminaries}

Let $(X,\mathscr{B},\mu, T)$ be a measure-theoretic dynamical system. This means that $(X, \mathscr{B}, \mu)$ 
is a measure space and the self-map $T : X \to X$ leaves the measure $\mu$ invariant (\textit{i.e.}, the 
push forward $T_{\#}\mu$ of $\mu$ by $T$ is equal to $\mu$).  
Assume that $\mu$ is $\sigma$-finite and $\mu(X) = +\infty$. 
The transfer operator $\widehat{T} : L^1(\mu) \to L^1(\mu)$ of the system is defined via the following identity : $\forall f\in L^1(\mu), \; \forall g\in L^{\infty}
(\mu), \; \int f\cdot(g\circ T)\,\dd\mu = \int (\widehat{T}f)\cdot g\,\dd\mu$. 
We say that $(X,\mathscr{B},\mu, T)$ is a ``conservative ergodic measure preserving 
transformation'' (CEMPT for short) if $\sum_{k\geq 0} \widehat{T}^ku = +\infty$ $\mu$-
a.e. for all 
$u \in L^1_+(\mu) := \big\{ u\in L^1(\mu)\;|\; u\geq 0,\, \int u\,\dd\mu > 0\big\}$ or, 
equivalently, if this is true for all
$u \in \mathcal{D}(\mu) := \big\{u \in L^1(\mu)\;|\; u\geq 0,\, \int u\,\dd\mu = 1\big\}$ (see \cite[Propostion 1.3.2]{Aar97}).

For $A \in \mathscr{B}$ and $x\in X$, let $r_A(x)= r_A^{(1)}(x)=\inf \{n \geq 1\;|\; T^nx \in A\}$ be the first time the orbits of $x$ hits $A$, and for $k\geq 1$, define inductively the $(k+1)$-th return time to $A$, namely $r^{(k+1)}_A(x) := \inf \{n > r^{(k)}_A(x)\;|\; T^nx \in A\}$, with the convention that $r_A^{(k)}(x) = +\infty$ if $A=\emptyset$. 
If $(X, \mathscr{B}, \mu, T)$ is a CEMPT and $\mu(A) > 0$, then $r_A^{(k)}$ is finite $\mu$-almost everywhere, for each $k\geq 1$. 
\begin{rem}
It is also possible to work with inter-arrival times.
Set $r_A^{\{1\}} := r_A = r_A^{(1)}$ and $r_A^{\{k+1\}} = r_A \circ T^{r_A^{(k)}}$ 
for $k\geq 1$. By construction, we have $r_A^{\{k\}} = r_A^{(k)} - r_A^{(k-1)}$ with 
the convention that $r_A^{(0)} = 0$. 
To distinguish between inter-arrival and return times, we use the super-scripts
$\{k\}$ and $(k)$, respectively.
Again, if $(X, \mathscr{B}, \mu, T)$ is a CEMPT and $\mu(A) > 0$, then $r_A^{\{k\}}$ 
is finite for every $k \geq 1$, $\mu$-almost everywhere. 
\end{rem}

\begin{rem}
Sometimes, we will consider return times on an induced system. For example,
if $A \subset Y \subset X$, we will write $r_A^{Y,(k)}$ for the $k$-th return time to $A$ in the induced dynamical system $(Y, \mathscr{B} \cap Y, T_Y, \mu_Y)$ where $T_Y(x) := T^{r_Y(x)}(x)$ for $x\in X$ and $\mu_Y = \mu(\cdot\cap Y)/\mu(Y)$. In 
particular, on $Y$,
\begin{align*}
\lr\!_A = \sum_{k=0}^{\mr_A^{Y} - 1} \lr_Y \circ T_Y^k = \lr_Y^{(r_A^Y)}.
\end{align*}
\end{rem}

We are now able to define the objects that we are going to study along this article. 

\begin{defn}[Process of hitting/return times]
For $A \in \mathscr{B}$, let 
\begin{align*}
\Phi_A :=  \big(\lr_A^{(1)}, \lr_A^{(2)}, \dots\big).
\end{align*}
Furthermore, for every $d \geq 1$, we will write $\Phi_A^{[d]}$ for $\big(\lr_A^{(1)}, \lr_A^{(2)}, \dots,\lr_A^{(d)}\big)$.
\end{defn}

\begin{rem}
For all $A \in \mathscr{B}$ and $x\in X$, $\Phi_A(x) \in (\overline{\mathbb{R}}_+)^{\mathbb{N}}$ where 
$\overline{\mathbb{R}}_+ := [0,+\infty]$. If $\mu(A) > 0$ and $(X,\mathscr{B}, \mu, T)$ is a CEMPT, then
$\Phi_A \in (\mathbb{R}_+)^{\mathbb{N}}$  $\mu$-almost everywhere. 
However, it is more convenient to see it as a function 
taking values in the compact space $(\overline{\mathbb{R}}_+)^{\mathbb{N}}$ (see \cite{Zwe16} for instance). 
On $(\overline{\mathbb{R}}_+)^{\mathbb{N}}$ one can take the product metric 
$d( (x_n)_{n\geq 1}, (y_n)_{n\geq 1}) = \sum_{n \geq 0} d_1(x_n,y_n)/(2^n(1+d_1(x_n,y_n)))$
where $d_1(s,t) := |e^{-s} - e^{-t}|$ is a standard distance on $\overline{\mathbb{R}}_+$.
Furthermore, as we consider successive return times, $\Phi_A$ is taking values inside the compact subset $\mathcal{W}$ of 
$(\overline{\mathbb{R}}_+)^{\mathbb{N}}$ where
\begin{align*}
\mathcal{W} := \left\{(\phi^{(i)})_{i\geq 1} \in 
(\overline{\mathbb{R}}_+)^{\mathbb{N}} \;\big|\;  \forall i \geq 0,\; \phi^{(i)} \leq 
\phi^{(i+1)} \right\}.
\end{align*}
\end{rem}

For a non-decreasing function $f: \mathbb{R}_+ \to \mathbb{R}_+$ and a process
$\Phi = (\phi_k)_{k\geq 1}$ taking values 
in $(\overline{\mathbb{R}}_+)^{\mathbb{N}}$, we write $f(\Phi)$ the process 
$(f(\phi_k))_{k\geq 1}$ with the convention 
that $f(+\infty) = +\infty$. Note that in this case,
$f(\mathcal{W}) \subseteq \mathcal{W}$.

\begin{defn}[Rare Event Point Process]
For a set $A\in \mathscr{B}$ and a function $\gamma : \mathbb{R}_+ \to \mathbb{R}_+$, the Rare Event Point Process
(REPP) is defined by
\begin{align*}
N_A^{\gamma} := \sum_{k\geq 1} \delta_{\gamma(\mu(A))\, \mr_{{A}}^{(k)}}. 
\end{align*}
\end{defn}
For a CEMPT, if $\mu(A) > 0$, $N_A^{\gamma}$ is well defined for almost every point (all the points such that
$r_{A}^{(k)} <+\infty$ for all $k\geq 1$ or equivalently when $\Phi_A \in (\mathbb{R}_+)^{\mathbb{N}}$) and belongs to 
the set $\mathcal{M}^{\text{Rad}}_{\text{atom}}(\mathbb{R}_+)$ of Radon atomic measures on $\mathbb{R}_+$. 
$\mathcal{M}^{\text{Rad}}_{\text{atom}}(\mathbb{R}_+)$ is endowed with the topology of vague convergence. 

When \(\mu(A) < +\infty\) and \(N_A^{\gamma}\) is treated as a random variable on the probability space
\((A, \mathscr{B} \cap A, \mu_A)\), we refer to it as the return REPP. Conversely, when \(N_A^{\gamma}\) is considered as 
a random variable on \((X, \mathscr{B}, \nu)\), where \(\nu\) is a probability measure absolutely continuous with respect 
to \(\mu\), it is termed the hitting REPP. For a fixed set \(A\), the hitting REPP depends on the specific choice of the 
probability measure \(\nu\). However, as our interest lies in the behavior of \(N_A^{\gamma}\) as \(\mu(A) \to 0\), 
\cite[Corollary 6]{Zwei07} ensures that any limiting distribution, if it exists, is independent of \(\nu\). This result 
justifies the use of the term ``hitting REPP'' without dependence on the particular choice of \(\nu\).

For a CEMPT preserving a \(\sigma\)-finite measure \(\mu\) with infinite mass, it is 
well-established that a direct analogue of the Birkhoff theorem is unattainable, as 
no normalization exists such that the time average along an orbit converges almost 
surely \cite[Theorem 2.4.1]{Aar97}. Nonetheless, many such systems exhibit a related property, which provides insight into their asymptotic behavior.

\begin{defn}[Pointwise dual ergodicity]
\label{defn:pointwise_dual_ergodic}
A CEMPT $(X, \mathscr{B}, \mu, T)$ is said to be pointwise dual ergodic (PDE) if there 
exists a sequence $(a_n)_{n\in \mathbb{N}}$ such that
\begin{equation}
\frac{1}{a_n} \sum_{k = 0}^{n-1} \widehat{T}^k u 
\xrightarrow[n\to +\infty]{\mu-\mathrm{a.e.}} \int u\,\dd\mu\,, \; 
\forall u \in L^1(\mu).
\end{equation}
In this case $(a_n)_{n\in \mathbb{N}}$ is called a normalizing sequence for 
$(X,\mathscr{B}, \mu, T)$. 
\end{defn}

For instance, examples include the Boole map \cite[\S 3.7]{Aar97}, interval maps 
with a finite number of indifferent fixed points \cite{Zwe98} or null recurrent 
Markov shifts \cite{Sar01_NullRecurrentPotentials}. In fact, the PDE property is 
equivalent to the existence of a uniform subset on which the convergence is stronger. 
Such sets are of paramount importance in the study of quantitative recurrence.

\begin{defn}[Uniform set]\label{defn:uniform_set}
A set $Y\in \mathscr{B}$ with $\mu(Y) > 0$ is said to be $f$-uniform for
$f\in L^1(\mu)$ if there exists a sequence $(a_n)_{n\geq 0}$ such that 
\begin{align*}
\frac{1}{a_n}\sum_{k=0}^{n-1}\widehat{T}^k f 
\xrightarrow[n\to + \infty]{L^{\infty}(\mu_Y)} \int f\,\dd\mu\,.
\end{align*}
We say that $Y$ is uniform if it is $f$-uniform for some $f\in L^1(\mu)$.
\end{defn}

The existence of uniform sets from the PDE property is immediate by Egorov's theorem. The proof of the reciprocal can be 
found in \cite[Proposition 3.7.5]{Aar97}. 
  
Additionally, in order to get convergence results, we usually require more information on the normalizing sequence 
through regular variation properties. A measurable function
$a: \mathbb{R}_+ \to \mathbb{R}_+$ is said to be regularly 
varying of index $\alpha \in \mathbb{R}$ at infinity if, for all $y\in\mathbb{R}_+$,
\begin{align*}
\lim_{x\to+\infty} \frac{a(xy)}{a(x)} = y^{\alpha}.
\end{align*}

The notion of regular variation at infinity can be generalized for sequences 
$(u_n)_{n\geq 0}$ by looking at the function $u : x \mapsto u_{\lfloor x \rfloor}$. 
In both cases, we will write $a \in \RV(\alpha)$ or
$(u_n)_{n\geq 0} \in \RV(\alpha)$.
We also say that a function
$b : \mathbb{R}_+ \to \mathbb{R}_+$ is regularly varying of index $\alpha$ at 0 if 
$x\mapsto b(1/x)$ is regularly varying of parameter $\alpha$ at infinity. We write 
$\RV_{0+}(\alpha)$ the set of such functions or simply $\RV(\alpha)$ when the limit $0$ or $+\infty$ is clear from the context.  

Given a normalizing sequence $(a(n))_{n\geq 0} \in \RV(\alpha)$, its asymptotic inverse $b$ is defined by the property $b(a(s))\sim a(b(s)) \sim s$
(see \cite{Bingham89_RegularVariation} for more on regular variation). Furthermore, we define the scaling function
\begin{align}
\label{eq:definition_of_gamma}
\gamma : s \mapsto \frac{1}{b(1/s)}\,,\; \forall s > 0.
\end{align}
If $a \in \RV(\alpha)$, then $\gamma \in \RV_{0+}(1/\alpha)$. The function $\gamma$ 
will be used to scale return times. 

We end this section with several notations.

\begin{notations}
We write $\xRightarrow[]{\mu}$ for the convergence in law under the law $\mu$. When the 
sequence of random variables is defined on different probability spaces with probability measures $\mu_n$, we will write 
$\xRightarrow[]{\mu_n}$ for the convergence in law. For example, if $X_n$ is a sequence of random variables defined on 
$(\Omega_n, \mathscr{B}_n, \mu_n)$ and $X$ is another random variable, $X_n \xRightarrow[]{\mu_n} X$ means that 
$(X_n)_{\#}\mu_n$ converges weakly towards the law of $X$. Finally, when the measure $\mu$ is infinite, we cannot draw 
random variables from it. However, we can still set a convergence, that we will write $\xRightarrow[]{\mathcal{L}(\mu)}$ 
if we have $\xRightarrow[]{\nu}$ for every probability $\nu$ that is absolutely continuous with respect to $\mu$. In this 
case, we say that we have strong convergence in law (see \cite[\S\,3.6]{Aar97} for more details). 
\end{notations}

\begin{notations} 
Every set equality is understood as an equality up to a set of 0 measure. Furthermore, we 
use $\sqcup$ for a disjoint union. We will write $(a_n)_{n\geq 0}$ or  $(a(n))_{n\geq 0}$ for sequence of real numbers, 
depending on the context. 
\end{notations}

We are now prepared to formally present our results. Section \ref{abstract_results} is devoted to general, abstract 
findings on quantitative recurrence for pointwise dual ergodic CEMPTs. 
Section \ref{Paragraph_map_with_one_indifferent_fixed_point} examines a specific family of maps with an 
indifferent fixed point, where we determine the limiting behavior of shrinking targets around every point in the phase 
space.

\subsection{Abstract results}
\label{abstract_results}

\subsubsection{The general relationship between hitting and return REPPs}
 
We say that a sequence $(B_n)_{n\geq 0} \in \mathscr{B}^{\mathbb{N}}$ is a sequence of asymptotically rare events 
with respect to a measure $\mu$ if $\mu(B_n) \xrightarrow[n\to +\infty]{} 0$. 
From now on, we shall consider that asymptotically rare events fulfill the following hypothesis.

\begin{enumerate}[label = $(A\arabic*)_{\alpha}$, leftmargin=*]
\setcounter{enumi}{-1}
\item \label{cond:Living_in_uniform_set} There exists a uniform set $Y$ with normalizing sequence $(a_n)_{n\geq 0} \in \RV(\alpha)$ such that $B_n \subset Y$ for all $n\geq 0$.
\end{enumerate}

We start by showing that, if \ref{cond:Living_in_uniform_set} is satisfied, then the convergence of the hitting REPP implies the convergence of the return REPP and reciprocally, each limit being determined by the other. In particular it generalizes \cite[Theorem 4.3]{RZ20} where the 
authors only consider the first hitting or return.

\begin{thm}
\label{thm:HTS-REPP_vs_RTS-REPP_infinite_measure_renormalized_measure}
Let $(X,\mathscr{B},\mu, T)$ be a PDE CEMPT with $\mu(X)=+\infty$.
Let $(B_n)_{n\geq 0}$ be a sequence of asymptotically rare events satisfying \ref{cond:Living_in_uniform_set}.
Let $\Psi, \widetilde{\Psi}$ be stochastic processes in $(\overline{\mathbb{R}}_+)^{\mathbb{N}}$.
Then, we have 
\begin{align*}
\gamma(\mu(E_n))\, \Phi_{E_n} \xRightarrow[n\to +\infty]{\mu_{E_n}} \widetilde{\Psi} \quad \text{if and only if} \quad 
\gamma(\mu(E_n))\, \Phi_{E_n} \xRightarrow[n\to +\infty]{\mathcal{L}(\mu)} \Psi. 
\end{align*}
Moreover, the distributions of $\Psi$ and $\widetilde{\Psi}$ uniquely determine each other in the following way. For all $d \geq 1$, denoting $F^{[d]}$ 
(respectively $\widetilde{F}^{[d]}$) the distribution function of the $d$ first coordinates of $\Psi$ (respectively 
$\widetilde{\Psi}$), we have, for all $0 \leq t_1 \leq \dots \leq t_d$,
\begin{align}
\label{eq:relation_hitting_return_infinite_measure}
F^{[d]}(t_1,\dots, t_d)
&= \alpha \int_0^{t_1} \biggl( \widetilde{F}^{[d-1]}\left(t_2 -t_1 + x, \dots, t_d - t_1 + x\right) \nonumber\\
& \quad\qquad \qquad - \widetilde{F}^{[d]}\left(x, t_2 - t_1 + x,\dots, t_d - t_1 +x\right)\biggr) 
(t_1 - x)^{\alpha - 1}\,\dd x\,,
\end{align}
with the convention $\widetilde{F}^{[0]} = 1$.
\end{thm}

\begin{rem}
Since, for all $d\geq 1$, $F^{[d]} \leq 1$
and $\alpha \int_0^{t_1} (t_1 - x)^{\alpha - 1} \dd x= t_1^{\alpha}$, for all $t_2, \dots, t_d = t_1$, we necessarily have
\begin{align*}
\lim_{x\to +\infty} \widetilde{F}^{[d]}(x,\dots, x) = \lim_{x\to +\infty} F^{[d - 1]}(x,\dots, x) = 1,
\end{align*}
justifying the fact that $\widetilde{\Psi}$ belongs to $(\mathbb{R}_+)^{\mathbb{N}}$ 
almost surely.
\end{rem}

\begin{rem} Theorem \ref{thm:HTS-REPP_vs_RTS-REPP_infinite_measure_renormalized_measure} considers stochastic processes of hitting/return times taking values in $(\overline{\mathbb{R}}_+)^{\mathbb{N}}$. However, point processes are more natural in our context. If the limits $\Psi$ and $\widetilde{\Psi}$ belong to a natural subset of $(\overline{\mathbb{R}}_+)^{\mathbb{N}}$, it will be enough to ensure the result for REPPs. \label{rem:from_stochastic_process_to_point_process}
Let \[\mathcal{W}' := \left\{(\phi^{(i)})_{i\geq 0} \in (\mathbb{R}_+)^{\mathbb{N}} 
\;\big|\;  \phi^{(i)} \leq \phi^{(i+1)} \; \text{and} \; \lim \phi^{(i)} = +\infty 
\right\}.\] 
and define the following map
\[
\Xi :
\begin{cases}
\mathcal{W} & \longrightarrow \quad \mathcal{M}_{\mathrm{atom}}^{\mathrm{Rad}}(\mathbb{R}_+) \\
\Psi = (\psi^{(i)})_{i\geq 1} & \xmapsto{\quad}
\begin{cases}
\sum_{i\geq 1} \delta_{\psi^{(i)}} & \text{if} \; \Psi \in \mathcal{W}'  \\
0 & \text{otherwise}.
\end{cases}
\end{cases}
\]
Then, $\Xi$ is continuous on $\mathcal{W}'$ (see also \cite[Remark 3.5]{Zwe22}). Thus, by the extended continuous mapping theorem, we immediately get the following corollary.
\end{rem}

\begin{cor}
    \label{cor:equivalence_HTS_RTS_true_Point_Processes}
    Under the same assumptions as in Theorem \ref{thm:HTS-REPP_vs_RTS-REPP_infinite_measure_renormalized_measure}. Assume furthermore that $\Psi, \widetilde{\Psi} \in \mathcal{W}'$ almost surely. Then, for $N = \Xi(\Psi)$ and $\widetilde{N} = \Xi(\widetilde{\Psi})$, we have
    \begin{align*}
        N_{E_n}^{\gamma} \xRightarrow[n\to +\infty]{\mu_{E_n}} \widetilde{N} \quad \text{and} \quad N_{E_n}^{\gamma} \xRightarrow[n\to +\infty]{\mathcal{L}(\mu)} N.
    \end{align*}
    The law of $N$ determines the law of $\widetilde{N}$ and reciprocally.
\end{cor}

\begin{rem}
    \label{rem:saying_that_the_two_sclaings_are_equivalents_for_hREPP_rREPP}
    Here, and in the following, we have chosen to change the renormalization by taking $\gamma$. We could have chosen to keep the normalization by the measure, but instead change the return random variable. As we assumed $\alpha \in (0,1]$, both result are equivalent (see \cite[Lemma 1]{BZ01} for example). The main difference is when $\alpha = 0$ (see \cite[Propostion 4.1]{RZ20}) but we do not consider this case here.
\end{rem}

  With the choice of scaling the stochastic process instead of scaling the measures, it gives the following theorem, equivalent to Theorem \ref{thm:HTS-REPP_vs_RTS-REPP_infinite_measure_renormalized_measure}.
\begin{thm}
\label{thm:HTS-REPP_vs_RTS-REPP_infinite_measure_renormalized_returns}
    Under the  same assumptions as in Theorem \ref{thm:HTS-REPP_vs_RTS-REPP_infinite_measure_renormalized_measure}, we have 
	\begin{align*}
			\mu(E_n)\, a(\Phi_{E_n}) \xRightarrow[n\to + \infty]{\mu_{E_n}} \widetilde{\Phi} \quad \text{if and only if} \quad 
			\mu(E_n)\, a(\Phi_{E_n}) \xRightarrow[n\to + \infty]{\mathcal{L}(\mu)} \Phi.
		\end{align*}
      Furthermore, the laws of $\Phi$ and $\widetilde{\Phi}$ determine one another through the distribution function of their marginals by a change of variable in \eqref{eq:relation_hitting_return_infinite_measure}.
\end{thm}

  In particular, Theorem \ref{thm:HTS-REPP_vs_RTS-REPP_infinite_measure_renormalized_measure} gives back the equivalence result for the first hitting and return times if we take the projection on the first component.

\begin{cor} \textup{\cite[Theorem 4.2]{RZ20}}
\label{cor:equivalence_HTS/RTS_infinite_measure}
    Under the same assumptions as in Theorem \ref{thm:HTS-REPP_vs_RTS-REPP_infinite_measure_renormalized_measure}, we have for $R, \widetilde{R}$ random variables in $\overline{\mathbb{R}}_+$, 
        \begin{align*}
            \gamma(\mu(B_n))\, \lr\!_{B_n} \xRightarrow[n\to +\infty]{\mathcal{L}(\mu)} R \quad \text{if and only if} \quad \gamma(\mu(B_n))\, \lr\!_{B_n}  \xRightarrow[n\to +\infty]{\mu_{B_n}} \widetilde{R}.
        \end{align*} 
        If convergence takes place, then
        \begin{align}
            \label{eq:relation_HTS_RTS_first_return}
            \mathbb{P}(R \leq t) = \alpha \int_0^{t} \big( 1 - \mathbb{P}(\widetilde{R} \leq u)\big)(t - u)^{\alpha - 1}\,\dd u, \; \forall t\geq 0.
        \end{align}
        In terms of the Laplace transform, this equation can be rewritten as
        \begin{align}
            \label{eq:Laplace_transform_HTS_from_RTS}
            \mathbb{E}\left[e^{-sR}\right] = \frac{\Gamma(1+\alpha)}{s^{\alpha}}\Big(1 - \mathbb{E}\Big[e^{-s\widetilde{R}}\,\Big]\Big), \; \forall s > 0,
        \end{align}
        where $\Gamma : z \mapsto \int_0^{+\infty} t^{z - 1}e^{-t}\,\dd t$ is the standard Gamma function.
    \end{cor}

More generally, some properties of a point process $N$ on $\mathbb{R}_+$ are defined 
throughout their evolution equation. In our case, 
\eqref{eq:relation_hitting_return_infinite_measure} can be used to find the evolution 
equation of the limit point process $N$ from the one of the point process 
$\widetilde{N}$. For a point process $N$ and $d\geq 1$,
let $P_{N}(d,t) := \mathbb{P}(N[0,t] = d)$. We see $P_N(d,\cdot)$ as a function 
from $\mathbb{R}_+$ to $[0,1]$. 

Before stating the results, we need to recall some basics notions of fractional 
calculus. For every $\beta > 0$, we define the Riemann-Liouville integral $I^{\beta}$ 
as follows. For every Riemann integrable $f$ and $t\in \mathbb{R}$
\begin{align}
\label{eq:definition_Riemann-Liouville_integral}
\big(I^{\beta}f\big)(t) := (I^{\beta}_{0+} f)(t) 
:= \frac{1}{\Gamma(\beta)}\int_0^t f(x)(t-x)^{\beta - 1}\,\dd x.
\end{align}
Associated to this integral, we can go backwards and define the Caputo derivative for 
differentiable functions $f$ by 
\begin{align}
\label{eq:definition_Caputo_derivative}
\prescript{\mathrm{Cap}}{}{D^{\beta}}f(t) := 
\big(I^{1 - \beta}f'\big)(t)\,,\; \forall t\in 
\mathbb{R}, \, 0 < \beta < 1. 
\end{align}
Then, Theorem \ref{thm:HTS-REPP_vs_RTS-REPP_infinite_measure_renormalized_measure} 
implies the following result.
\begin{cor}
\label{cor:generalized_Kolmogorov-Feller_HTSvsRTS}
Assume the same hypothesis as in Corollary 
\ref{cor:equivalence_HTS_RTS_true_Point_Processes}. Then, we have 
\begin{align}
\label{eq:generalized_Kolmogorov-Feller_HTSvsRTS}
\prescript{\mathrm{Cap}}{}{D^{\alpha}} P_N(d,\cdot) = \Gamma(1+\alpha)
(P_{\widetilde{N}}(d-1, \cdot) - P_{\widetilde{N}}(d, \cdot))\,,\; \forall d\geq 0,
\end{align}
with the convention that $P_{\widetilde{N}}(-1,t) := 0$ for all $t\geq 0$ and 
$P_N(0,0) = 1$.
\end{cor}

\subsubsection{The fractional Poisson process as a common limit law}

  We now focus on characterizing point processes that will serve as potential limits for hitting and return REPP. As a preliminary step, we first revisit key concepts from renewal theory and present some properties of the Fractional Poisson Process.

\begin{defn}[Renewal point process]
Let $W$ be a non negative random variable such that $\mathbb{P}(W > 0) > 0$, and 
$(W_i)_{i\geq 1}$ i.i.d. random variables having the same law as $W$. The renewal 
point process with waiting times $(W_i)_{i\geq 1}$ is defined by
\[
\RPP(W) \eqlaw \sum_{i = 1}^{+\infty} \delta_{T_i},
\]
where $T_{i+1} - T_{i} = W_{i+1}$ for all $i\geq 0$, with the convention $T_0 = 0$.
Note that $\RPP(W) \in \mathcal{M}^{\text{Rad}}_{\text{atom}}(\mathbb{R}_+)$.
\end{defn}

\begin{defn}[Fractional Poisson Process]
\label{defn:Fractional_Poisson_Process}\leavevmode
The Fractional Poisson Process $\fPp_{\alpha}(\lambda)$ of parameters $\alpha \in 
(0,1]$ and $\lambda > 0$ is the renewal point process $\RPP(H_{\alpha}(\lambda))$ 
where $H_{\alpha}(\lambda)$ is a Mittag-Leffler law of the first type characterized 
by its Laplace transform
\[
\mathbb{E}\big[e^{-sH_{\alpha}(\lambda)}\big]
:= \frac{\lambda}{\lambda + s^{\alpha}} \,,\; \forall s \geq 0.
\]
\end{defn}

\begin{rem}\leavevmode
    \begin{enumerate}[label = \normalfont(\roman*)]
        \item When $\lambda = 1$, we will simply write $\fPp_{\alpha}$   and $H_{\alpha}$. 
        \item We can easily see that $H_{\alpha}(\lambda) \eqlaw \frac{1}{\lambda^{1/\alpha}}H_{\alpha}$.
        \item If $\alpha = 1$, then $H_{1}(\lambda)$ is the exponential law of parameter $\lambda$ and $\fPp_1(\lambda)$ is the homogeneous Poisson point process of parameter $\lambda$.
        \item The first type Mittag-Leffler law $H_{\alpha}(\lambda)$ should not be confused with the second type Mittag-Leffler law $Y_{\alpha}$ commonly called Mittag-Leffler law in the context of infinite measure dynamical systems (see \cite[\S\,3.6]{Aar97} for a definition). \footnote{Both are called Mittag-Leffler because they are defined from the Mittag-Leffler function $E_{\alpha}(z) :=  \sum_{k=0}^{+\infty} \frac{z^p}{\Gamma(\alpha k + 1)}$, but we have $\mathbb{E}[e^{zY_{\alpha}}] = E_{\alpha}(\Gamma(1+\alpha) z)$ for $z \in \mathbb{R}$, whereas $\mathbb{P}(H_{\alpha} > t) = E_{\alpha}( - t^{\alpha})$ for $t \geq 0$.} 
        \item When $0 < \alpha < 1$, $\mathbb{E}[H_{\alpha}(\lambda)] = +\infty$.
    \end{enumerate}
\end{rem}

\begin{rem}
    The fractional Poisson process $\fPp_{\alpha}(\lambda)$ was introduced as a fractional generalization of the standard Poisson process through its Kolmogorov-Feller equation \cite{Las03}. In particular, a process $N$ having the law of $\fPp_{\alpha}(\lambda)$ is characterized by the independence of its waiting times and the generalized Kolmogorov-Feller evolution equation 
    \begin{align}
        \label{eq:generalized_Kolmogorov-Feller_definition_FPP}
        \prescript{\mathrm{Cap}}{}{D^{\alpha}} P_N(d,\cdot) = \lambda(P_{N}(d-1, \cdot) - P_{N}(d, \cdot)) \,,\; \forall d\geq 0,
    \end{align}
    with the convention $P_N(-1, \cdot) = 0$ (see \cite[Equation (19)]{Las03} or \cite[Equation (7.10) p. 207]{MS19_Stochastic_models_for_fractional_calculus}).
\end{rem}

For the majority of well-behaved rare events, we anticipate that the hitting REPP and 
return REPP converge to the same limiting point process. Consequently, the point 
processes of primary interest are the fixed points of 
\eqref{eq:relation_hitting_return_infinite_measure}. Here, we see that 
\eqref{eq:generalized_Kolmogorov-Feller_definition_FPP} is exactly 
\eqref{eq:generalized_Kolmogorov-Feller_HTSvsRTS} when $N$ and $\widetilde{N}$ have 
the same law and $\lambda = \Gamma(1+\alpha)$ pointing $\fPp_{\alpha}
(\Gamma(1+\alpha))$ as a potential fixed point. The following proposition confirms 
that this is indeed a fixed point and further ensures its uniqueness.

\begin{prop}
\label{prop:fix_point_equation_FPP}
The fractional Poisson process $\fPp_{\alpha}(\Gamma(1+\alpha))$ is the only process 
such that the distribution functions of its finite-dimensional marginals are the fixed 
points of \eqref{eq:relation_hitting_return_infinite_measure}.
\end{prop}

\begin{rem}
    \label{rem:FPP_PhiFPP}
    In Proposition \ref{prop:fix_point_equation_FPP}, we took a slight liberty in stating that the fractional Poisson process is the fixed point of the equation. It is in fact $\Phi_{\fPp_{\alpha}(\Gamma(1+\alpha))} := (\phi^{(i)})_{i\geq 1}$ taking values in $\overline{\mathbb{R}}_+$ where for every $i\geq 1$, $\phi^{(i)} = \sum_{k=1}^i X_k$ where $(X_k)_{k\geq 1}$ are i.i.d. with common law that of $H_{\alpha}(\Gamma(1+\alpha))$.
    However, if
    \begin{align*}
            \mathcal{W}'' := \big\{(\phi^{(i)})_{i\geq 0} \in (\mathbb{R}_+)^{\mathbb{N}} \;\big|\;  \forall i \geq 0,\; \phi^{(i)} < \phi^{(i+1)} \;\text{and}\; \lim_{i\to +\infty} \phi^{(i)} = +\infty\big\},
    \end{align*}
    we then have $\Phi_{\fPp_{\alpha}(\Gamma(1+\alpha))} \in \mathcal{W}''$ almost surely, so the law of $\Phi_{\fPp_{\alpha}(\Gamma(1+\alpha))}$ and $\fPp_{\alpha}(\Gamma(1+\alpha))$ are uniquely defined by one another. In the sequel, we will not distinguish between the point process, which takes values in $\mathcal{M}_{\mathrm{atom}}^{\mathrm{Rad}}(\mathbb{R}_+)$, and the stochastic process, which takes values in $(\mathbb{R}_+)^{\mathbb{N}}$, hence both will be denoted by $\fPp_{\alpha}$.
\end{rem}

  While the Fractional Point Process will act as the limit in most cases, other point processes naturally emerge. To accommodate situations where the first waiting time differs from subsequent ones or the mass associated to each point is random, we introduce delayed renewal point processes and compound point processes.

\begin{defn}[Delayed renewal point process]
Let $V,W$ be two non negative random variables such that $\mathbb{P}(V > 0), 
\mathbb{P}(W > 0) > 0$, and
let $(W_i)_{i\geq 1}$ be i.i.d. random variables with the same law as $W$. 
The delayed renewal point process with delay $V$ and waiting times $(W_i)_{i\geq 1}$ 
is defined by
\[
\DRPP(V,W) \eqlaw \sum_{i=1}^{+\infty} \delta_{T_i},
\]
where $T_{i+1} - T_i = W_{i+1}$ for all $i\geq 1$ and $T_1 = V$. Note that 
$\DRPP(V,W) \in \mathcal{M}^{\text{Rad}}_{\text{atom}}(\mathbb{R}_+)$.
\end{defn}

\begin{defn}[Compound point process]
For a simple point process $P = \sum_{i=1}^{+\infty}\delta_{T_i}$, we define the 
associated compound point process $c(P)(\pi)$ of multiplicity $\pi$ (where $\pi$ is a 
probability distribution on $\mathbb{N}$) as
\[
c(P)(\pi) \eqlaw \sum_{i=1}^{+\infty}X_i\delta_{T_i},
\]
where $(X_i)_{i\geq 1}$ are i.i.d. random variables distributed according to $\pi$ and
independent of $(T_i)_{i\geq 1}$. 
\end{defn}

\begin{rem}
This is not the standard approach to defining a compound process. Here, we utilize 
the structure of \(\mathbb{R}\) and the fact that the multiplicity \(\pi\) takes 
integer values. However, compound point processes are typically constructed in a more general framework, particularly the compound Poisson process, as described in detail in \cite[Chapter 15]{LastPenrose18_LecturesOnPoissonProcess}. 
\end{rem}

\begin{defn}[Compound Fractional Poisson Process]
    We define $\cfPp_{\alpha}(\lambda, \pi)$ as the compound process with multiplicity $\pi$ associated to the fractional Poisson process $\fPp_{\alpha}(\lambda)$. 
\end{defn}

\begin{rem}
\label{rem:relation_compound_processes_and_renewal_processes}
In our setting, we do not impose the waiting time of a renewal process to be positive 
almost surely, so compound processes can also be characterized as renewal processes. 
For instance, let $H_{\alpha}(\lambda)$ be the waiting time of the fractional Poisson 
process. Let $\theta \in (0,1]$ and let $W_{\alpha,\theta}(\lambda)$ be a non-
negative random variable with distribution function 
\begin{align}
\label{eq:def_of_W_alpha_theta}
\mathbb{P}(W_{\alpha,\theta}(\lambda) \leq t) = 1 - \theta + \theta\, \mathbb{P}
(H_{\alpha}(\lambda)\leq t).
\end{align}
It means that with probability $1 - \theta$, $W_{\alpha,\theta}(\lambda) = 0$ and 
$(W_{\alpha,\theta}(\lambda)|W_{\alpha,\theta}(\lambda) > 0)$ has the same law as 
$H_{\alpha}(\lambda)$. Then, 
\[
\DRPP(H_{\alpha}(\lambda), W_{\alpha,\theta}(\lambda)) \eqlaw \cfPp_{\alpha}(\lambda, \geo(\theta)),
\]
where $\geo(\theta)$ is the positive geometric law, \textit{i.e.} if 
$Y\eqlaw\geo(\theta)$, for all $k\geq 1$,
$\mathbb{P}(Y = k) = \theta(1 - \theta)^{k-1}$.
\end{rem}

To prove convergence of the hitting REPP and the return REPPs, further conditions on 
the sequence of targets $(B_n)_{n\geq 0}$ are necessary.\newline 

\noindent  \textbf{Assumptions $(A)_{\alpha}$.} \newline 
A sequence $(B_n) \in \mathscr{B}^{\mathbb{N}}$ of asymptotically rare events 
satisfies $(A)_{\alpha}$ if it satisfies \ref{cond:Living_in_uniform_set} and the 
following conditions :
\begin{enumerate}[label = $(A\arabic*)_{\alpha}$, leftmargin=*]
\item \label{cond_CFPP:extremal_index} For every $n \geq 1$, we can write 
$B_n = U(B_n) \sqcup Q(B_n)$, and $\lim_{n\to +\infty} \mu(Q(B_n))/\mu(B_n) = \theta 
\in (0,1]$ (the extremal index).
\item \label{cond_CFPP:good_density_after_tau_n} There exists a sequence of 
measurable functions $\tau_n : B_n \to \mathbb{N}$ and a compact subset $\mathcal{U}$ 
of $L^1(\mu)$ such that
\[
\widehat{T^{\tau_n}}\left(\frac{\mathbf{1}_{Q(B_n)}}{\mu(Q(B_n))}\right)\in 
\mathcal{U} \;, \forall n\geq 1.
\]
\item \label{cond_CFPP:tau_n_small_enough} The sequence $(\tau_n)_{n\geq 0}$
satisfies $\gamma(\mu(B_n))\,\tau_n \xRightarrow[n\to +\infty]{\mu_{B_n}} 0,$ where 
$\gamma$ is defined from $(a_n)_{n\geq 0}$ by \eqref{eq:definition_of_gamma}.
\item \label{cond_CFPP:cluster_compatible_tau_n_no_cluster_from_Q} The sequence 
$(Q(B_n))_{n\geq 0}$ is such that $\mu_{Q(B_n)}(\lr_{B_n} < \tau_n) \xrightarrow[n\to 
+\infty]{} 0.$
\end{enumerate}
Furthermore, if $U(B_n) \neq \emptyset$, we have 
\begin{enumerate}[label = $(A\arabic*)_{\alpha}$, leftmargin=*]
\setcounter{enumi}{4}
\item \label{cond_CFPP:cluster_compatible_tau_n_cluster_from_U}
The sequence $(U(B_n))_{n\geq 0}$ is such that $\mu_{U(B_n)}(\lr_{B_n} > 
\tau_n)\xrightarrow[n\to +\infty]{} 0$
\item \label{cond_CFPP:Compatibility_Geometric_law} We have the following limit
\begin{align*}
\left\| \widehat{T_{B_n}}\left(\frac{\mathbf{1}_{U(B_n)}}{\mu(U(B_n))}\right) - 
\frac{1}{\mu(B_n)}\mathbf{1}_{B_n}\right\|_{L^{\infty}(\mu_{B_n})} \xrightarrow[n\to 
+\infty]{} 0.
\end{align*}
\end{enumerate}

Let us analyze these conditions. The assumptions $(A)_{\alpha}$ are identical to those outlined in 
\cite[Theorem 7.2]{RZ20}, where only the first return is considered. We show that these assumptions 
are sufficiently robust to establish the convergence of the point process as well. Assumption 
\ref{cond:Living_in_uniform_set} ensures that Theorem
\ref{thm:HTS-REPP_vs_RTS-REPP_infinite_measure_renormalized_measure}. can be applied effectively.

Condition \ref{cond_CFPP:extremal_index} serves two purposes: it identifies points that are likely 
to form clusters within $U(B_n)$ and points that escape the target $Q(B_n)$, while also 
guaranteeing that the extremal index $\theta$ is well-defined.

Condition \ref{cond_CFPP:good_density_after_tau_n} specifies an appropriate waiting time 
$\tau_n$ such that the density within the target set transforms into a desirable structure under the dynamics. However, $\tau_n$ must remain asymptotically negligible, a requirement addressed by \ref{cond_CFPP:tau_n_small_enough}.

Cluster compatibility is handled through \ref{cond_CFPP:cluster_compatible_tau_n_no_cluster_from_Q}
and \ref{cond_CFPP:cluster_compatible_tau_n_cluster_from_U}, ensuring that $\tau_n$ aligns with clustering behavior.

Finally, Condition \ref{cond_CFPP:Compatibility_Geometric_law}
explains the geometric distribution of multiplicities. While this condition could be generalized to 
allow for other multiplicity distributions, we retain it as stated, since the geometric law is the 
only distribution arising in our examples.

We can now state our main theorem, which establishes that Assumptions
\((A)_{\alpha}\) are sufficient conditions for convergence to the Compound 
Poisson Process.

\begin{thm}
\label{thm:sufficient_conditions_convergence_compound_FPP}
Let $(X,\mathscr{B},\mu, T)$ be a PDE CEMPT with $\mu(X) = +\infty$. Let 
$(B_n)_{n\geq 0}$ be a sequence of asymptotically rare events satisfying 
$(A)_{\alpha}$. 
Then,
    \begin{align*}
        N_{B_n}^{\gamma} \xRightarrow[n\to +\infty]{\mathcal{L}(\mu)} \cfPp_{\alpha}(\theta\Gamma(1+\alpha), \geo(\theta))
    \end{align*}
    and 
    \begin{align*}
        N_{B_n}^{\gamma} \xRightarrow[n\to +\infty]{\mu_{B_n}} \RPP(W_{\alpha,\theta}(\theta\Gamma(1+\alpha))).
    \end{align*}
    In particular, if $\theta = 1$, we have 
    \begin{align*}
        N_{B_n}^{\gamma} \xRightarrow[n\to +\infty]{\mathcal{L}(\mu)} \fPp_{\alpha}(\Gamma(1+\alpha)) \quad \text{and} \quad  N_{B_n}^{\gamma} \xRightarrow[n\to +\infty]{\mu_{B_n}} \fPp_{\alpha}(\Gamma(1+\alpha)).
    \end{align*}
\end{thm}

  The primary goal in the following is to apply Theorem \ref{thm:sufficient_conditions_convergence_compound_FPP} to interval maps with a single indifferent fixed point, using shrinking cylinders as rare events. However, the theorem is broadly applicable to a wider range of maps and asymptotically rare events, making it a versatile tool for many other contexts.

\subsection{Rare event point processes in maps with an indifferent fixed point}
\label{Paragraph_map_with_one_indifferent_fixed_point}

We now focus on the specific case of interval maps with an indifferent fixed point, as defined by \eqref{def:LSVmap}. Recall that $x\in \left[0,1\right]$
and
\[
T_p(x) =
\begin{cases}
x + 2^px^{p+1}, & 0\leq x < 1/2,\\
2x - 1, & 1/2\leq x \leq 1.   
\end{cases} 
\]
We define \(\alpha := p^{-1}\). Once \(p\) is fixed, for simplicity, we omit its 
dependence in the index and we will simply write \(T\) instead of \(T_p\).
Let $T_1$ and $T_2$ be the two diffeomorphic branches 
of $T$ (more precisely their extensions to $[0,1/2]$ and $[1/2,1]$ respectively) and set 
$c_n := T_1^{-n}1$. In particular, $c_0 = 1$ and $c_1 = 1/2$. 

For \( 0 \leq p < 1 \), the absolutely continuous invariant measure \(\mu\) is 
unique (up to scaling) and normalized to a probability measure. In contrast, 
for \( p \geq 1 \), \(\mu\) remains unique but becomes infinite. 
Since 
we are interested by the infinite case, we assume $p \geq 1$ and fix a scaling so 
that $\mu([1/2,1]) = 1$. The density $\rho = \dd\mu/\dd\leb$ is the following: 
\begin{align}
\label{eq:formula_density_LSV_map}
\rho(x) = h_0(x)\frac{x}{x - T_1^{-1}x}, \quad  x\in [0,1],
\end{align}
for some continuous function $h_0$ bounded away from $0$ and $+\infty$. For more 
details about such dynamical systems, see \cite{LSV97},\cite{You99} or 
\cite[Section 3.5]{Alv20} or \cite{Tha80_EstimatesInvariantDensities, Tha83}.  

The dynamical system $([0,1], T, \mu)$ is a PDE CEMPT with normalizing sequence 
$(a_n)_{n\in \mathbb{N}} \in \RV(\alpha)$ and every interval
$I = [c,1]$ with $c> 0$ is uniform (see \cite{Zwe98} for example).\\

Set $Y := [c_1, c_0] = [1/2, 1]$ and define
$\xi := \{[c_{k+1},c_k], \; k\geq 0\} = \{Y\}\cup \{Y^c\cap \{r_Y = k\}, n\geq 1\}$ 
the measurable partition of $[0,1]$ defined from return times to $Y$. Then, for every 
$n\geq 1$, let $\xi_n = \bigvee_{i=0}^{n-1} T^{-i}\xi$. The partition $\xi$ 
dynamically generates the Borel $\sigma$-algebra $\mathscr{B}$ and for a point 
$x \in [0,1]$, we write $\xi_n(x)$ the element of $\xi_n$ that contains $x$. If $x$ 
is not a preimage of $0$, it is easy to see that $\xi_n(x)$ is well defined for every 
$n\geq 1$.

In this case, $\xi_n(x)$ shrinks to $\{x\}$ as $n$ goes to $+\infty$ and will serve 
as our asymptotically rare events in the study of the hitting and return REPP. 
  
For the preimages of $0$, such a definition for asymptotically rare events is not 
possible anymore. However, if $T^kx = 0$, the connected component
$C^n_x$ of $T^{-k}[0,c_n]$ containing $x$ is a union of elements in $\xi_k$ for all $n\geq 1$. As $n$ goes to $+\infty$, $C^n_x$ shrinks towards $x$ and thus $(C^n_x)_{n\geq 1}$ will serve as our 
sequence of asymptotically rare events associated to $x$ in this case. 

\subsubsection{Points that are not preimages of the indifferent fixed point}

If \( x \) is not a preimage of \( 0 \), as determined by the chosen partition \(\xi\), we define the asymptotically rare events as \( B_n := \xi_n(x) \) for all \( n \geq 1 \). This setup results in two distinct behaviors depending on whether \( x \) is periodic. If \( x \) is not periodic, both the hitting and return REPPs converge to the fractional Poisson process. 

\begin{thm}
    \label{thm:REPP_LSV_nonperiodic_points}
    Let $p \geq 1$. Assume $x \in [0,1]$ is not periodic and not a preimage of $0$. For all $n\geq 1$, set $B_n := \xi_n(x)$. Then,
    \begin{align*}
        N_{B_n}^{\gamma} \xRightarrow[n\to +\infty]{\mathcal{L}(\mu)} \fPp_{\alpha}(\Gamma(1+\alpha)) \quad \text{and} \quad N_{B_n}^{\gamma} \xRightarrow[n\to +\infty]{\mu_{B_n}} \fPp_{\alpha}(\Gamma(1+\alpha))\,.
    \end{align*}
\end{thm}

When $x$ is $q$-periodic, clusters of hittings appear and thus we get a 
compound fractional Poisson process in the limit.

\begin{thm}
\label{thm:REPP_LSV_periodic_points}
Let $p \geq 1$. Assume $x \in (0,1]$ is periodic of prime period $q$. For all $n\geq 1$, set $B_n := \xi_n(x)$. Then,
\begin{align*}
N_{B_n}^{\gamma} \xRightarrow[n\to +\infty]{\mathcal{L}(\mu)} \cfPp_{\alpha}
(\theta\Gamma(1+\alpha), \geo(\theta)) \quad \text{and} \quad N_{B_n}^{\gamma} 
\xRightarrow[n\to +\infty]{\mu_{B_n}} \RPP(W_{\alpha,\theta}(\theta\Gamma(1+\alpha)))\,,
\end{align*}
where the extremal index $\theta = 1  - |(T^q)'(x)|^{-1}$.
\end{thm}

The framework closely mirrors that of the finite measure case, with the significant 
difference being that the Poisson point process is replaced by the fractional Poisson 
process. Notably, in the barely infinite case where \( p = 1 \), the limits are 
instead the standard Poisson point process and the compound Poisson point process, 
characterized by parameters \(\theta\) and multiplicity distributed as
\(\geo(\theta)\).

\subsubsection{The indifferent fixed point and its preimages}

If $p > 0$, the neutral fixed point $0$ always requires a special treatment, because 
its extremal index $\theta = 0$ as the neutrality gives a cluster of infinite 
length on average. In the finite measure context, for shrinking neighborhoods $(B_n)_{n\geq 1}$ of $0$, $\mu(B_n)r_{B_n}$ converges almost surely to $+\infty$ when one starts in $B_n$ and to $0$ when one starts off $B_n$. However, the exponential law can be recovered with the scaling $\mu(Q(B_n))$ instead of $\mu(B_n)$, where $Q(B_n):= B_n \backslash T^{-1}B_n$.
\begin{prop} \textup{\cite[Theorem 2]{FFTV16}, \cite[Theorem 5.1]{Zwe18}}
Assume $0 < p < 1$ and let $(B_n)_{n\in \mathbb{N}}$ be a nested sequence of 
intervals neighborhoods of $0$. Then, we have
\begin{align*}
\mu(Q(B_n))\,\lr_{B_n} \xRightarrow[n\to +\infty]{\mu} \mathcal{E}\,,
\end{align*}
where $\mathcal{E}$ is an exponential random variable.
\end{prop}

This highlights the robustness of the exponential law in the finite measure case. The 
key idea of \cite[Theorem 5.1]{Zwe18} is to relate returns to \(0\) with returns to \(1/2\), the only 
other preimage of \(0\). Unlike \(0\), the point \(1/2\) is more favorable for analysis 
since it is not a fixed point and lies within the common inducing set \(Y = [1/2, 1]\). 
Within this set, standard inducing techniques can be employed to establish the 
convergence of the return process to the standard homogeneous Poisson point process for 
shrinking neighborhoods of \(1/2\). The result for shrinking neighborhoods of \(0\) is 
then derived as a consequence of this approach.

When \( p \geq 1 \), there are no established techniques to directly address the point
\( 1/2 \). However, it is possible to analyze hitting times for neighborhoods of
\( 0 \) by viewing them as the apex of the tower constructed via the induction
\( Y = [1/2, 1] \). At this stage, it becomes necessary to distinguish between two 
scenarios: the ``proper infinite case'' (\( p > 1 \)) and the ``barely infinite case''
(\( p = 1 \)).

For \( p > 1 \), \cite[Theorem 2]{Zwe08} demonstrates that a non-degenerate limit can 
be achieved with an appropriately chosen scaling. However, the resulting limit does not 
align with the exponential law observed in the finite measure case or the Mittag-
Leffler law \( H_{\alpha}(\Gamma(1+\alpha))\) typical for generic points.

\begin{thm} \textup{\cite[Theorem 2]{Zwe08}}
\label{thm:Hitting_to_0_infinite}
Let $p > 1$. Let $B_n := [0,\eta_n]$ such that
$\eta_n \xrightarrow[n\to +\infty]{} 0$. Then, 
    \begin{align*}
        \frac{1}{I(\eta_n)}\,\lr_{B_n} \xRightarrow[n\to +\infty]{\mathcal{L}(\mu)} \mathcal{J}_{\alpha},
    \end{align*}
    where $\mathcal{J}_{\alpha}$ is characterized by its Laplace transform
    \begin{align*}
        \mathbb{E}\left[e^{-s\mathcal{J}_{\alpha}}\right] = \frac{1}{e^{-s} + s\int_0^1 y^{-{\alpha}}e^{-sy}\,\dd y}\,,\; \forall s \geq 0,
    \end{align*}
    and 
    \begin{align*}
        I(\eta_n) := \int_{\eta_n}^1 \frac{\dd x}{x - T_1^{-1}(x)}. 
    \end{align*}
\end{thm}

  Using Lemma \ref{lem:Comparison_renormalization_0}, we can rewrite this result in our framework, using the scaling $\gamma$ defined in \eqref{eq:definition_of_gamma}.  
\begin{cor}
\label{cor:HTS_neighborhood_0_normalization_gamma}
Under the same assumptions as in Theorem \ref{thm:Hitting_to_0_infinite}. Then, 
\begin{align*}
\gamma(\mu(T_2^{-1}(B_n)))\,\lr_{B_n} \xRightarrow[n\to +\infty]{\mathcal{L}(\mu)} 
\mathfrak{J}_{\alpha},
\end{align*}
where $\mathfrak{J}_{\alpha} := d_{\alpha}\mathcal{J}_{\alpha}$ and $d_{\alpha} := (\Gamma(1+\alpha)\Gamma(1 - \alpha))^{-1/\alpha} = (\sin(\pi\alpha)/(\pi\alpha))^{1/\alpha}$.
\end{cor}

  Using \eqref{eq:Laplace_transform_HTS_from_RTS}, we can construct a random variable $\widetilde{\mathfrak{J}}_{\alpha}$ associated to $\mathfrak{J}_{\alpha}$. In particular, the Laplace function of $\widetilde{\mathfrak{J}}_{\alpha}$ is the following
\begin{align}
    \label{eq:laplace_transform_J_alpha_tilde}
    \mathbb{E}\left[e^{-s\widetilde{\mathfrak{J}}_{\alpha}}\right] = 1 - \frac{s^{\alpha}}{\Gamma(1+\alpha)} \left(e^{-sd_{\alpha}} + sd_{\alpha}\int_0^{1} y^{-\alpha}e^{-d_{\alpha}sy}\,\dd y\right)^{-1}, \quad s\geq 0.
\end{align}

The random variable \( \widetilde{\mathfrak{J}}_{\alpha} \) plays a central role in 
characterizing the limiting point processes obtained for shrinking targets around 
points that include \( 0 \) in their orbit. Specifically, any such limit can be 
represented as a thinning and rescaling of the renewal point process
\( \RPP(\widetilde{\mathfrak{J}}_{\alpha}) \). To formalize the concept of thinning and 
rescaling a point process, we rely on the following definition.

\begin{defn}[Thinning and rescaling]
    For a point process $N = \sum_{i= 1}^{+\infty} \delta_{T_i}$, $\tau > 0$ and $v > 0$, we call $\tau$-thinning and $v$-rescaling of $N$ the point process
    \begin{align*}
        N^{(\tau,v)} \eqlaw \sum_{i = 1}^{+\infty} X_i\, \delta_{(vT_i)}\,,
    \end{align*}
    where $(X_i)_{i\geq 1}$ are i.i.d. random variables, independent from $N$ and $X_1 \sim \text{Bernoulli}(\tau)$.
\end{defn}

  We are now able to state our theorem for shrinking neighborhoods of preimages of $0$. 
\begin{thm}
    \label{thm:REPP_LSV_preimages_of_0}
    Let $p > 1$. Let $k \geq 0$ and assume $x \in T^{-(k+1)}\{0\}$. Let $B_n := C_x^n$ be the connected component of $T^{-(k+1)}[0,c_n]$ containing $x$. Then, 
    \begin{align*}
                N_{B_n}^{\gamma} \xRightarrow[n \to +\infty]{\mu_{B_n}} \RPP(\widetilde{\mathfrak{J}}_{\alpha})^{(\mathbb{Q}_k(x), \mathbb{Q}_k(x)^{1/\alpha})} \quad \text{and} \quad N_{B_n}^{\gamma} \xRightarrow[n \to +\infty]{\mathcal{L}(\mu)} \DRPP(\mathfrak   {J}_{\alpha}, \widetilde{\mathfrak{J}}_{\alpha})^{(\mathbb{Q}_k(x), \mathbb{Q}_k(x)^{1/\alpha})} \,,
    \end{align*}
        where 
        \[\mathbb{Q}_k(x) := \frac{\rho(x)}{\rho(1/2) (T^{k})'(x)}.\]
\end{thm}

\begin{rem}
    For every $k\geq 0$, $\mathbb{Q}_k$ is a probability on $T^{-(k+1)}\{0\}$ because $\rho$ is a fixed point of the transfer operator $\widehat{T}$, whence
    \begin{align*}
        \rho(1/2) = \big(\widehat{T}^k\rho\big)(1/2) = \sum_{T^ky = \frac12}\frac{\rho(y)}{(T^k)'(y)}.
    \end{align*}
\end{rem}

Theorem \ref{thm:REPP_LSV_preimages_of_0} reveals that points whose orbits include
\(0\) exhibit markedly different behavior compared to the cases outlined in Theorems 
\ref{thm:REPP_LSV_nonperiodic_points} and \ref{thm:REPP_LSV_periodic_points}, where 
the transition from the finite measure setting is more straightforward. Specifically, 
in the finite measure case, preimages of \(0\) conform to the standard Poisson 
process. However, in the infinite measure scenario, the interplay between extended 
excursions and visits to neighborhoods of \(0\)'s preimages disrupts similar outcomes. 
That said, points located farther from \(0\) are expected to exhibit a diminishing 
dependence on \(0\). This phenomenon is explored further in the next proposition.

\begin{prop}
    \label{prop:convergence_to_FPP_when_going_further_from_0}
    For every sequence of $(x_k)_{k\geq 1}$, with $x_k \in T^{-(k+1)}\{0\}$, we have 
    \begin{align*}
        \RPP(\widetilde{\mathfrak{J}}_{\alpha})^{(\mathbb{Q}_k(x_k), \mathbb{Q}_k(x_k)^{1/\alpha})} \xRightarrow[k\to +\infty]{} \fPp_{\alpha}(\Gamma(1+\alpha)),
    \end{align*}
    and
    \begin{align*}
        \DRPP(\mathfrak   {J}_{\alpha}, \widetilde{\mathfrak{J}}_{\alpha})^{(\mathbb{Q}_k(x_k), \mathbb{Q}_k(x_k)^{1/\alpha})} \xRightarrow[k\to +\infty]{} \fPp_{\alpha}(\Gamma(1+\alpha))\,.
    \end{align*}
\end{prop}

For periodic points, the longer the period, the closer the extremal index \(\theta\) approaches \(1\), resulting in the limiting point process converging more closely to the fractional Poisson process. Similarly, Proposition \ref{prop:convergence_to_FPP_when_going_further_from_0} guarantees a parallel trend for preimages of \(0\): the farther a point lies from the neutral fixed point \(0\) (in the sense of time iterations needed to hit \(0\)), the more closely its limiting behavior aligns with that of the fractional Poisson process.

\paragraph{Barely infinite case $p = 1$.}

  When, $p=1$, the picture is different. Indeed, for neighborhoods of $0$ and a well-chosen scaling, the limit law obtained is again the exponential, as in the finite measure case.

\begin{thm}
\label{thm:Collet_Galves_p_equals_1}
\textup{\cite[Theorem 5]{CG93_Statisticsofclosevisitstotheindifferentfixedpointofanintervalmap}, \cite[Theorem 3.3]{CampaninoIsola95}} Let $p = 1$ and $B_n = [0,c_n]$ for $n\geq 0$. Then, 
\begin{align*}
    \mathbb{E}_{\mu_Y}[\lr_{B_n}]^{-1}\,\lr_{B_n} \xRightarrow[n\to +\infty]{\mathcal{L}(\mu)} \mathcal{E}
\end{align*}
where $\mathcal{E}$ is the exponential law.
\end{thm}

  As for Corollary \ref{cor:HTS_neighborhood_0_normalization_gamma} for $p >1$, we need to make sure that the result can be written in our setting to make the scaling gamma appear. This is ensured by the following lemma. 

\begin{lem}
    \label{lem:comparison_renormalizations_for_0_p=1}
    With the same notation as in Theorem \ref{thm:Collet_Galves_p_equals_1}, we have 
    \begin{align*}
        \gamma(\mu(T_2^{-1}B_n)) \uset{\widesim}{n\to+\infty} \mathbb{E}_{\mu_Y}[r_{B_n}]^{-1}.
    \end{align*}
\end{lem}

  Using this and applying the same approach as for $p > 1$, we get the following theorem

\begin{thm}
    \label{thm:REPP_preimages_of_0_barely_infinite_case}
    Let $p=1$, $k \geq 0$ and $x \in T^{-{k+1}}\{0\}$. Let $B_n := C_x^n$. Then, 
    \begin{align*}
        N_{B_n}^{\gamma} \xRightarrow[n\to +\infty]{\mu_{B_n}} \textup{PPP}(1) \quad \text{and} \quad N_{B_n}^{\gamma} \xRightarrow[n\to +\infty]{\mathcal{L}(\mu)} \textup{PPP}(1).
    \end{align*}
\end{thm}
Therefore, in the special case $p=1$, the preimages of $0$ have the same limit behavior as any 
other non periodic point and thus we only have a dichotomy.

\section{Return and Hitting Rare Event Point Processes and the Fractional Poisson Process}
\label{section:proofs_abstract_results}

\subsection{Equivalence between Hitting and Return Point Processes}

  This part is devoted to the proofs of Theorems \ref{thm:HTS-REPP_vs_RTS-REPP_infinite_measure_renormalized_measure} and \ref{thm:HTS-REPP_vs_RTS-REPP_infinite_measure_renormalized_returns} and their applications. 

\subsubsection{Proof of Theorems \ref{thm:HTS-REPP_vs_RTS-REPP_infinite_measure_renormalized_measure}-\ref{thm:HTS-REPP_vs_RTS-REPP_infinite_measure_renormalized_returns}}

  As stated in Remark \ref{rem:saying_that_the_two_sclaings_are_equivalents_for_hREPP_rREPP}, Theorems \ref{thm:HTS-REPP_vs_RTS-REPP_infinite_measure_renormalized_measure} and \ref{thm:HTS-REPP_vs_RTS-REPP_infinite_measure_renormalized_returns} are equivalent since $\alpha \in (0,1]$. We are going in fact to prove Theorem \ref{thm:HTS-REPP_vs_RTS-REPP_infinite_measure_renormalized_returns} and study the sequence of processes $\big(\mu(E_n)\,a(\Phi_{E_n})\big)_{n\geq 0}$. For that, we introduce some more notation. For $A \in \mathscr{B}$,  $u\in L^1(\mu), \; u\geq 0,\; d\geq 1$ and $t_1,\dots,t_d > 0$, let 
\begin{align*}
    G^{[d]}_{u,A}(t_1, \dots, t_d) := \mu_u\left(\mu(A)\,a\big(\Phi^{[d]}_A\big) \leq (t_1, \dots,t_d)\right),
\end{align*}
where $\mu_u$ is the measure  absolutely continuous with respect to $\mu$ with density $u$. In particular, if $\int u\,\dd\mu = 1$, $G^{[d]}_{u,A}$ is the distribution function of the random variable $\mu(A)\,a(\Phi_A^{[d]})$ under the probability $\mu_u$. We also define
\begin{align*}
  \widetilde{G}^{[d]}_{A}(t_1, \dots, t_d) = \mu_{A}\left(\mu(A)\,a\big(\Phi^{[d]}_A\big) \leq (t_1, \dots,t_d)\right) \quad \forall d\geq 1, \; t_1,\dots,t_d \geq 0, 
\end{align*}
as the distribution function of $\mu(A)\,a(\Phi_A^{[d]})$ under $\mu_A$. Before delving into the proof, we first establish the following lemma, an adaptation of \cite[Lemma 5.1]{RZ20}, which provides a decomposition for multiple return events.

\begin{lem}
\label{lem:decomposition_returns}
For all $A\in \mathscr{B}$, $n_1,\dots, n_d \in \mathbb{N}$, we have the following decomposition:
    \begin{equation*}
	\left\{\Phi_A^{[d]} \leq (n_1, \dots, n_d)\right\} = \bigsqcup_{l = 1}^{n_1} T^{-\ell}\left(A \cap \left\{\Phi_A^{[d-1]} \leq \left(n_2 -\ell,\dots, n_d - \ell\right)\right\} \cap \left\{\Phi_{A}^{[d]}\leq \left(n_1-\ell, ,\dots, n_d-\ell\right)\right\}^c \right) \,.
	\end{equation*}
\end{lem}

\begin{rem}
    In fact, we also have the easier decomposition
    \[
    \left\{\Phi_A^{[d]} \leq (n_1, \dots, n_d)\right\} = \bigcup_{l = 1}^{n_1} T^{-\ell}\left(A \cap \left\{\Phi_A^{[d-1]} \leq \left(n_2 -\ell,\dots, n_d - \ell\right)\right\} \right),
    \]
    but we will later take advantage of the disjointness.
\end{rem}

\begin{proof}[Proof (of Lemma \ref{lem:decomposition_returns})]
      For every $\ell \geq 1$, $T^{-\ell}\left( A \cap  \left\{\Phi_A^{[d-1]} \leq \left(n_2 -\ell,\dots, n_d - \ell\right)\right\}\right) \subset \left\{\Phi_A^{[d]} \leq (n_1, \dots, n_d)\right\}$, hence we have the inclusion of the right member into the left one.\\
    Now, let $x \in \left\{\Phi_A^{[d]} \leq (n_1, \dots, n_d)\right\}$. Consider $\ell := \max \big\{p\leq n_1\;|\; x\in T^{-p} (A\cap \{\Phi_A^{[d-1]} \leq \left(n_2 -p,\dots, n_d - p\right)\})\big\}$. The maximum \(\ell\) is well-defined, as the set is non-empty with \(r_A(x)\) included as an element. By construction, $x\in T^{-\ell}\big(A\cap \big\{\Phi_A^{[d-1]} \leq \left(n_2 -\ell,\dots, n_d - \ell\right)\big\}\big)$. If $x\in T^{-\ell}\big\{\Phi_{A}^{[d]}\leq \left(n_1-\ell, ,\dots, n_d-\ell\right)\big\}$, then $\ell + r_A(T^{\ell}x)$ is also in the set which is a contradiction by definition of $\ell$. 
    
      The same approach shows that the union is disjoint. For  $1\leq \ell_1 < \ell_2 \leq n_1$, we have 
    \begin{align*}
	& T^{-\ell_1}\left(A \cap \left\{\Phi_A^{[d-1]} \leq \left(n_2 -\ell_1,\dots, n_d - \ell_1\right)\right\} \cap \left\{\Phi_{A}^{[d]}\leq \left(n_1-\ell_1, ,\dots, n_d-\ell_1\right)\right\}^c \right) \\
    & \qquad \qquad \cap T^{-\ell_2}\left(A \cap \left\{\Phi_A^{[d-1]} \leq \left(n_2 -\ell_2,\dots, n_d - \ell_2\right)\right\} \cap \left\{\Phi_{A}^{[d]}\leq \left(n_1-\ell_2, ,\dots, n_d-\ell_2\right)\right\}^c \right) \\
    & \qquad  \subset T^{-\ell_1} \Big( A \cap \big\{\Phi_{A}^{[d]}\leq (n_1-\ell_1, ,\dots, n_d-\ell_1)\big\}^c \\
    & \qquad \qquad \cap T^{-(\ell_2 - \ell_1)}\big( A \cap \big\{\Phi_A^{[d-1]} \leq (n_2 -\ell_2,\dots, n_d - \ell_2)\big\}\big)\Big) \\
    & \qquad \subset T^{\ell_1}\Big( A \cap \big\{\Phi_{A}^{[d]}\leq (n_1-\ell_1, ,\dots, n_d-\ell_1)\big\}^c \cap \big\{ \Phi_A^{[d]} \leq (\ell_2 - \ell_1,n_2 - \ell_1, \dots, n_d - \ell_1)\big\}\Big) \\
    & \qquad \subset T^{-\ell_1} \Big( A \cap \big\{\Phi_{A}^{[d]}\leq (n_1-\ell_1, ,\dots, n_d-\ell_1)\big\}^c \cap \big\{ \Phi_A^{[d]} \leq (n_1 - \ell_1,n_2 - \ell_1, \dots, n_d - \ell_1)\big\}\Big) = \emptyset.
    \end{align*}
\end{proof}

  We are now ready to go on with the proof of Theorem \ref{thm:HTS-REPP_vs_RTS-REPP_infinite_measure_renormalized_returns}. 

\begin{proof}[Proof (of Theorem \ref{thm:HTS-REPP_vs_RTS-REPP_infinite_measure_renormalized_returns})] Set $R_n := \mu(E_n)\,a(\Phi_{E_n})$ and for all $t\geq 0$, let $p_n^{[t]}$ be the integer such that 
\begin{align*}
    \big\{\mu(E_n)a(r_{E_n}) \leq t\big\} = \big\{ r_{E_n} \leq p_n^{[t]}\big\} \,.
\end{align*}
  Conversely, define for all $0\leq \ell \leq p_n^{[t]}$, the number $\vartheta_{n,\ell}^{[t]}$ such that 
\begin{align*}
    \big\{\mu(E_n)a(r_{E_n}) \leq \vartheta_{n,\ell}^{[t]}\big\} = \big\{r_{E_n} \leq p_n^{[t]} - \ell\big\}.
\end{align*}
Thus, for all $d\geq 1$ and $t_1,\dots, t_d \geq 0$ we get
\begin{align}
    \label{eq:reecriture_R_n_Phi_n_p_nt}
    \big\{R_n^{[d]} \leq (t_1,\dots,t_d)\big\} = 
    \big\{ \Phi_{E_n}^{[d]} \leq \big(p_n^{[t_1]},\dots p_n^{[t_d]}\big)\big\}.
\end{align}
and 
\begin{align}
    \label{eq:reecriture_Phi_n_R_n_theta_n_l}
    \big\{ R_n^{[d]} \leq (\vartheta_{n,\ell}^{[t_1]},\dots, \vartheta_{n,\ell}^{[t_d]})\big\} = \big\{\Phi_{E_n}^{[d]} \leq \big(p_n^{[t_1]} - \ell, \dots, p_n^{[t_d]} - \ell\big)\big\}.
\end{align}
The quantities $p_n^{[t]}$ and $\theta_{n,\ell}^{[t]}$ can be computed using $b$ the asymptotic inverse of $a$ and $a$ itself. We get 
\begin{align*}
    p_n^{[t]} := b(t/\mu(E_n)) \quad \text{and} \quad \vartheta_{n,\ell}^{[t]} = \mu(E_n)\,a\big(p_n^{[t]} - \ell\big).
\end{align*}

  First, assume the convergence of the return process, that is to say assume that $R_n \xRightarrow[n\to +\infty]{\mu_{E_n}} \widetilde{\Phi}$. This is equivalent to the convergence of $R_n^{[d]}$ for every $d\geq 1$ and thus 
assume that $\widetilde{G}_{E_n}^{[d]}(t_1, \dots, t_d) \xRightarrow[]{} \widetilde{G}^{[d]}(t_1, \dots, t_d)$ at every point of continuity $(t_1,\dots, t_d)$ of $\widetilde{G}^{[d]}$ such that $t_1 ( (t_2/t_1)^{1/\alpha} - m/M)^{\alpha}$ for all $0\leq m\leq M$ are also continuity points of $\widetilde{G}^{[d]}$. Without loss of generality, we can assume that $t_1\leq \dots \leq t_d$ as the sequence of returns is increasing. In the following, we write $\widetilde{G}^{[d]}_n$ instead of $\widetilde{G}^{[d]}_{E_n}$ to ease the notations and we will do the same for $G_u^{[d]}$. \\

  Now, we have

\begin{align*}
	& G^{[d]}_{u,n}(t_1,\dots, t_d) = \int_{\left\{R_n^{[d]} \leq (t_1,\dots, t_d)\right\}} u\,\dd\mu = \int_{\{\Phi^{[d]}_{E_n} \leq (p_n^{[t_1]}, \dots, p_n^{[t_d]})\}}u\,\dd\mu \quad \text{(by \eqref{eq:reecriture_R_n_Phi_n_p_nt})} \\
	& = \sum_{\ell = 1}^{p_n^{[t_1]}} \bigg(\int_{E_n \cap \left\{\Phi_{E_n}^{[d-1]} \leq \left(p_n^{[t_2]} -\ell,\dots, p_n^{[t_d]} - \ell\right)\right\} \cap \left\{\Phi_{E_n}^{[d]}\leq \left(p_n^{[t_1]}-\ell, ,\dots, p_n^{[t_d]}-\ell\right)\right\}^c} \widehat{T}^{\ell}u\,\dd\mu\bigg) \quad \text{(by Lemma \ref{lem:decomposition_returns})}\\
	& = \sum_{\ell = 1}^{p_n^{[t_1]}} \bigg(\int_{E_n \cap \left\{\Phi_{E_n}^{[d-1]} \leq \left(p_n^{[t_2]} -\ell,\dots, p_n^{[t_d]} - \ell\right)\right\}} \widehat{T}^{\ell}u\,\dd\mu \\
    & \qquad \qquad \qquad \qquad - \int_{E_n \cap \left\{\Phi_{E_n}^{[d]} \leq \left(p_n^{[t_1]} -\ell,\dots, p_n^{[t_d]} - \ell\right)\right\}}\widehat{T}^{\ell}u\,\dd\mu\bigg) \quad \text{(because $p_n^{[t_i]} \leq p_n^{[t_{i+1}]}$)}\\
	& = \sum_{m=0}^{M-1}  \sum_{\ell = \left\lfloor \frac{m}{M}p_n^{[t_1]}\right\rfloor + 1}^{\left\lfloor\frac{m+1}{M}p_n^{[t_1]}\right\rfloor} \bigg(\int_{E_n \cap \left\{R_n^{[d-1]} \leq \left(\vartheta_{n,\ell}^{[t_2]},\dots, \vartheta_{n,\ell}^{[t_d]}\right)\right\}} \widehat{T}^{\ell}u\,\dd\mu \\
    & \qquad \qquad \qquad \qquad - \int_{E_n \cap \left\{R_n^{[d]} \leq \left(\vartheta_{n,\ell}^{[t_1]}, \dots, \vartheta_{n,\ell}^{[t_d]}\right)\right\}}\widehat{T}^{\ell}u\,\dd\mu \bigg) \quad \text{(by \eqref{eq:reecriture_Phi_n_R_n_theta_n_l})}\\
    & \leq \sum_{m=0}^{M-1}  \sum_{\ell = \left\lfloor \frac{m}{M}p_n^{[t_1]}\right\rfloor + 1}^{\left\lfloor\frac{m+1}{M}p_n^{[t_1]}\right\rfloor} 
	\biggl(\int_{E_n \cap \left\{R_n^{[d-1]} \leq \left(\vartheta_{n,\left\lfloor \frac{m}{M}p_n^{[t_1]}\right\rfloor}^{[t_2]},\dots, \vartheta_{n,\left\lfloor \frac{m}{M}p_n^{[t_1]}\right\rfloor}^{[t_d]}\right)\right\}} \widehat{T}^{\ell}u\,\dd\mu \\ 
    & \qquad \qquad \qquad \qquad  - \int_{E_n \cap \left\{R_n^{[d]} \leq \left(\vartheta_{n,\left\lfloor \frac{m+1}{M}p_n^{[t_1]}\right\rfloor}^{[t_1]}, \dots, \vartheta_{n,\left\lfloor \frac{m+1}{M}p_n^{[t_1]}\right\rfloor}^{[t_d]}\right)\right\}}\widehat{T}^{\ell}u\,\dd\mu \biggr) \\
    & \leq \sum_{m=0}^{M-1} \int_{E_n \cap \left\{R_n^{[d-1]} \leq \left(\vartheta_{n,\left\lfloor \frac{m}{M}p_n^{[t_1]}\right\rfloor}^{[t_2]},\dots, \vartheta_{n,\left\lfloor \frac{m}{M}p_n^{[t_1]}\right\rfloor}^{[t_d]}\right)\right\}} \sum_{\ell = \left\lfloor \frac{m}{M}p_n^{[t_1]}\right\rfloor + 1}^{\left\lfloor\frac{m+1}{M}p_n^{[t_1]}\right\rfloor}\widehat{T}^{\ell}u\,\dd\mu\\
    & \qquad \qquad \qquad \qquad  - \int_{E_n \cap \left\{R_n^{[d]} \leq \left(\vartheta_{n,\left\lfloor \frac{m+1}{M}p_n^{[t_1]}\right\rfloor}^{[t_1]}, \dots, \vartheta_{n,\left\lfloor \frac{m+1}{M}p_n^{[t_1]}\right\rfloor}^{[t_d]}\right)\right\}} \sum_{\ell = \left\lfloor \frac{m}{M}p_n^{[t_1]}\right\rfloor + 1}^{\left\lfloor\frac{m+1}{M}p_n^{[t_1]}\right\rfloor} \widehat{T}^{\ell}u\,\dd\mu\,.
    \end{align*}
    At this point, we can take advantage of the fact that $Y$ is uniform for $u$. In particular, it means that of all $0 \leq c_1 < c_2$, we have 
    \begin{align*}
        \sum_{\ell = c_1k}^{c_2k - 1} \widehat{T}^{\ell}u \sim (c_2^{\alpha} - c_1^{\alpha})\,a_k \quad \text{uniformly mod $\mu$ on $Y$,}
    \end{align*}
    and thus 
    \begin{align*}
        \sum_{\ell = \left\lfloor \frac{m}{M}p_n^{[t_1]}\right\rfloor + 1}^{\left\lfloor\frac{m+1}{M}p_n^{[t_1]}\right\rfloor} \widehat{T}^{\ell}u \sim \left(\left(\frac{m+1}{M} \right)^{\alpha} - \left(\frac{m+1}{M} \right)^{\alpha}\right) a_{p_n^{[t_1]}} \quad \text{uniformly mod $\mu$ on $Y$,}
    \end{align*}
    meaning that for all $\varepsilon > 0$ and $n$ large enough,
    \begin{align*}
        \sum_{\ell = \left\lfloor \frac{m}{M}p_n^{[t_1]}\right\rfloor + 1}^{\left\lfloor\frac{m+1}{M}p_n^{[t_1]}\right\rfloor} \widehat{T}^{\ell}u \leq (1 + \varepsilon)\left(\left(\frac{m+1}{M} \right)^{\alpha} - \left(\frac{m+1}{M} \right)^{\alpha}\right) a_{p_n^{[t_1]}}\,.
    \end{align*}
    By definition, $a_{p_n^{[t_1]}} = a(b(t_1/\mu(E_n))) \sim t_1/\mu(E_n)$ and by the mean value theorem $((m+1)/M)^{\alpha} - (m/M)^{\alpha} \leq \alpha m^{\alpha - 1}/M^{\alpha}$.  Hence, for $n$ large enough, we get 
    \begin{align}
    G^{[d]}_{u,n}(t_1,\dots, t_d)& \leq (1+\varepsilon)\,\alpha\, t_1 \sum_{m=0}^{M-1} \frac{1}{M} \left(\frac{m}{M}\right)^{\alpha - 1} \left[\mu_{E_n}\Big(R_n^{[d-1]} \leq \Big(\vartheta_{n,\left\lfloor \frac{m}{M}p_n^{[t_1]}\right\rfloor}^{[t_2]},\dots, \vartheta_{k,\left\lfloor \frac{m}{M}p_n^{[t_1]}\right\rfloor}^{[t_d]}\Big)\Big) \nonumber \right.\\
	& \left.\qquad \qquad - \mu_{E_n}\Big( R_n^{[d]} \leq \Big(\vartheta_{k,\left\lfloor \frac{m+1}{M}p_n^{[t_1]}\right\rfloor}^{[t_1]}, \dots, \vartheta_{n,\left\lfloor \frac{m+1}{M}p_n^{[t_1]}\right\rfloor}^{[t_d]}\Big)\Big)\right]\nonumber \\
	 & \leq (1+\varepsilon)\,\alpha\, t_1 \sum_{m=0}^{M-1} \frac{1}{M} \left(\frac{m}{M}\right)^{\alpha - 1}
	\left[ \widetilde{G}_k^{[d-1]} \Big(\vartheta_{k,\left\lfloor \frac{m}{M}p_n^{[t_1]}\right\rfloor}^{[t_2]},\dots, \vartheta_{n,\left\lfloor \frac{m}{M}p_n^{[t_1]}\right\rfloor}^{[t_d]}\Big) \nonumber \right.\\ 
	& \left. \qquad \qquad - \widetilde{G}_n^{[d]}\Big(\vartheta_{n,\left\lfloor \frac{m+1}{M}p_n^{[t_1]}\right\rfloor}^{[t_1]},\dots, \vartheta_{n,\left\lfloor \frac{m+1}{M}p_n^{[t_1]}\right\rfloor}^{[t_d]}\Big) \right]\,.
    \label{eq:equation_Riemann_upper_bound}
\end{align}
    The lower bound is obtained similarly. Now, $\forall t_2 \geq t_1 > 0, \forall 0<c<1$, we have
\begin{align*}
     \quad \vartheta^{[t_2]}_{n, cp_n^{[t_1]}}  &= \mu(E_n)\, a\big( p_n^{[t_2]} - cp_n^{[t_1]}\big) = \mu(E_n)\, a\bigg( b\Big(\frac{t_2}{\mu(E_n)}\Big) - c b\Big(\frac{t_1} {\mu(E_n)}\Big) \bigg) \\
    & = \mu(E_n)\,a\bigg( b\Big(\frac{t_1}{\mu(E_n)}\Big) \Big( b(t_2/\mu(E_n))/b(t_1/\mu(E_n)) - c\Big) \bigg)\\
    & \sim \mu(E_n)\, a\left( b(t_1/\mu(E_n))\left(\left(\frac{t_2}{t_1}\right)^{1/\alpha} -c\right)\right) \\
    & \sim t_1 \left(\left(\frac{t_2}{t_1}\right) - c\right)^{\alpha},
\end{align*}
where we use the hypothesis $a\in \RV(\alpha)$ which implies $b\in \RV(\alpha^{-1})$. Since we assumed the convergence for the return time process, it yields
\begin{align*}
    &\widetilde{G}_n^{[d-1]} \Big(\vartheta_{n,\left\lfloor \frac{m}{M}p_n^{[t_1]}\right\rfloor}^{[t_2]},\dots, \vartheta_{n,\left\lfloor \frac{m}{M}p_n^{[t_1]}\right\rfloor}^{[t_d]}\Big) \\
    &\qquad \xrightarrow[n\to +\infty]{} \widetilde{G}^{[d-1]}\bigg( t_1 \Big(\Big(\frac{t_2}{t_1}\Big)^{1/\alpha}- \frac{m}{M}\Big)^{\alpha},\dots, t_1 \Big(\Big(\frac{t_d}{t_1}\Big)^{1/\alpha}- \frac{m}{M}\Big)^{\alpha}\bigg)
\end{align*}
  and 
\begin{align*}
    &\widetilde{G}_n^{[d]}\Big(\vartheta_{n,\left\lfloor \frac{m+1}{M}p_n^{[t_1]}\right\rfloor}^{[t_1]}, \dots, \vartheta_{n,\left\lfloor \frac{m+1}{M}p_n^{[t_1]}\right\rfloor}^{[t_d]}\Big) \\
    &\qquad \xrightarrow[n\to + \infty]{} \widetilde{G}^{[d]}\biggl( t_1 \Big(1 - \frac{m+1}{M}\Big)^{\alpha},\dots,t_1 \Big(\Big(\frac{t_d}{t_1}\Big)^{1/\alpha}- \frac{m+1}{M}\Big)^{\alpha}\biggr)\,.
\end{align*}

  Thus, together with \eqref{eq:equation_Riemann_upper_bound}, recognizing a Riemann integral and letting $\varepsilon$ go to $0$, we get 
\begin{align*}
	& G_{u,n}^{[d]}(t_1, \dots, t_d) \xrightarrow[k\to +\infty]{} \alpha\, t_1\int_0^{1} \left[ \widetilde{G}^{[d-1]}\left(\left( t_2^{1/\alpha}- t_1^{1/\alpha}r\right)^{\alpha},\dots,\left( t_d^{1/\alpha}- t_1^{1/\alpha}r\right)^{\alpha}\right)  \right. \\
	& \left. \qquad \qquad \qquad \qquad \qquad \qquad \qquad \;
 - \,\widetilde{G}^{[d]}\bigg( t_1 \left(1 - r\right)^{\alpha}, \left( t_2^{1/\alpha}- t_1^{1/\alpha}r\right)^{\alpha},
	\dots,\left( t_d^{1/\alpha}- t_1^{1/\alpha}r\right)^{\alpha}\bigg) \right] r^{\alpha - 1}\,\dd r.
\end{align*}
  This means that we have convergence of the normalized hitting time point process with distribution functions
\begin{align}
    G^{[d]}(t_1,\dots, t_d) & := \alpha\, t_1  \int_0^{1} \left[ \widetilde{G}^{[d-1]}\Big(\left( t_2^{1/\alpha}- t_1^{1/\alpha}r\right)^{\alpha},\dots,\left( t_d^{1/\alpha}- t_1^{1/\alpha}r\Big)^{\alpha}\right)  \nonumber \right. \\
    & \left. \qquad\qquad \quad \; - \,\widetilde{G}^{[d]}\Bigl( t_1 \left(1 - r\right)^{\alpha}, \left( t_2^{1/\alpha}- t_1^{1/\alpha}r\right)^{\alpha},
	\dots,\left( t_d^{1/\alpha}- t_1^{1/\alpha}r\right)^{\alpha}\Bigr) \right] r^{\alpha - 1}\,\dd r \,.
    \label{eq:equation_relation_G_d_tilde_G_d}
\end{align}

  Now, assume conversely the convergence of the hitting process. By Helly selection theorem in $\mathbb{R}^d$ (see \cite[Theorem 5.19]{Kal02_SecondEdition}) and diagonal extraction, we can find a subsequence $(n_k)$ such that $\widetilde{G}_{n_k}^{[d]}$ converges pointwise towards $\widetilde{G}^{[d]}_*$ for every $d\geq 0$. But, in this case, alongside this subsequence we have $G^{[d]}_{u,n_k}$ that converges pointwise towards $G^{[d]}_*$ defined by \eqref{eq:equation_relation_G_d_tilde_G_d} from $\widetilde{G}^{[d]}_*$. Since the equation defines uniquely $\widetilde{G}^{[d]}_*$ from $G^{[d]}_* = G^{[d]}$ by Lemma \ref{lem:functional_equation_uniquely_determined_by_one_another_normalized_returns}, it gives the wanted convergence.   
\end{proof}

  It remains to prove the following Lemma ensuring that the relationship found determines uniquely the law of one limit process from the laws of the other.

\begin{lem}
    \label{lem:functional_equation_uniquely_determined_by_one_another_normalized_returns}
    Let $\Phi$ and $\widetilde{\Phi}$ be two processes on the phase space $\overline{\mathbb{R}}_+$. For all $d\geq 1$, write $G^{[d]}$ and $\widetilde{G}^{[d]}$ the distribution functions of their marginals. By convention let $G^{[0]} = \widetilde{G}^{[0]} = 1$. Then, the equation $\forall d\geq 1$, $\forall 0\leq t_1 \leq \dots \leq t_d$,
    \begin{align*}
    G^{[d]}(t_1,\dots, t_d) & := \alpha\, t_1  \int_0^{1} \left[ \widetilde{G}^{[d-1]}\Big(\big( t_2^{1/\alpha}- t_1^{1/\alpha}s\big)^{\alpha},\dots,\big( t_d^{1/\alpha}- t_1^{1/\alpha}s\big)^{\alpha}\Big) \right. \\
	& \left. \qquad\qquad\quad\; 
    - \,\widetilde{G}^{[d]}\Bigl( t_1 \left(1 - s\right)^{\alpha}, \left( t_2^{1/\alpha}- t_1^{1/\alpha}s\right)^{\alpha},
	\dots,\left( t_d^{1/\alpha}- t_1^{1/\alpha}s\right)^{\alpha}\Bigr) \right] \,s^{\alpha - 1}\,\dd s, 
\end{align*}
    defines uniquely the law of $\Phi$ from the law of $\widetilde{\Phi}$ and conversely.
\end{lem}

\begin{rem}
    When $\alpha \in (0,1]$, instead of looking at the distribution function $G^{[d]}$, we can look at $F^{[d]}$ defined by
    \begin{align*}
        F^{[d]}(t_1,\dots,t_d) := G^{[d]}(t_1^{\alpha}, \dots, t_d^{\alpha})\;, \forall t_1,\dots,t_d > 0.
    \end{align*}
    In fact, $F$ is exactly the family of distribution functions of the limit point processes $\Psi$ and $\widetilde{\Psi}$ of Theorem \ref{thm:HTS-REPP_vs_RTS-REPP_infinite_measure_renormalized_measure} by Remark \ref{rem:saying_that_the_two_sclaings_are_equivalents_for_hREPP_rREPP}. Thus, Lemma \ref{lem:functional_equation_uniquely_determined_by_one_another_normalized_returns} is equivalent to Lemma \ref{lem:functional_equation_uniquely_determined_by_one_another} (below), and the equation in Lemma \ref{lem:functional_equation_uniquely_determined_by_one_another} gives \eqref{eq:relation_hitting_return_infinite_measure}, concluding the proof of Theorems \ref{thm:HTS-REPP_vs_RTS-REPP_infinite_measure_renormalized_measure}-\ref{thm:HTS-REPP_vs_RTS-REPP_infinite_measure_renormalized_returns}.
\end{rem}

\begin{lem}
    \label{lem:functional_equation_uniquely_determined_by_one_another}
    Let $\Phi$ and $\widetilde{\Phi}$ be two process on the phase space $\overline{\mathbb{R}}_+$. For all $d\geq 1$, write $F^{[d]}$ and $\widetilde{F}^{[d]}$ the distribution of their marginals. By convention, let $F^{[0]} = \widetilde{F}^{[0]} = 1$. Then, the equations 
    \begin{align*}
        F^{[d]}(t_1,\dots, t_d) &= \alpha \int_0^{t_1} \Bigl( \widetilde{F}^{[d-1]}\left(t_2 -t_1 + x, \dots, t_d - t_1 + x\right) \\
        & \qquad\qquad\; - \widetilde{F}^{[d]}\left(x, t_2 - t_1 + x,\dots, t_d - t_1 +x\right)\Bigr)  (t_1 - x)^{\alpha - 1}\,\dd x, 
    \end{align*}
    where $d\geq 1$ and one takes any $d$-uplet $0\leq t_1\leq \dots \leq t_d$,
    uniquely determines the law of $\Phi$ from the law of $\widetilde{\Phi}$, and vice versa.
\end{lem}

\begin{proof}[Proof (of Lemma \ref{lem:functional_equation_uniquely_determined_by_one_another})]
    It is immediate that the law of $\Phi$ is uniquely determined by the law of $\widetilde{\Phi}$. Now, assume that there are two stochastic processes $\widetilde{\Phi}_1$ and $\widetilde{\Phi}_2$ leading to the same $\Phi$. We will argue by induction. The case $d=1$ corresponds to the equivalence between hitting and return for the first return and has already been dealt in \cite[Remark 4.2.b)]{RZ20}. For $d\geq 1$, assuming the distribution functions are equal up to $d-1$, it means that for all $0\leq t_1\leq \dots \leq t_d$,
    \begin{align*}
        &\alpha \int_0^{t_1} \widetilde{F}_1^{[d]} ( x, t_2 - t_1 + x, \dots, t_d - t_1 + x)(t_1 - x)^{\alpha - 1}\,\dd x  \\
    &\qquad = \alpha \int_0^{t_1} \widetilde{F}_2^{[d]} ( x, t_2 - t_1 + x, \dots, t_d - t_1 + x)(t_1 - x)^{\alpha - 1}\,\dd x\,.
    \end{align*}

  Taking $s_2, \dots, s_d$ such that $t_i - t_1 = s_i$ for $2\leq i \leq d$, we obtain, for all $t_1,s_d,\dots,s_d\geq 0$,
\begin{align*}
        &\alpha \int_0^{t_1} \widetilde{F}_1^{[d]} ( x, x + s_2, \dots, x+ s_d )(t_1 - x)^{\alpha - 1}\,\dd x  \\
    &\qquad = \alpha \int_0^{t_1} \widetilde{F}_2^{[d]} ( x, x + s_2, \dots, x + s_d)(t_1 - x)^{\alpha - 1}\,\dd x. 
\end{align*}
Thus, if we define
\begin{align*}
    f : x \mapsto \widetilde{F}_1^{[d]} ( x, x + s_2, \dots, x+ s_d ) - \widetilde{F}_2^{[d]} ( x, x + s_2, \dots, x + s_d) \,,
\end{align*}
we then have 
\begin{align*}
    \int_0^{t_1} f(x)(t_1 - x)^{\alpha - 1}\,\dd x = 0\,,\; \forall t_1 \geq 0.
\end{align*}

  Thus, by definition of the Riemann-Liouville integral \eqref{eq:definition_Riemann-Liouville_integral}, we have $I^{\alpha}f = 0$ which implies that $f = 0$, meaning that $\widetilde{F}_1^{[d]}= \widetilde{F}_2^{[d]}$. It completes the induction and the proof.
\end{proof}

  We turn now to the proof of Corollaries \ref{cor:equivalence_HTS/RTS_infinite_measure}-\ref{cor:generalized_Kolmogorov-Feller_HTSvsRTS}

\begin{proof}[Proof (of Corollary \ref{cor:equivalence_HTS/RTS_infinite_measure})] This is immediate by looking at the projection on the first coordinate, \textit{i.e.} $d = 1$, in \eqref{eq:relation_hitting_return_infinite_measure}. With $\widetilde{F}^{[0]} = 1$, we get
\begin{align*}
    F^{[1]}(t) = \alpha \int_0^1 (1 - \widetilde{F}^{[1]}(u) (t-u)^{\alpha - 1}\,\dd u, \quad t\geq 0,
\end{align*}
which is exactly \eqref{eq:relation_HTS_RTS_first_return}. Equation \eqref{eq:Laplace_transform_HTS_from_RTS} can be found by the same method as in \cite[Lemma 7]{PSZ11}.
\end{proof}

\begin{proof}[Proof (of Corollary \ref{cor:generalized_Kolmogorov-Feller_HTSvsRTS})] Recall the definition of $P_N(d,t) = \mathbb{P}(N[0,t] = d)$ for a point process $N$, $d\geq 0$ and $t\in \mathbb{R}$. Let $d\geq 1$ and $t\geq 0$. Then, \eqref{eq:relation_hitting_return_infinite_measure} with $t_1,\dots, t_d = t$ is
\begin{align*}
    F^{[d]}(t,\dots, t) = \alpha \int_0^t \left(\widetilde{F}^{[d-1]}(x,\dots, x) - \widetilde{F}^{[d]}(x, \dots, x) \right)(t-x)^{\alpha - 1}\,\dd x.
\end{align*}
Since $F^{[d]}(t,\dots,t) = \mathbb{P}(N[0,t] \geq d)$ and $\widetilde{F}^{[d]}(t,\dots,t) = \mathbb{P}(\widetilde{N}[0,t] \geq d)$ for all $d\geq 0$, it yields
\begin{align*}
    \mathbb{P}(N[0,t] \geq  d) &= \alpha \int_0^t \mathbb{P}(\widetilde{N}[0,t] = d-1)(t - x)^{\alpha - 1}\,\dd x\,.
\end{align*}
  Thus, for every $d\geq 1$, we get
\begin{align*}
    1 - \sum_{k = 0}^{d} P_N(k, \cdot) &= \Gamma(1+\alpha)\, I^{\alpha}\big(P_{\widetilde{N}}(d,\cdot)\big)\,,
\end{align*}
whence
\begin{align*}
    P_{N}(d, \cdot) = \Gamma(1+\alpha)\, I^{\alpha} \big( P_{\widetilde{N}}(d, \cdot) - P_{\widetilde{N}}(d-1,\cdot)\big) 
\end{align*}
  which is the integral formulation of \eqref{eq:generalized_Kolmogorov-Feller_HTSvsRTS}. Similarly, for $d = 0$, we obtain
\begin{align*}
    P_{N}(0,\cdot) = 1 - \Gamma(1+\alpha)I^{\alpha}\big(P_{\widetilde{N}}(0, \cdot)\big).
\end{align*}
\end{proof}

\subsubsection{Proof of Proposition \ref{prop:fix_point_equation_FPP}}

  As suggested in Remark \ref{rem:FPP_PhiFPP}, we actually prove the characterization of the law and its uniqueness for the stochastic process.  It entails the uniqueness of the point process via the composition by $\Xi$. The process $\Phi_{\fPp_{\alpha}(\Gamma(1+\alpha))} \eqlaw (\phi^{(i)})_{i\geq 1}$ with $\phi^{(i)} = \sum_{k = 1}^i X_k$ and $(X_k)_{k\geq 1}$ i.i.d with common law $H_{\alpha}(\Gamma(1+\alpha))$, is such that $\Xi(\Phi_{\fPp_{\alpha}(\Gamma(1+\alpha))}) \eqlaw \fPp_{\alpha}(\Gamma(1+\alpha))$. \\
We can do the same with $\Phi_{\RPP(W_{\alpha, \theta}(\theta\Gamma(1+\alpha)))} = (\psi^{(i)})_{i\geq 1}$ where $\psi^{(i)} = \sum_{k = 1}^i W_k$ and $(W_k)_{k\geq 1}$ are i.i.d having as a common law $W_{\alpha, \theta}(\theta\Gamma(1+\alpha))$ (see \eqref{eq:def_of_W_alpha_theta} for the definition of these random variables). Then, $\Xi (\Phi_{\RPP(W_{\alpha, \theta}(\theta\Gamma(1+\alpha))}) \eqlaw  \RPP(W_{\alpha, \theta}(\theta\Gamma(1+\alpha)))$. 

\begin{proof}[Proof (of Proposition \ref{prop:fix_point_equation_FPP})]
We will actually prove that for $\theta \in (0,1]$, $\Phi_{\RPP(W_{\alpha, \theta}(\theta\Gamma(1+\alpha)))}$ is the only stochastic process such that for all $d\geq 1$ and $t_1\leq \dots \leq t_d$,
    \begin{align}
        \label{eq:equation_CFPP_fix_point}
        \widetilde{F}^{[d]}(t_1,\dots, t_d) & = (1 - \theta) \widetilde{F}^{[d-1]}(t_1,\dots, t_d) \nonumber \\
        & \quad + \theta \alpha \int_0^{t_1} \Bigl( \widetilde{F}^{[d-1]}\left(t_2 -t_1 + x, \dots, t_d - t_1 + x\right) \\
        & \quad - \widetilde{F}^{[d]}\left(x, t_2 - t_1 + x,\dots, t_d - t_1 +x\right)\Bigr)  (t_1 - x)^{\alpha - 1}\,\dd x.\nonumber
    \end{align}
      In particular, when $\theta = 1$, we will get that $\fPp_{\alpha}(\Gamma(1+\alpha))$ is the only process such that its finite-dimensional marginals are fixed points of the transformation from return times to hitting times.\\
    
      As in Lemma \ref{lem:functional_equation_uniquely_determined_by_one_another}, we show by induction on $d$ that there is at most a one fixed distribution $F^{[d]}$. For $d = 1$, \eqref{eq:Laplace_transform_HTS_from_RTS} in Corollary \ref{cor:equivalence_HTS/RTS_infinite_measure} clearly ensures that there is a unique fixed law and using the Laplace transform, $W_{\alpha,\theta}(\theta\Gamma(1+\alpha))$ can be identified as the fixed law\footnote{This result corresponds to \cite[Lemma 7.1]{RZ20} and $W_{\alpha,\theta}(\theta\Gamma(1+\alpha))$ corresponds to the law $\widetilde{H}_{\alpha, \theta}$ in the article. However, in \cite{RZ20}, the constant $\Gamma(1+\alpha)$ has been forgotten in the identification.}. \\
    
      Now assume $d\geq 2$ and that the result is true for $d-1$. Let $F_1^{[d]}$ and $F_2^{[d]}$ be two distribution functions compatible with $F_1^{[d-1]}$ and $F_2^{[d]}$ satisfying \eqref{eq:equation_CFPP_fix_point}. Since $F_1^{[d-1]} = F_2^{[d-1]}$ by hypothesis,
    for all $t_1,s_2,\dots,s_d \geq 0$, we have 
    \begin{align*}
        &F_1^{[d]}(t_1,t_1+s_2,\dots, t_1 + s_d) - F_2^{[d]}(t_1, t_1 + s_2, \dots, t_1 + s_d) \nonumber \\
        & \quad = \alpha \int_0^{t_1} \Big( F_2^{[d]}(x, x + s_2,\dots, s+ s_d) - F_1^{[d]}(x, x + s_2,\dots, s+ s_d) \Big) (t_1 - x)^{\alpha - 1}\,\dd x.
    \end{align*}

  Letting 
\begin{align*}
    h : x \mapsto F_1^{[d]}(x, x+ s_2,\dots, x+ s_d) - F_2^{[d]}(x, x+ s_2, \dots, x + s_d)
\end{align*}
the equation can be rewritten as
\begin{align*}
    h = -\Gamma(1+\alpha)\, I^{\alpha}h\,.
\end{align*}
  Going into the Laplace domain (which is possible because $h$ is locally integrable as its absolute value is bounded by $1$), we get $h = 0$ and thus $F_1^{[d]} = F_2^{[d]}$. 

Thus, there can be at most one process whose distribution functions are fixed points of 
\eqref{eq:equation_CFPP_fix_point}. It remains to demonstrate the existence of at least one fixed 
point for this equation. Fortunately, the theory applies to certain examples of null-recurrent 
Markov chains for which results concerning first hitting and return times are well-established (see 
\cite{BZ01, PSZ13} for the case \(\theta = 1\), and \cite{RZ20} for the general case). 

By leveraging the strong Markov property, these results extend to the point process level in cases 
where the waiting times are independent. In such scenarios, the process \(\Phi_{\RPP(W_{\alpha, 
\theta}(\theta\Gamma(1+\alpha)))}\) emerges naturally and must satisfy 
\eqref{eq:equation_CFPP_fix_point}. This demonstrates that a fixed point indeed exists. 

Consequently, the fixed point is both unique and well-defined\footnote{In \cite{RZ20},
the first-return case confirms this fixed point, although it omits the constant
\(\Gamma(1+\alpha)\).}.
\end{proof}

\subsection{Sufficient conditions for convergence towards $\fPp_{\alpha}$ and $\cfPp_{\alpha}$}

  In this section, we prove Theorem \ref{thm:sufficient_conditions_convergence_compound_FPP}. We first need to recall the following lemma giving a uniform control of the convergence to $0$ for the average of the iterations by the transfer operator for functions in a compact subset of $L^1(\mu)$.

\begin{lem} \textup{\cite[Theorem 3.1]{Zwe07_InfiniteMeasurePreservingTransformationsWithCompactFirstRegeneration}}
    \label{lem:convergence_Birkhoff_transfer_operator_compact}
    Let $(X, \mathscr{B}, \mu, T)$ be a CEMPT and $\mathcal{U}$ a compact subset of $L_1(\mu)$ such that $\int u\,\dd\mu = 1$ for all $u\in \mathcal{U}$. Then, uniformly in $u, u^* \in \mathcal{U}$, we have 
    \begin{align*}
        \left\| \frac{1}{M}\sum_{j = 0}^{M-1} \widehat{T}^ju - \frac{1}{M}\sum_{j = 0}^{M-1} \widehat{T}^ju^*\right\|_{L^1(\mu)} \xrightarrow[M\to +\infty]{} 0.
    \end{align*}
\end{lem}

  Now, we can dive into the proof of Theorem \ref{thm:sufficient_conditions_convergence_compound_FPP}. 

\begin{proof}[Proof (of Theorem \ref{thm:sufficient_conditions_convergence_compound_FPP})] As discussed in Remark \ref{rem:FPP_PhiFPP} and in the proof of Proposition \ref{prop:fix_point_equation_FPP}, we will actually show the convergence of the process. Then, the application of $\Xi$ and the extended continuous mapping theorem give the convergence of the point process. \\
  We consider the sequence of stochastic processes $(\gamma(\mu(B_n))\Phi_{B_n})_{n\geq 0}$ on $\overline{\mathbb{R}}_+$ and the sequence of probability spaces $(B_n,\mu_{B_n})$, \textit{i.e} we look at the succesive return times. As $(\overline{\mathbb{R}}_+)^{\mathbb{N}}$is compact, this sequence is tight so up to a subsequence and that the convergence of the stochastic process is characterized by the convergence of its finite-dimensional marginals, we can assume that the familly of distribution functions $(\widetilde{F}_{B_n}^{[d]})_{d\geq 1}$ converges towards the family of functions $(\widetilde{F}^{[d]})_{d\geq 0}$ (meaning that we have pointwise convergence for each one of them at the continuity points of $\widetilde{F}^{[d]}$). We are going to show that the only possible limits are the distribution functions of the successive return times of the stochastic process $\Phi_{\RPP(W_{\alpha, \theta}(\theta\Gamma(1+\alpha)))}$, which is enough to get the convergence towards this process. By Theorem \ref{thm:HTS-REPP_vs_RTS-REPP_infinite_measure_renormalized_measure}, for any given 
density \( u \), let \((F_{B_n, v}^{[d]})_{d \geq 1}\) denote the family of distribution functions 
corresponding to the finite-dimensional marginals of the stochastic process
\((\gamma(\mu(B_n))\Phi_{B_n})\) on the probability space \((X, \mu_{v})\), where \(\mu_v\) is a 
probability measure absolutely continuous with respect to \(\mu\), having density \( v \). Then, 
the family of renormalized distribution functions \((\widetilde{F}_{B_n,v}^{[d]})_{d \geq 1}\) 
converges to the family of functions \((F^{[d]})_{d \geq 1}\). 
Both \((\widetilde{F}_{B_n,v}^{[d]})_{d \geq 1}\) and \((F^{[d]})_{d \geq 1}\) satisfy the 
relationship given in \eqref{eq:relation_hitting_return_infinite_measure}.

      For all $d\geq 1$ and $0\leq t_1\leq \dots \leq t_d$ such that $(t_1,\dots, t_d)$ is a continuity point of $\widetilde{F}^{[d]}$, we have 
    \begin{align}
        \label{eq:CFPP_decomposition_between_immediate_return_and_escaping_annulus}
        \widetilde{F}_{B_n}^{[d]}(t_1,\dots, t_d) =& \frac{\mu(U(B_n))}{\mu(B_n)}\mu_{U(B_n)}\big(\gamma(\mu_{B_n})\,\lr_{B_n} \leq t_1,\dots, \gamma(\mu(B_n))\,r_{B_n}^{(d)} \leq t_d\big) \nonumber\\
        & +\frac{\mu(Q(B_n))}{\mu(B_n)} \mu_{Q(B_n)}\big(\gamma(\mu_{B_n})\,\lr_{B_n} \leq t_1,\dots, \gamma(\mu(B_n))\,r_{B_n}^{(d)} \leq t_d\big)\,.
    \end{align}
    We are going to study the two members separately, in Lemma \ref{lem:first_term_decomposition_in_proof_sufficient_condition_CFPP} and Lemma \ref{lem:second_term_decomposition_in_proof_sufficient_condition_CFPP}. When $U(B_n) =\emptyset$ the first term is trivially equal to 0 by convention and we can jump to Lemma \ref{lem:second_term_decomposition_in_proof_sufficient_condition_CFPP} directly.
    
    \begin{lem}
        \label{lem:first_term_decomposition_in_proof_sufficient_condition_CFPP}
        Assume $U(B_n)\neq \emptyset$. For all $d\geq 1$ and $0\leq t_1\leq \dots \leq t_d$ such that $(t_2,\dots, t_d)$ is a continuity point of $\widetilde{F}^{[d-1]}$, we have
        \begin{align*}
            \frac{\mu(U(B_n))}{\mu(B_n)}\mu_{U(B_n)}\big(\gamma(\mu(B_n))\,\lr_{B_n} \leq t_1, \dots, \gamma(\mu(B_n))\,\lr_{B_n}^{(d)}\leq t_d\big) \xrightarrow[n\to +\infty]{} (1 - \theta)\widetilde{F}^{[d-1]}(t_2, \dots, t_d).
        \end{align*}
    \end{lem}

    \begin{proof} [Proof (of Lemma \ref{lem:first_term_decomposition_in_proof_sufficient_condition_CFPP})]
      By \ref{cond_CFPP:extremal_index}, we have $\mu(U(B_n))/\mu(B_n) \xrightarrow[n\to +\infty]{} 1 - \theta$ . It vanishes when $\theta = 1$ and else we have
    \begin{align*}
        & \mu_{U(B_n)}\big(\gamma(\mu(B_n))\,\lr_{B_n} \leq t_1,\dots, \gamma(\mu(B_n))\,\lr_{B_n}^{(d)} \leq t_d \big) \\
        & \qquad \qquad \leq \mu_{U(B_n)}\left(\bigcap_{i = 1}^{d-1} \left\{ \gamma(\mu(B_n))\,\lr_{B_n}^{(i)} \circ T_{B_n} \leq t_{i+1}\right\} \right) \\
        & \qquad \qquad \leq \mu(B_n)\int_{B_n} \frac{\mathbf{1}_{U(B_n)}}{\mu(U(B_n))} \left(\prod_{i = 1}^{d-1} \mathbf{1}_{\{\gamma(\mu(B_n))\,\lr_{B_n}^{(i)} \leq t_{i+1}\}}\right) \circ T_{B_n}\,\dd\mu_{B_n} \\
        & \qquad \qquad \leq \mu(B_n)\int_{B_n} \widehat{T_{B_n}}\left(\frac{\mathbf{1}_{U(B_n)}}{\mu(U(B_n))}\right)  \left(\prod_{i = 1}^{d-1} \mathbf{1}_{\{\gamma(\mu(B_n))\,r_{B_n}^{(i)} \leq t_{i+1}\}}\right) \,\dd\mu_{B_n}  \\
        & \qquad \qquad \xrightarrow[n\to +\infty]{} \int_{B_n} \prod_{i = 1}^{d-1} \mathbf{1}_{\{\gamma(\mu(B_n))\,\lr_{B_n}^{(i)} \leq t_{i+1}\}} \,\dd\mu_{B_n} = \widetilde{F}^{[d-1]}(t_2, \dots,  t_d)\hspace{0.4cm} \text{(by \ref{cond_CFPP:Compatibility_Geometric_law})}.
    \end{align*}
    On the other hand, we have 
    \begin{align*}
        & \mu_{U(B_n)}\big(\gamma(\mu(B_n))\,\lr_{B_n} \leq t_1,\dots, \gamma(\mu(B_n))\,\lr_{B_n}^{(d)} \leq t_d \big) \\
        & \quad = \mu_{U(B_n))}\left( \big\{\gamma(\mu(B_n))\,\lr_{B_n} \leq t_1\big\} \cap \bigcap_{i = 1}^{d-1} \left\{\gamma(\mu(B_n))\, \lr_{B_n}^{(i)}\circ T_{B_n} \leq t_d - \gamma(\mu(B_n))\,\lr_{B_n}\right\}\right)
    \end{align*}
    and for all $\varepsilon > 0$ such that $(t_2 - \varepsilon, \dots, t_d - \varepsilon)$ is a continuity point of $\widetilde{F}^{[d-1]}$, we have
    \begin{align*}
        & \mu_{U(B_n)}(\gamma(\mu(B_n))\,\lr_{B_n} \leq t_1,\dots, \gamma(\mu(B_n))\,\lr_{B_n}^{(d)} \leq t_d) \\
        & \quad \geq \mu_{U(B_n)}\left(\bigcap_{i = 1}^{d-1} \left\{ \gamma(\mu(B_n))\,\lr_{B_n}^{(i)} \circ T_{B_n} \leq t_{i+1}- \varepsilon\right\} \right) \\
        & \qquad - \mu_{U(B_n)}(\gamma(\mu(B_n))\,\lr_{B_n}\geq t_1) - \mu(\gamma(\mu(B_n))\,\lr_{B_n} \geq \varepsilon).
    \end{align*}
      Again by \ref{cond_CFPP:Compatibility_Geometric_law}, the first term converges towards $\widetilde{F}^{[d]}(t_2 - \varepsilon, \dots, t_d - \varepsilon)$. For the two other terms, by \ref{cond_CFPP:tau_n_small_enough} we have $\mu_{B_n}(\gamma(\mu(B_n)) \tau_n \geq \varepsilon) \xrightarrow[n\to +\infty]{} 0$ and thus $\mu_{U_n}(\gamma(\mu(B_n)) \tau_n \geq \varepsilon) \xrightarrow[n\to +\infty]{} 0$ because $\theta < 1$. Together with \ref{cond_CFPP:cluster_compatible_tau_n_cluster_from_U} it ensures the convergence to $0$. Since $(t_1, \dots, t_d)$ is a continuity point of $\widetilde{F}^{[d-1]}$, letting $\varepsilon \to 0$ gives the result.
    \end{proof}

    \begin{lem}
        \label{lem:second_term_decomposition_in_proof_sufficient_condition_CFPP}
        For every $d\geq 1$ and $0\leq t_1\leq \dots \leq t_d$ such that $(t_1,\dots, t_d)$ is a continuity point of $F^{[d]}$, we have
        \begin{align*}
            \frac{\mu(Q(B_n))}{\mu(B_n)}\,\mu_{Q_{B_n}}\big(\gamma(\mu(B_n))\,\lr_{B_n} \leq t_1, \dots, \gamma(\mu(B_n))\,\lr^{(d)}_{B_n} \leq t_d \big) \xrightarrow[n\to +\infty]{} \theta F^{[d]}(t_1, \dots, t_d)\,.
        \end{align*}  
    \end{lem}
    
    \begin{proof}[Proof (of Lemma \ref{lem:second_term_decomposition_in_proof_sufficient_condition_CFPP})]
      This time \ref{cond_CFPP:extremal_index} gives $\mu(Q(B_n))/\mu(B_n) \xrightarrow[n\to +\infty]{} \theta$. By \ref{cond:Living_in_uniform_set}, there exist a function $u \in L^1(\mu)$ such that $Y$ is $u$-uniform (without loss of generality we assume $\int u \,\dd\mu = 1$). We consider $\mu_{v_n}$ the probability absolutely continuous with respect to $\mu$ and of density $v_n :=  \widehat{T^{\tau_n}}(\mathbf{1}_{Q(B_n)}/\mu(Q(B_n)))$ and write $(F_{B_n,v_n}^{[d]})_{d\geq 1}$ the family of distribution functions of $\gamma(\mu(B_n))\Phi_{B_n}$ drawn from $\mu_{v_n}$. By \ref{cond_CFPP:good_density_after_tau_n}, for all $n\geq 1$, $v_n\in \mathcal{U}$. 

      By definition of $v_n$, we have $(\gamma(\mu(B_n))\Phi_{B_n})_{\#} \mu_{v_n} = (\gamma(\mu(B_n))\Phi_{B_n} \circ T^{\tau_n})_{\#} \mu_{Q(B_n)}$. On $\{r_{B_n} > \tau_n\}$, we have $\gamma(\mu(B_n))\Phi_{B_n} = \gamma(\mu(B_n))\Phi_{B_n} \circ T^{\tau_n} +  \gamma(\mu(B_n))\tau_n$, so by \ref{cond_CFPP:tau_n_small_enough}- \ref{cond_CFPP:cluster_compatible_tau_n_no_cluster_from_Q} we get $d(\gamma(\mu(B_n))\Phi_{B_n}, \gamma(\mu(B_n))\Phi_{B_n}\circ T^{\tau_n}) \xrightarrow[n\to +\infty]{\mu_{Q(B_n)}} 0$. In particular, for all $d\geq 1$, we have 
    \begin{align}
        \label{eq:first_step_equivalence_convergence_cluster}
        \mu_{Q(B_n)}(\gamma(\mu(B_n))\,\lr_{B_n} \leq t_1,\dots, \gamma(\mu(B_n))\,\lr_{B_n}^{(d)}\leq t_d) - F_{B_n, v_n}^{[d]}(t_1,\dots, t_d) \xrightarrow[n\to +\infty]{} 0. 
    \end{align}

      Now, we show that the stochastic processes $(\gamma(\mu(B_n))\Phi_{B_n}$ drawn from $\mu_{v_n}$ and $\mu_v$ share the same limit. We denote $(F_{B_n,v_n}^{[d]})_{d\geq 1}$ and $(F_{B_n,v}^{[d]})_{d\geq 1}$ their respective family of distribution functions of the finite-dimensional marginals. For every $d\geq 1$, consider a bounded Lipschitz function $\psi : \mathbb{R}^d\to \mathbb{R}^d$. We are going to show that 
    \begin{align*}
        \int \psi \circ (\gamma(\mu(B_n))\Phi^{[d]}_{B_n})\, u\,\dd\mu - \int \psi \circ (\gamma(\mu(B_n))\Phi^{[d]}_{B_n})\, u^*\,\dd\mu \xrightarrow[n\to +\infty]{} 0 \quad \text{uniformly in $u,u^* \in \mathcal{U}$.}
    \end{align*}
      Let $\varepsilon > 0$ and consider $M$ large enough so that the quantity in Lemma \ref{lem:convergence_Birkhoff_transfer_operator_compact} is smaller than $\varepsilon$. It yields,
    \begin{align*}
        \left|\int \psi \circ \big(\gamma(\mu(B_n))\Phi^{[d]}_{B_n}\big) \cdot \frac{1}{M}\sum_{j = 0}^{M - 1} \big(\widehat{T}^ju - \widehat{T}^ju^*\big)\,\dd\mu \right| & \leq \sup |\psi| \left\| \frac{1}{M}\sum_{j = 0}^{M - 1} \big(\widehat{T}^ju - \widehat{T}^ju^* \big) \right\|_{L^1(\mu)}\\
        & \leq \varepsilon \sup |\psi| .
    \end{align*}
    Furthermore, we also have 
    \begin{align*}
        &\left| \int \psi \circ \big(\gamma(\mu(B_n))\Phi^{[d]}_{B_n}\big) \Big(u - \frac{1}{M}\sum_{j= 0}^{M-1}\widehat{T}^ju\Big)\,\dd\mu\right| \\
        & \qquad\leq \frac{1}{M} \sum_{j=0}^{M-1} \int \left| \psi \circ (\gamma(\mu(B_n))\Phi^{[d]}_{B_n}) - \psi \circ \big(\gamma(\mu(B_n))\Phi^{[d]}_{B_n}\big) \circ T^{j} \right|u\,\dd\mu \\
        & \qquad \leq \frac{1}{M} \sum_{j= 0}^{M-1} \left( 2 \sup |\psi| \int_{\{r_{B_n} \leq j\}} u\,\dd\mu + \lip(\psi)\,\gamma(\mu(B_n))j\right) \\
        & \qquad \leq 2\sup |\psi| \int_{\{r_{B_n}\leq M\}} u\,\dd\mu + \lip(\psi) \,\gamma(\mu(B_n))M \\
        & \qquad \leq \varepsilon \; \text{for $n$ large enough and uniformly on $u\in \mathcal{U}$ since $\mathcal{U}$ is uniformly integrable.}
    \end{align*}
    
      Thus, by Portemanteau theorem, for every $d\geq 1$
    \begin{align}
        \label{eq:second_step_equivalence_convergence_cluster}
        F_{B_n,v_n}^{[d]}(t_1,\dots, t_d) - F_{B_n,v}^{[d]}(t_1,\dots, t_d) \xrightarrow[n\to +\infty]{} 0.
    \end{align}
    Since $F_{B_n,v}^{[d]}(t_1,\dots, t_d)$ converges towards $F^{[d]}(t_1, \dots, t_d)$, \eqref{eq:first_step_equivalence_convergence_cluster} and \eqref{eq:second_step_equivalence_convergence_cluster} together give the desired result.
    \end{proof}

      Now, going back to \eqref{eq:CFPP_decomposition_between_immediate_return_and_escaping_annulus} and with Lemmas \ref{lem:first_term_decomposition_in_proof_sufficient_condition_CFPP} and \ref{lem:second_term_decomposition_in_proof_sufficient_condition_CFPP}, we get, for all $d\geq 1$ and $t_1 \leq\dots \leq t_d$ such that $(t_1,\dots, t_d)$ and $(t_2,\dots, t_d)$ are continuity points of $\widetilde{F}^[d], F^{[d]}$ and $\widetilde{F}^{[d-1]}$ we have 
    \begin{align*}
        \widetilde{F}_{B_n}^{[d]} (t_1,\dots, t_d) \xrightarrow[n\to +\infty]{} (1 - \theta)\widetilde{F}^{[d-1]}(t_2,\dots, t_d) + \theta F^{[d]}(t_1,\dots, t_d)\,.
    \end{align*}
    But we assumed that $\widetilde{F}_{B_n}^{[d]} (t_1,\dots, t_d) \xrightarrow[n\to +\infty]{} \widetilde{F}^{[d]}(t_1,\dots, t_d)$. Hence, we get the following equality
    \begin{align*}
        \widetilde{F}^{[d]}(t_1,\dots, t_d) = (1 - \theta)\widetilde{F}^{[d-1]}(t_2,\dots, t_d) + \theta F^{[d]}(t_1,\dots, t_d).
    \end{align*}
      Using Theorem \ref{thm:HTS-REPP_vs_RTS-REPP_infinite_measure_renormalized_measure}, we can express $F^{[d]}$ from $\widetilde{F}^{[d]}$ and $\widetilde{F}^{[d-1]}$ by \eqref{eq:relation_hitting_return_infinite_measure} we get
    \begin{align*}
        \widetilde{F}^{[d]}(t_1,\dots, t_d) & = (1 - \theta) \widetilde{F}^{[d-1]}(t_1,\dots, t_d) \nonumber \\
        & \quad + \theta \alpha \int_0^{t_1} \Bigl( \widetilde{F}^{[d-1]}\left(t_2 -t_1 + x, \dots, t_d - t_1 + x\right) \\
        & \quad - \widetilde{F}^{[d]}\left(x, t_2 - t_1 + x,\dots, t_d - t_1 +x\right)\Bigr)  (t_1 - x)^{\alpha - 1}\,\dd x\,,
    \end{align*}
    which is exactly \eqref{eq:equation_CFPP_fix_point}. In the proof of Proposition \ref{prop:fix_point_equation_FPP}, we showed that the only process such that the distribution functions of its marginals satisfy \eqref{eq:equation_CFPP_fix_point} is $\Phi_{\RPP(W_{\alpha,\theta}(\theta\Gamma(1+\alpha)))}$. Thus, we get the convergence
    \begin{align*}
        N_{B_n}^{\gamma} \xRightarrow[n\to +\infty]{\mu_{B_n}} \RPP(W_{\alpha,\theta}(\theta\Gamma(1+\alpha)))\,.
    \end{align*}
    By \eqref{eq:relation_hitting_return_infinite_measure} and Corollary \ref{cor:equivalence_HTS_RTS_true_Point_Processes}, we also obtain
    \begin{align*}
        N_{B_n}^{\gamma} \xRightarrow[n\to +\infty]{\mathcal{L}(\mu)} \cfPp_{\alpha}(\theta\Gamma(1+\alpha), \geo(\theta)).
    \end{align*}
\end{proof}

\section{Quantitative recurrence for maps with an indifferent fixed point}
\label{section:proofs_LSV_map}

  We now focus on the family of maps with one indifferent fixed point with LSV parameters and prove the results presented in Section \ref{Paragraph_map_with_one_indifferent_fixed_point}.

\subsection{Additional properties}

\label{subsection:preliminaries_LSVmaps}

  We recall the basic properties of the map $T$ defined in \eqref{def:LSVmap} where $p\geq 1$. Recall that there is a unique absolutely continuous invariant measure $\mu$ and $\mu([0,1]) = +\infty$ if and only if $p\geq 1$. Recall the definitions of $c_n = T_1^{-n} 1$, where $T_1$ is the first branch of $T$, and of the partitions $\xi = \{[c_{n+1},c_n],\, n\geq 0\}$ and $\xi_k  = \bigvee_{j=0}^{k-1} T^{-j}\xi$. \\

Although we will not directly use the coding of the map by a renewal shift, we will take advantage of this perspective and use symbolic notations in what follows. 

  For all $n \geq 0$, let $[n] := [c_{n+1},c_n]$. In particular,  $\xi= \{[n], \;n\geq 0\}$. By definition of $T$, for $n\geq 1$,  $T_{|[n]} : [n] \to [n-1]$ and $T_{|[0]} : [0] \to [0,1]$ are diffeomorphisms. Similarly, the partition $\xi_k$ can be expressed with cylinder notations.

\begin{lem}
    \label{lem:description_partition_xi_with_symbolic_notations}
    For all $k \geq 0$, we have 
    \begin{align*}
    \xi_k = \bigvee_{j = 0}^{k-1} T^{-j}\xi = \big\{[a_0^{k-1}] \;\big| \; (a_0^{k-1})\in \mathbb{N}^k,\; \forall 0\leq i \leq k-2,\; a_{i+1} = a_i - 1 \;\text{or} \; a_i = 0\big\},  
    \end{align*}
    where $[a_0^{k-1}] = \{ x \in [0,1]\; |\; T^j(x) \in [a_j]\; \forall\, 0\leq j \leq k-1\}$.
\end{lem}

  Using the terminology of symbolic dynamics, we will say that $(a_0^{k-1})$ is admissible if $[a_0^{k-1}] \neq \emptyset$.

\begin{proof}[Proof (of Lemma \ref{lem:description_partition_xi_with_symbolic_notations})]
    This can be easily seen by induction. When $k = 1$, this is the definition of $\xi$. Now, $\xi_{k+1} = \xi_{k} \vee T^{-1}\xi_k$. Consider an element $[a_0^{k-1}]$ in $\xi_k$. Then, $T^{-1}[a_0^{k-1}] = T_1^{-1}[a_0^{k-1}] \sqcup T_2^{-1}[a_0^{k-1}]$. By definition $T_2^{-1}[a_0^{k-1}] = [0a_0^{k-1}]$ and since $T_1^{-1}[a_0] = [a_0+1]$, we have $T_1^{-1}[a_0^{k-1}] = [(a_0 +1)a_0^{k-1}]$ and $(0 a_0^{k-1})$ and $((a_0+1)a_0^{k-1})$ are admissible. Reciprocally, for an admissible sequence, we have
    \begin{align}
        \label{eq:representation_cylinders_in_[0,1]}
        [a_0^{k-1}] = T^{-1}_{\sigma(a_{0})} \cdots T^{-1}_{\sigma(a_{k-2})}[a_{k-1}] \in \xi_k, 
    \end{align}
    where $\sigma(0) := 2$ and $\sigma(k) := 1$ for all $k\geq 1$.
\end{proof}

\begin{rem}
    Due to the Markov property of the partition, for all $k\geq 1$ and all $(a_0^{k-1}) \in \mathbb{N}^k$ admissible, $T^j : [a_0^{k-1}] \to [a_j^{k-1}]$ is a homeomorphism. Furthermore, we always have $[a_0^k] = [a_0^k (a_{k} -1)(a_k -2)\dots 0]$.
\end{rem}

  For all $0 \leq n \leq q \leq \infty$, we define the sets 
\begin{align*}
    [[n,q]] := \bigsqcup_{n\leq k\leq q} [k] = [c_{q+1},c_{n}]. 
\end{align*}
  When $q = \infty$, we simply write $[(\geq n)] := [[n,\infty]] = [0,c_{n}]$ and when $n = 0$, $[(\leq q)] := [[0,q]] = [c_{q+1},1]$.

\begin{lem}
    \label{lem:cylinders_are_intervals}
    For all $k \geq 1$ and $(a_0^{k-1}) \in \mathbb{N}^k$ admissible, $[a_0^{k-1}]$ is an interval. If furthermore $a_{k-1} = 0$, then, for all $0\leq n \leq q \leq \infty$, $[a_0^{k-1}[n,q]]$ is also an interval.
\end{lem}

\begin{proof}[Proof (of Lemma \ref{lem:cylinders_are_intervals})]
    This is immediate from \eqref{eq:representation_cylinders_in_[0,1]} and because $[[n,q]]$ are intervals for every $0\leq n\leq q \leq \infty$. 
\end{proof}

  For every $j\geq 0$, it is known that the induced map on $Y := [c_{j+1},1]$ has a Gibbs-Markov structure. Indeed, define
\begin{align*}
    \xi^Y := \xi^Y_1 = \{[i],\; 1\leq i\leq j\} \sqcup \{[0i],\; i\geq 0\},
\end{align*}
  and for all $k\geq 1$
\begin{align*}
    \xi^Y_k := \bigvee_{i = 0}^{k-1} T_Y^{-i} \xi^Y_1.
\end{align*}
  Then, $(Y,T_Y, \mu_Y)$ is Gibbs-Markov with respect to the partition $\xi^Y$. We can also define 
\begin{align*}
    \xi^Y_0 := T_Y\xi^Y = \xi \cap Y = \{[i],\; 0\leq i\leq j\} = \{[c_{i+1},c_i],\; 0\leq i\leq j]\}.
\end{align*}

\begin{rem}
    If $j = 0$, \textit{i.e.} $Y = [1/2,1]$, the system is full-branched Gibbs-Markov. Otherwise, it is Gibbs-Markov with the ``big image'' property.
\end{rem}

  We also recall some useful estimates. This can be found in \cite{You99} or \cite{LSV97} for example.
\begin{lem}
    We have the following asymptotic results
    \begin{align}
        \label{eq:measure_union_cylindres_asymptotique}
        \mu[0(\geq n)] \asymp \frac{1}{n^{\alpha}} 
    \end{align}
    and 
    \begin{align}
        \label{eq:measure_1-cylindres_asymptotique}
        \mu[0n] \asymp \frac{1}{n^{\alpha+1}}\,,
    \end{align}
    where $a_n\asymp b_n$ means that there exists $C>1$ such that $C^{-1}b_n\leq a_n\leq C b_n$, for all $n\geq 1$.
\end{lem}

  Finally, we recall the following lemma (see \cite{PSZ13} p.31 and \cite{Tha80_EstimatesInvariantDensities, Tha83}) giving a distortion bound for good subsets (this lemma is true for the bigger class of $\textup{AFN}$-maps). 

\begin{lem}
    \label{lem:Distortion_estimates_Thaler}
    Let $Y$ be of the form $[c_{j+1},1]$ for some $j\geq 0$. Then, if $V,W \subset Y$ are intervals such that there exists $m\in \mathbb{N}$ with $T^m : V \to W$ being a homeomorphism onto $W$, then $\widehat{T}^m\mathbf{1}_V \in \mathcal{C}_C(W)$ for some $C > 0$ depending only on $Y$, where
    \begin{align*}
        \mathcal{C}_C(W) :=  \{f : Y \to [0,+\infty)\;|\; |f(x)/f(y) - 1| \leq Cd(x,y) \; \forall x,y \in W\}.  
    \end{align*}
\end{lem}

  From this lemma we are able to derive the following corollary that is going to be of crucial importance in the proof of Theorem \ref{thm:REPP_LSV_preimages_of_0}.

\begin{cor}
    \label{cor:distortion_bounds_comparison_measures}
    There exists a constant $C > 0$ such that for all $k,\ell,m \geq 0$, for all $(a_0^{k-1})$, $(b_0^{\ell - 1})$, $(c_0^{m -1})$ admissible (with the possibility that $b_{\ell - 1}$ and $c_{m-1}$ are of the form $[i,j]$ for some $0\leq i\leq j \leq \infty$) such that $[0a_0^{k-1}0b_0^{\ell - 1}] \neq \emptyset$ and $[0a_0^{k-1}0c_0^{m-1}] \neq \emptyset$, we get
    \begin{align*}
        &\frac{\mu[0a_0^{k-1}0b_0^{\ell - 1}]}{\mu[0a_0^{k-1}0c_0^{m-1}]} =  \left( 1 \pm C\diam([0b_0^{\ell-1}] \cup [0c_0^{m-1}]) \right) \frac{\mu[0b_0^{\ell - 1}]}{\mu[0c_0^{m-1}]}\,,
    \end{align*}
    where $x = (1\pm C)y$ means that $(1-C)y\leq x \leq (1+C)y$.
\end{cor}

\begin{proof}[Proof (of Corollary \ref{cor:distortion_bounds_comparison_measures})]
    Consider the constant $C$ from Lemma \ref{lem:Distortion_estimates_Thaler} with $Y = [1/2,1] =[0]$. Let $j := \min\{k\geq 0\;|\; b_k \neq c_k\}$ with the convention that $j := \min\{\ell,m\} - 1$ if $b_0^{\min\{\ell,m\} - 1} = c_0^{\min\{\ell,m\} - 1}$. In particular, we have $\diam([0b_0^{\ell-1}] \cup [0c_0^{m-1}]) = \diam([0b_0^{j-1}])$. By a property of the map $T$, we know that $T^{k+1} : [0a_0^{k-1}0b_0^{j-1}] \mapsto [0b_0^{j-1}]$ is an homeomorphism with both $[0a_0^{k-1}0b_0^{j-1}]$ and $[0b_0^{j-1}]$ being intervals of $Y$. Thus, by Lemma \ref{lem:Distortion_estimates_Thaler}, we have $\widehat{T}^{k+1} \mathbf{1}_{[0a_0^{k-1}0b_0^{j-1}]} \in \mathcal{C}_C([0b_0^{j-1}])$. Consequently,
    \begin{align*}
        \mu[0a_0^{k-1}0b_0^{\ell-1}] & = \int \mathbf{1}_{[0a_0^{k-1}0b_0^{\ell-1}]}\,\dd\mu = \int \widehat{T}^{k+1} \mathbf{1}_{[0a_0^{k-1}0b_0^{\ell-1}]}\dd\mu \\
        & =  \int \mathbf{1}_{[0b_0^{\ell-1}]}\cdot \widehat{T}^{k+1} \mathbf{1}_{[0a_0^{k-1}0b_0^{j-1}]}\,\dd\mu \\
        & = \widehat{T}^{k+1} \mathbf{1}_{[0a_0^{k-1}0b_0^{j-1}]}(y) \int \mathbf{1}_{[0b_0^{\ell-1}]}(x) \cdot \frac{\widehat{T}^{k+1} \mathbf{1}_{[0a_0^{k-1}0b_0^{j-1}]}(x)}{\widehat{T}^{k+1} \mathbf{1}_{[0a_0^{k-1}0b_0^{j-1}]}(y)}\,\dd\mu \quad (\forall y \in [0c_0^{m-1}])\\
        & = \left(1 \pm Cd(x,y)\right) \widehat{T}^{k+1} \mathbf{1}_{[0a_0^{k-1}0b_0^{j-1}]}(y)\int \mathbf{1}_{[0b_0^{\ell-1}]}\dd\mu \quad \forall y \in [0c_0^{m-1}]\\
        & = \left(1\pm C\diam([0b_0^{j-1}])\right) \mu[0b_0^{\ell - 1}]\, \mathbb{E}_{\mu_{[0c_0^{m-1}]}}\left[\widehat{T}^{k+1} \mathbf{1}_{[0a_0^{k-1}0b_0^{j-1}]}\right]\\
        & = \left(1\pm C\diam([0b_0^{j-1}])\right) \frac{\mu[0b_0^{\ell - 1}]}{\mu[0c_0^{m-1}]}\int \mathbf{1}_{[0c_0^{m-1}]}\, \widehat{T}^{k+1} \mathbf{1}_{[0a_0^{k-1}0b_0^{j-1}]} \,\dd\mu \\
        & = \left(1\pm C\diam([0b_0^{j-1}])\right) \frac{\mu[0b_0^{\ell - 1}]}{\mu[0c_0^{m-1}]} \,\mu[0a_0^{k-1}0c_0^{m-1}].
    \end{align*}
\end{proof}

\begin{rem}
    Although we are using notations coming from the symbolic dynamics, we are still working on the interval with its topology. For example, $\diam([0(\geq p)]) \xrightarrow[p\to +\infty]{} 0$ in our case, whereas $\diam([0(\geq p)]) = 1/4$ for all $p\geq 0$ with the usual geometric topology of countable Markov shifts.  
\end{rem}

\subsection{Proof of Theorems \ref{thm:REPP_LSV_nonperiodic_points}-\ref{thm:REPP_LSV_periodic_points}}

  Once $x \in [0,1]\backslash \bigcup_{k\geq 0} T^{-k}\{0\}$ is chosen, we fix $Y$ of the form $[c_{j+1}, 1]$ such that $x \in Y$. Furthermore, there exists a sequence of integers $(a_k)_{k\geq 0}$ such that $(a_0^{n-1})$ is admissible for all $n$ and $B_n = \xi_n(x) =  [a_0^{n-1}]$. Set $s_n := S_{n - 1} \mathbf{1}_Y(x)$ the number of visits to $Y$ before $n$ (note that $s_n \geq 1$ for $n\geq 2$ since we assumed $x \in Y$). 
\begin{lem}
    \label{lem:equivalence_cylinders_induced_and_cylinders}
    For all $n\geq 0$, we have 
    \begin{align*}
        B_n = \xi^Y_{s_n}(x). 
    \end{align*}
\end{lem}

\begin{proof}[Proof (of Lemma \ref{lem:equivalence_cylinders_induced_and_cylinders})]
    We have $B_1 = \xi(x) = \xi^Y_0(x)$ and $s_1 = S_0\mathbf{1}_Y(x) = 0$. Assume that the result is true for some $n\geq 1$. We have $B_n = \xi_{n}(x) = [a_0^{n-1}]$ and $B_{n+1} = \xi_{n+1}(x) = [a_0^n]$. There are three cases, if $a_{n-1} > j$, $B_{n+1} = [a_0^{n-1} (a_{n-1} - 1)] = [a_0^{n-1}] = B_n$ et $s_{n+1} = S_n\mathbf{1}_Y(x) = S_{n-1}\mathbf{1}_Y(x) = s_n$ meaning that $\xi^Y_{s_{n+1}}(x) = \xi^Y_{s_{n}}(x) = B_n = B_{n+1}$. If $1 \leq a_{n-1} \leq j$, $s_{n+1} = s_n + 1$ but $\xi^Y_{s_n + 1}(x) = \xi^Y_{s_n}(x)$ and again $\xi^Y_{s_{n+1}}(x) = \xi^Y_{s_n}(x) = B_n = B_{n+1}$. \\
    Finally, if $a_{n-1} = 0$, then $s_{n+1} = s_n + 1$ and $\xi^Y_{s_{n} + 1}(x) = [a_0^n] = B_{n+1}$ . 
\end{proof}

\paragraph{Proof of Theorem \ref{thm:REPP_LSV_nonperiodic_points}: Non periodic points.}

\begin{proof}[Proof (of Theorem \ref{thm:REPP_LSV_nonperiodic_points})]
    Let $x \in [0,1]$ be non periodic and not a preimage of $0$. Let $\tau_n := \min\{i \geq n - 1 \;|\; T^ix\in Y\}$. Thus $s_n = S_{n -1}\mathbf{1}_Y(x) =  S_{\tau_n}\mathbf{1}_Y(x)$ and $B_n = \xi_n(x) = \xi^{Y}_{s_n}(x)$ for $n\geq 1$ by Lemma \ref{lem:equivalence_cylinders_induced_and_cylinders}. Since $(Y, T_Y, \mu_Y, \xi^Y)$ is Gibbs-Markov, there exists $\kappa, c > 0$ such that $\mu(B_n) \leq \kappa e^{-cs_n}$.\\
    We will use Theorem \ref{thm:sufficient_conditions_convergence_compound_FPP} to get the result. Thus, we need to prove conditions \ref{cond:Living_in_uniform_set}-\ref{cond_CFPP:Compatibility_Geometric_law}. \ref{cond:Living_in_uniform_set} is immediate since $Y$ is a uniform set. Since there is no appearance of clusters in the non periodic case, we set $U(B_n) = \emptyset$ and $Q(B_n) = 1$ for \ref{cond_CFPP:extremal_index}. It remains to show \ref{cond_CFPP:good_density_after_tau_n}-\ref{cond_CFPP:cluster_compatible_tau_n_no_cluster_from_Q} for our particular deterministic choice $\tau_n$. We start with condition \ref{cond_CFPP:tau_n_small_enough}. In fact, we will even show the stronger result $\tau_n\mu(B_n) \xrightarrow[n\to +\infty]{\mu_{B_n}} 0$. For that there are two cases, either $x$ is a "well-behaved" point, meaning is will return to $Y$ often enough, or $x$ does not come back sufficiently often to $Y$. In the first case, the exponential decay of the cylinder measures for the induced map is enough to get the result. In the second case, since $x$ is a special point, the measure of its cylinder will already be small enough.
    Consider a sequence $(n_k)$ such that $\lim_{k\to +\infty} \tau_{n_k}\mu(B_{n_k}) = \limsup_{n\to +\infty} \tau_n\mu(B_n)$. Without loss of generality, we can assume that $(S_{\tau_{n_k}}\mathbf{1}_Y(x)/\log(\tau_{n_k}))_{k\geq 0}$ is either bounded or diverges to $+\infty$. \\
    If it diverges to $+\infty$, we have 
    \begin{align*}
        \tau_{n_k}\mu(B_{n_k}) \leq \tau_{n_k} \kappa\, e^{-cs_{n_k}} = \tau_{n_k} \kappa\, e^{-cS_{\tau_{n_k}}\mathbf{1}_{Y}(x)} \xrightarrow[k\to +\infty]{} 0.
    \end{align*}
    Otherwise, if the sequence is bounded by $K > 0$, we have $S_{\tau_{n_k}}\mathbf{1}_Y(x) \leq K\log(\tau_{n_k})$ and thus, there exists some $j\leq K\log(\tau_{n_k})$ such that $r_Y^{\{j\}}(x)\geq \tau_{n_k}/(K\log(\tau_{n_k}))$. It yields,
    \begin{align*}
        \tau_{n_k}\mu(B_{n_k}) & \leq \tau_{n_k}\mu\big(Y \cap T_Y^{-j+1}\{r_Y = r_Y^{\{j\}}(x)\}\big) \leq \tau_{n_k}\mu\big(Y \cap \{r_Y = r_Y^{\{j\}}(x)\}\big) \\
        & \leq C\tau_{n_k} \frac{1}{r_Y^{\{j\}}(x)^{\alpha+1}} \leq CK^{\alpha} \frac{\tau_{n_k}\log(\tau_{n_k})^{\alpha+1}}{\tau_{n_k}^{\alpha+1}} \xrightarrow[k\to +\infty]{} 0 \quad \text{(using \eqref{eq:measure_1-cylindres_asymptotique})}.
    \end{align*}
      Hence condition \ref{cond_CFPP:tau_n_small_enough} is satisfied. For condition \ref{cond_CFPP:cluster_compatible_tau_n_no_cluster_from_Q}, we can use \ref{cond_CFPP:tau_n_small_enough} and the fact that we have return time statistics towards the exponential law for the induced map (because it is Gibbs-Markov). We get 
    \begin{align*}
        \mu_{B_n}(\lr_{B_n} < \tau_n) &\leq \mu_{B_n}\Big(r_Y^{(r^Y_{B_n})} < \tau_n\Big) \leq \mu_{B_n} \big(S_{\tau_n}\mathbf{1}_Y > \lr^Y_{B_n}\big) \\
        & \leq  \mu_{B_n}\big(\mu_Y(B_n)\,\lr^Y_{B_n} < \mu_Y(B_n) S_{\tau_n}\mathbf{1}_Y) \\
        & \leq \mu_{B_n}(\mu_Y(B_n)\,\lr^Y_{B_n} < t) + \mu_{B_n}\big(\mu_Y(B_n)S_{\tau_n}\mathbf{1}_Y > t\big),
    \end{align*}
    for all $t > 0$. Fix $\varepsilon > 0$. Since $B_n$ has return time statistics for the induced map, we can take $t$ small enough so that $\limsup_{n\to +\infty} \mu_{B_n}(\mu_Y(B_n)r^Y_{B_n} < t) \leq \varepsilon$. On the other hand, $S_{\tau_n}\mathbf{1}_Y \leq \tau_n$ and using $\mu(B_n)\tau_n \xrightarrow[n\to +\infty]{\mu_{B_n}} 0$, it yields
    \begin{align*}
        \limsup_{n\to +\infty} \mu_{B_n}(r_{B_n} < \tau_n) \leq \varepsilon,
    \end{align*}
    proving \ref{cond_CFPP:cluster_compatible_tau_n_no_cluster_from_Q}.\\

      It remains to show \ref{cond_CFPP:good_density_after_tau_n}. Since $B_n = \xi_{s_n}^Y(x)$, we have
    \begin{align*}
        \widehat{T^{\tau_n}}(\mathbf{1}_{B_n}/\mu(B_n)) = \widehat{T_Y^{s_n}} (\mathbf{1}_{\xi_{s_n}^Y(x)}/\mu(\xi_{s_n}^Y(x)))\,.
    \end{align*}
    As the return map to $Y$ is Gibbs-Markov, by \cite[\S 4.7]{Aar97} (see also \cite[Section 10.2]{Zwe22}), there exists some $\mathcal{U}$ compact in $L^1(\mu)$ such that \(\widehat{T^{\tau_n}}(\mathbf{1}_{B_n}/\mu(B_n)) \in \mathcal{U}\) for every $n\geq 1$. 

      Thus, since every condition of Theorem \ref{thm:sufficient_conditions_convergence_compound_FPP} in the case $U(B_n) = \emptyset$ is satisfied, it gives us the wanted result.
\end{proof}

\paragraph{Proof of Theorem \ref{thm:REPP_LSV_periodic_points}: Periodic points.} 
\begin{proof}[Proof (of Theorem \ref{thm:REPP_LSV_periodic_points})]
Assume now that $x \in [0,1] \backslash \bigcup_{k\geq 0} T^{-k}\{0\}$ is periodic of prime period $q$. Again, we want to use Theorem \ref{thm:sufficient_conditions_convergence_compound_FPP} and thus we have to show \ref{cond:Living_in_uniform_set}-\ref{cond_CFPP:Compatibility_Geometric_law}. \ref{cond:Living_in_uniform_set} is immediate since $Y$ is a uniform set. Recall that $B_n = \xi_n(x) = \xi^Y_{s_n}(x)$ with $s_n = S_{n-1}\mathbf{1}_Y(x)$. Let $Q(B_n) := B_n \cap T^{-q}B_n^c$ be the escape annulus and $U(B_n) = B_n \backslash Q(B_n) = B_n \cap T^{-q}B_n$. We start by conditions \ref{cond_CFPP:extremal_index}and \ref{cond_CFPP:Compatibility_Geometric_law}. We have 
\begin{align*}
    \mu(U(B_n)) & = \mu(B_n \cap T^{-q}B_n) = \int \mathbf{1}_{B_n}\cdot \mathbf{1}_{B_n}\circ T^q \,\dd\mu = \int \mathbf{1}_{B_n}\cdot \mathbf{1}_{B_n}\circ T^q \rho \,\dd\leb\,.
\end{align*}
For $n$ large enough such that $T^jB_n \cap B_n = \emptyset$ for $1\leq j < q$ and by construction of $U(B_n)$, $T_{B_n} = T^q$ on $U(B_n)$ and $T^q : U(B_n) \to B_n$ is a diffeomorphism. Thus, for every $y \in B_n$
\begin{align*}
    \widehat{T_{B_n}}\mathbf{1}_{U(B_n)}(y) = \widehat{T}^q\mathbf{1}_{U(B_n)}(y) \mathbf{1}_{B_n}(y) = \frac{1}{(T^q)'(T^{-q}_{U(B_n)}y)}
\end{align*}
where $T^{-q}_{U(B_n)}y$ is the only $q$-preimage of $y$ in $U(B_n)$. Since $T'$ is continuous at $x$ (and $T'(x) > 0$), $\|1/(T'\circ T^{-q}_{U(B_n)}) - 1/(T')(x)\|_{L^{\infty}(\mu_{B_n})} \xrightarrow[n\to +\infty]{} 0$. Thus, using that 
$\widehat{T_{B_n}}\left(\frac{\mathbf{1}_{U(B_n)}}{\mu(B_n)}\right)$ and $\frac{\mathbf{1}_{B_n}}{\mu(B_n)}$ are probability densities, we have 
\begin{align*}
    \left\|\,\widehat{T_{B_n}}\left(\frac{\mathbf{1}_{U(B_n)}}{\mu(B_n)}\right) - \frac{\mathbf{1}_{B_n}}{\mu(B_n)} \right\|_{L^{\infty}(\mu_{B_n})} \xrightarrow[n\to +\infty]{} 0
\end{align*}
and 
\begin{align*}
    \frac{\mu(U(B_n))}{\mu(B_n)} \xrightarrow[n\to +\infty]{} \frac{1}{(T^q)'(x)},
\end{align*}
proving conditions \ref{cond_CFPP:extremal_index} and \ref{cond_CFPP:Compatibility_Geometric_law} with 
\begin{align*}
    \theta = 1 - \frac{1}{(T^q)'(x)}.
\end{align*}

  For \ref{cond_CFPP:cluster_compatible_tau_n_cluster_from_U}, by construction, we have $U(B_n) = B_{n+q} = \xi^Y_{s_{n + q}}(x)$ and we set $\tau_n := \min \{ i \geq n+q-1\; |\; T^ix \in Y\}$ so that $s_{n+q} = S_{\tau_n}\mathbf{1}_Y(x)$. First, since $r_{B_n} = q$ on $U(B_n)$, we immediately get
\begin{align*}
    \mu_{U(B_n)}(\lr_{B_n} > \tau_n) = \mu_{U(B_n)}(q > \tau_n) \xrightarrow[n\to +\infty]{} 0.
\end{align*}
Furthermore, since $x$ is a periodic point, $S_n\mathbf{1}_Y(x)/ \log(n) \xrightarrow[n\to +\infty]{} +\infty$. Thus, since the measure of cylinders in Gibbs-Markov maps decays exponentially and $s_{n + q} \leq s_n + q$ 
\begin{align*}
    \tau_n\, \mu(B_n) \leq \tau_n \kappa\, e^{- c s_n} \leq \tau_n\,  \kappa e^{q} e^{-cs_{n+q}} = (\kappa e^q)\tau_n e^{S_{\tau_n}\mathbf{1}_Y(x)} \xrightarrow[n\to +\infty]{} 0,
\end{align*}
  which proves \ref{cond_CFPP:tau_n_small_enough}. For \ref{cond_CFPP:cluster_compatible_tau_n_no_cluster_from_Q}, as in the non periodic case, we can take advantage of the return time statistics for the induced Gibbs-Markov system. This time, the convergence of $\mu_Y(B_n)r_{B_n}^Y$ is not to an exponential law but to a law with distribution function $t\mapsto (1 -\theta) + \theta(1 - e^{-\theta t})$. However, starting from $Q(B_n)$ instead of $B_n$ implies that we do not have the cluster at $0$ and a convergence to the same law as for the hitting time, that is to say the law with distribution function $t\mapsto 1 - e^{-\theta t}$ \cite[Theorem 3.2]{Zwe22}. Thus we can  use the same decomposition as before to get 
\begin{align*}
    \mu_{Q(B_n)}(\lr_{B_n} < \tau_n) &\leq \mu_{Q(B_n)}(\mu_Y(B_n)\,\lr^Y_{B_n} < t) + \mu_{Q(B_n)}(\mu_Y(B_n)S_{\tau_n}\mathbf{1}_Y > t)
\end{align*}
for every $t > 0$. Fix $\varepsilon > 0$. By the 
convergence for the induced map, the first term will be smaller than $\varepsilon$ if $t$ is chosen small enough. For the second term, $S_{\tau_n}\mathbf{1}_Y \leq \tau_n$ and we already proved that $\tau_n \mu(B_n) \xrightarrow[n\to +\infty]{\mu_{B_n}} 0$ and thus, for all $t > 0$,
\begin{align*}
    \mu_{Q(B_n)}(\mu(B_n)\,\tau_n > t) & = \frac{1}{\mu(Q(B_n))} \int_{Q(B_n)} \mathbf{1}_{\{\mu(B_n)\,\tau_n > t\}}\,\dd\mu\\
    & \leq \frac{\mu(B_n)}{\mu(Q(B_n))}\, \mu_{B_n} (\mu(B_n)\tau_n > t) \xrightarrow[n\to +\infty]{} \theta^{-1}\times 0 = 0.
\end{align*}
For the remaining condition \ref{cond_CFPP:good_density_after_tau_n}, we can again use \cite[\S 4.7]{Aar97} (see also \cite[Section 10.2]{Zwe22}), to ensure the existence of some $\mathcal{U}$ compact (and convex) in $L^1(\mu)$ such that $\widehat{T^{\tau_n}}(\mathbf{1}_{Q(B_n)}/\mu(Q(B_n)) \in \mathcal{U}$ for every $n\geq 1$ because $Q(B_n)$ is a finite union of $\xi_{s_{n+q}}$-cylinders.
\end{proof}

\subsection{Proofs of Theorems \ref{thm:REPP_LSV_preimages_of_0} and \ref{thm:REPP_preimages_of_0_barely_infinite_case}}

\subsubsection{The case $x = 1/2$}

\begin{prop}
    \label{prop:HTS/RTS_1/2}
    Let $B_n := [0(\geq n)] = [1/2,\delta_n]$ be a sequence of nested balls shrinking towards $1/2$, where $\delta_n = (1+c_n)/2$. Then, 
    \begin{equation}
        \label{eq:HTS_right_neighborhood_1/2}
        \gamma(\mu(B_n))\,\lr_{B_n} \xRightarrow[n\to +\infty]{\mathcal{L}(\mu)} \mathfrak{J}_{\alpha}
    \end{equation}
    and 
    \begin{equation}
        \label{eq:RTS_right_neighborhood_1/2}
        \gamma(\mu(B_n))\,\lr_{B_n} \xRightarrow[n\to +\infty]{\mu_{B_n}} \mathfrak{\widetilde{J}}_{\alpha}\,.
    \end{equation}
\end{prop}

\begin{proof}[Proof (of Proposition \ref{prop:HTS/RTS_1/2})]
    Since $B_n$ is uniformly bounded away from $0$, it lies inside a uniform set for all $n$ and thus Corollary \ref{cor:equivalence_HTS/RTS_infinite_measure}-2) ensures that \eqref{eq:HTS_right_neighborhood_1/2} and \eqref{eq:RTS_right_neighborhood_1/2} are equivalent. Furthermore, with the strong distributional convergence \cite[Lemma 4.1]{RZ20}, \eqref{eq:HTS_right_neighborhood_1/2} is equivalent to \(\gamma(\mu(B_n))\,\lr_{B_n} \xRightarrow[n\to +\infty]{\mu_{[3/4,1]}} \mathfrak{J}_{\alpha}\). Notice that $T(B_n) = [0,c_n]$ is a right neighborhood of $0$ and that on $[3/4,1]$, we have $\lr_{B_n} = \lr_{T(B_n)} - 1$. Using \ref{cor:HTS_neighborhood_0_normalization_gamma}, we obtain 
    \begin{align*}
        \gamma\big(\mu(T_2^{-1}(T(B_n)))\big)\,\lr_{T(B_n)} \xRightarrow[n\to +\infty]{\mu_{[3/4,1]}} \mathfrak{J}_{\alpha}.
    \end{align*}
    Since $T_2^{-1}(T(B_n)) = B_n$, it concludes the proof.
\end{proof}

\begin{prop}
    \label{prop:result_REPP_point_1/2}
    Let $B_n := [1/2,\delta_n]$ be as in Proposition \ref{prop:HTS/RTS_1/2}. Then,
    \begin{equation}
        \label{eq:HREPP_1/2}
        N_{B_n}^{\gamma} \xRightarrow[n\to +\infty]{\mathcal{L}(\mu)} \DRPP(\mathfrak{J}_{\alpha},\widetilde{\mathfrak{J}}_{\alpha}),
    \end{equation}
    and 
        \begin{equation}
        \label{eq:RREPP_1/2}
        N_{B_n}^{\gamma} \xRightarrow[n\to +\infty]{\mu_{B_n}} \RPP(\widetilde{\mathfrak{J}}_{\alpha}).
    \end{equation}
\end{prop}

\begin{proof}[Proof (of Proposition \ref{prop:result_REPP_point_1/2})]
     To show the convergence, it is enough to prove that the stochastic process of the waiting times converges towards the stochastic process $(W_k)_{k\geq 1}$ of independent waiting times with $W_1 \eqlaw \mathfrak{J}_{\alpha}$ and $W_i \eqlaw \widetilde{\mathfrak{J}}_{\alpha}$ for $i\geq 2$ when the convergence is under $\mathcal{L}(\mu)$ and the i.i.d sequence of $\widetilde{\mathfrak{J}}_{\alpha}$ when the convergence is under $\mu_{B_n}$. This is enough to get the convergence of the point processes, as discussed in Remark \ref{rem:FPP_PhiFPP}. We will focus only on the return REPP as the proof is similar for the hitting REPP. It follows the argument used in \cite[(xii)-(xiii)]{PSZ13} and we proceed by induction on successive returns. The initialization is exactly Proposition \ref{prop:HTS/RTS_1/2}. Consider now $d\geq 1$. Our goal is to show that the distribution of the $(d+1)$-interarrival time is independent of the previous interarrivals and the law is the same. Let $t_0,\dots,t_{d-1}, t_d \in \mathbb{R}_+^*$ and let $M_n := B_n \cap \bigcap_{i= 0}^{d-1} \{ \gamma(\mu(B_n))\,r_{B_n}\circ T^{i}_{B_n} \leq t_i\}$. We have
    \begin{align*}
        \mu_{B_n}\left(M_n \cap \{\gamma(\mu(B_n))\,r_{B_n}\circ T_{B_n}^{d} \leq t_d\}\right) &= \frac{\mu\left(M_n \cap \{\gamma(\mu(B_n))\,r_{B_n}\circ T_{B_n}^{d} \leq t_d\}\right)}{\mu(B_n)} \\
        &= \mu_{B_n}(M_n) \,\mu_{M_n}(\gamma(\mu(B_n))\,r_{B_n}\circ T_{B_n}^{d} \leq t_d)\,.
    \end{align*}
    By assumption, we know the limit of $\mu_{B_n}(M_n)$. Hence, we need to prove the convergence of $\gamma(\mu(B_n))\,r_{B_n}\circ T_{B_n}^{d}$ under $\mu_{M_n}$. We have 
    \begin{align*}
        \mu_{M_n}\left(\gamma(\mu(B_n))\,r_{B_n}\circ T_{B_n}^{d} \leq t_d \right) & = \frac{\mu(B_n)}{\mu(M_n)} \int_{B_n} \mathbf{1}_{M_n} \mathbf{1}_{\left\{\gamma(\mu(B_n))\,r_{B_n} \leq t_d\right\}} \circ T_{B_n}^{d} \,\dd\mu_{B_n} \\
        & = \int_{B_n} \mu(B_n)\,\widehat{T}_{B_n}^d\left(\frac{\mathbf{1}_{M_n}}{\mu(M_n)}\right) \mathbf{1}_{\left\{\gamma(\mu(B_n))\,r_{B_n} \leq t_d\right\}}\,\dd\mu_{B_n}.
    \end{align*}
      Since $B_n = [0(\geq n)]$ is an interval and a union of cylinders build from $\xi$, the induced map is piecewise and a partition $\xi_{B_n}$ of $B_n$ can naturally be defined. On each element of $\xi_{B_n}$, $r_{B_n}$ is constant. Furthermore, for all $i\geq 1$, we can define $\xi_{B_n,i} = \bigvee_{k=0}^{i-1} T_{B_n}^{-k}\xi_{B_n}$. In particular, $M_n$ is $\xi_{B_n,d-1}$ measurable. We write $\kappa_n$ for the element of $\xi_{B_n,d-1}$ contained in $M_n$. On each $V \in \kappa_n$,
    there exists some $m_V$ such that $T^{d}_{B_n}|_V = T^{m_V} : V \to B_n$ is an homeomorphism and thus, $\widehat{T}^{m_V}\mathbf{1}_V \in C_r(B_n)$ by Lemma \ref{lem:Distortion_estimates_Thaler}. Hence,
    \begin{align*}
        \widehat{T}^{d}_{B_n} \left(\frac{\mathbf{1}_{M_n}}{\mu(M_n)}\right) & = \frac{1}{\mu(M_n)} \sum_{V\in \kappa_n} \widehat{T}^d_{B_n} \mathbf{1}_{V} = \frac{1}{\mu(M_n)} \sum_{V\in \kappa_n} \widehat{T}^{m_V} \mathbf{1}_{V} \in C_r(B_n).
    \end{align*}
    Now, writing $u_n := \mu(B_n)\,\widehat{T}_{B_n}^d\left(\frac{\mathbf{1}_{M_n}}{\mu(M_n)}\right)$ and remarking that $\int u_n\dd\mu = \mu(B_n)$, it gives, since $\text{supp}(u_n) \subset B_n$, that $\inf_{B_n} u_n \leq 1 \leq \sup_{B_n} u_n$. Since $u_n \in C_r(B_n)$ we have, for every $x,y\in B_n$, 
    \begin{align*}
        u_n(y)(1 - r\diam(B_n)) \leq u_n(x) \leq u_n(y)(1 + r\diam(B_n))\,.
    \end{align*}
    Hence, 
    \begin{align*}
        1 - r\diam(B_n)\leq \inf_{B_n} u_n \leq \sup_{B_n} u_n \leq 1 + r\diam(B_n)\,,
    \end{align*}
    which, since $\diam(B_n) \xrightarrow[n\to +\infty]{} 0$, implies
    \begin{align}
        \label{eq:asymptotic_renewal}
        \left\|\,\mu(B_n)\,\widehat{T}_{B_n}^d\left(\frac{\mathbf{1}_{M_n}}{\mu(M_n)}\right) - \mathbf{1}_{B_n}\right\|_{L^{\infty}(\mu_{B_n})} \xrightarrow[n\to +\infty]{} 0.
    \end{align}
      Thus
    \begin{align*}
        \mu_{M_n}\big(\gamma(\mu(B_n))\,\lr_{B_n} \circ T_{B_n}^d \leq t_d\big)\uset{\widesim}{n\to +\infty} \int \mathbf{1}_{\{\gamma(\mu(B_n))\,r_{B_n}\leq t_d\}}\,\dd\mu_{B_n} \xrightarrow[n\to +\infty]{} \mathbb{P}(\widetilde{\mathfrak{J}}_{\alpha} \leq t_d)\,,
    \end{align*}
    where we used again Proposition \ref{prop:HTS/RTS_1/2} to get the convergence and this concludes the proof.
\end{proof}

  In particular, Proposition \ref{prop:result_REPP_point_1/2} is Theorem \ref{thm:REPP_LSV_preimages_of_0} for $1/2$ or, otherwise stated, for $k = 0$. We will now capitalize on the convergence for $1/2$ and go backwards for the further preimages of $0$.

\subsubsection{Preimages of $1/2$}

  Fix $Y = [1/2,1]$ and for all $n\geq 1$, let $E_n = T_2^{-1}[0,c_n] = [0(\geq n)] = [1/2,\delta_n]$ where $\delta_n := (1 + c_n)/2$. In particular, with the notations chosen, we have for all $0\leq p\leq n$, $[0[p,n]] = [\delta_{n+1}, \delta_p] = Y \cap \{p+1\leq r_Y \leq n+1\}$. Fix $k\geq 1$ and let $B_n := T^{-k}E_n$. By invariance of $\mu$ we have $\mu(E_n) = \mu(B_n)$. By definition of the map $T$, for $n$ large enough, we have 
\begin{align*}
    B_n = \bigsqcup_{z \in T^{-k}\{1/2\}} B_{z,n},
\end{align*}
where $B_{z,n} = [z,z+\eta_{z,n}]$. With symbolic notations, we have $B_{z,n} = [z_0^{k-1}0(\geq n)]$, where $z = T^{-1}_{\sigma(a_0)}\cdots T^{-1}_{\sigma(a_{k-1})} 1/2$. So, to an element $z\in T^{-k}\{1/2\}$, we can associate a unique sequence $(z_0^{k-1})$ such that $(z_0^{k-1} 0)$ is admissible and reciprocally, for every admissible $(z_0^{k-1}0)$ we can associate $z \in T^{-k}\{1/2\}$. We define the localization map $\phi_n : B_n \to T^{-k}\{1/2\}$ by $\phi_n(x) = z$ if $x\in B_{z,n}$. In fact, if $n_0\geq 1$ is such that $(B_{z,n_0})_{z}$ are disjoint, then $\phi_n = \phi_{n_0}|_{B_n}$ for all $n\geq n_0$ and thus we can consider $\phi := \phi_{n_0}$. \\

  Let $\Psi^{\gamma,k}_n$ be a point process on $\mathbb{R} \times T^{-k}\{1/2\}$ defined by
\begin{align*}
    \Psi^{\gamma,k}_n := \sum_{j\geq 1} \delta_{(\gamma(\mu(B_n))r^{(j)}_{B_n},\, \phi \circ T_{B_n}^j)}. 
\end{align*}

\begin{thm}
    \label{thm:Marking_RREPP}
    Let $k\geq 1$ and $\Psi^{\gamma,k}_n$ be defined as previously, then we have 
    \begin{align*}
        \Psi^{\gamma,k}_n \xRightarrow[n\to +\infty]{\mu_{E_n}} P_k,
    \end{align*}
    where $P_k$ is an independent $\mathbb{Q}_k$-marking of $\RPP(\widetilde{\mathfrak{J}}_{\alpha})$.
\end{thm}
  Recall that $\mathbb{Q}_k$ is the probability on $T^{-k}\{1/2\}$ such that for $\rho$ the density of $\mu$ and all $z \in T^{-k}\{1/2\}$, 
\begin{align*}
        \mathbb{Q}_k(z) = \frac{\rho(z)}{\rho(1/2) (T^k)'(z)} \quad \forall z\in T^{-k}\{1/2\}.
    \end{align*}

\subsubsection{Proof of Theorem \ref{thm:Marking_RREPP}}

We start with the following easy observations.
\begin{lem}
    \label{lem:quotient_measures_B_nz}
    For all $k\geq 1$, $z\in T^{-k}\{1/2\}$ and all $d\geq 0$, we have $\mu_{B_n}(\phi \circ T_{B_n}^d = z) \sim \mathbb{Q}_k(z)$.
\end{lem}
\begin{proof}[Proof (of Lemma \ref{lem:quotient_measures_B_nz})]
    By invariance of the induced measure on $B_n$, for all $d\geq 0$, we have 
    \begin{align*}
        \mu_{B_n}(\phi \circ T_{B_n}^d = z) & = \mu_{B_n}(\phi = z) = \mu_{B_n}(B_{z,n})
         =  \frac{\mu(B_{z,n})}{\mu(E_n)}\\
         & \uset{\widesim}{n\to +\infty} \frac{\rho(z)\leb(B_{z,n})}{\rho(1/2)\leb(E_n)} \\
        & \uset{\widesim}{n\to +\infty} \frac{\rho(z)}{\rho(1/2)(T^k)'(z)}.
    \end{align*}
\end{proof}

\begin{lem}
    \label{lem:convergence_r_B_n_starting_from_E_n}
    We have 
    \begin{align*}
        \gamma(\mu(B_n))\,r_{B_n} \xRightarrow[n\to +\infty]{\mu_{E_n}} \widetilde{\mathfrak{J}}_{\alpha}\,.
    \end{align*}
\end{lem}

\begin{proof}[Proof (of Lemma \ref{lem:convergence_r_B_n_starting_from_E_n})]
    On $E_n$, by definition of $B_n$, we have $r_{B_n} = r_{E_n} - k$. Thus, for all $t > 0$ and by Proposition \ref{prop:HTS/RTS_1/2},
    \begin{align*}
        \mu_{E_n}(\gamma(\mu(B_n))\,r_{B_n} \leq t) &= \mu_{E_n}(\gamma(\mu(B_n))\,r_{E_n} \leq t) + \mu_{E_n}\left(t\leq \gamma(\mu(B_n))\,r_{E_n} \leq t + k\gamma(\mu(B_n))\right) \\ 
        &\xrightarrow[n\to +\infty]{} \mathbb{P}(\widetilde{\mathfrak{J}}_{\alpha} \leq t).
    \end{align*}
\end{proof}

  To establish Theorem \ref{thm:Marking_RREPP}, the strategy is to analyze the last visit to the interval \([0[p,n)] = [\delta_n, \delta_p]\) before the system returns to \(E_n\). Leveraging the bounded distortion property (Corollary \ref{cor:distortion_bounds_comparison_measures}), we obtain accurate estimates for the limiting process. Throughout this analysis, it is essential to ensure that the errors introduced by approximations diminish in the limit. This concern is addressed through a series of supporting lemmas and the pivotal Proposition \ref{prop:first_return_B_n_preimages_of_1/2}. To proceed systematically, we begin by introducing the necessary notations.

  Fix $k \geq 1$. Let $t > 0$ and take $p\geq 1$ large enough so that $[1/2,\delta_p] \cap \bigcup_{j= 1}^{k} T^{-j}\{1/2\} = \emptyset$. In particular, it implies that for all $z \in T^{-k}\{1/2\}$ and $0 \leq i \leq k-1$, $z_i < p$ (recall that $(z_0^{k-1})$ is the path in symbolic notations that corresponds to $z$ as a $k$-preimage of $1/2$). Let 
\begin{align*}
    D_n^p(t) := E_n \cap \left\{r_{[\delta_n,\delta_p]} \leq \frac{t}{\gamma(\mu(B_n))}\right\} \cap \left\{r_{[\delta_n,\delta_p]} \leq r_{E_n}\right\},
\end{align*}
  and let $\kappa_n^p(t)$ be a collection of branches belonging to $D_{n}^p(t)$ defined as follows
\begin{align*}
    \kappa_n^p(t) &:= \{ V \subset W \in \xi_m  \cap E_n \; \big|\; T^m(V) = [\delta_n,\delta_p],\; m \leq t/\gamma(\mu(B_n)) \; \text{and} \; \mu(T^jV \cap E_n) = 0 \; \forall 0 \leq j\leq m \} \\
    & = \{[0a_0^{m-2}0[p,n)] \;\big|\; 2\leq m\leq t/\gamma(\mu(B_n)), \; a_0\geq n, \; \text{and} \;\exists! \,0\leq j\leq m-2, \; a_j = n\}.
\end{align*}
  For $V \in \kappa_n^p(t)$, we write $m_V$ its associated $m$ coming from the definition of $V$ and $\kappa_n^p(t)$. Note that $\kappa_n^p(t)$ is not a partition of $D_n^p(t)$ as multiple returns to $[\delta_n,\delta_p]$ are possible but we can build a partition of the sets 
\begin{align*}
    D_{n,z}^p(t) = D_n^p(t) \cap \left\{ \max_{k\geq 1} \left\{r_{[\delta_n, \delta_p]}^{(k)}\;|\; r_{[\delta_n, \delta_p]}^{(k)} < r_{E_n}\right\} \leq \frac{t}{\gamma(\mu(B_n)}\right\} \cap \{r_{E_n}  = r_{B_{n,z}} + k\}, \quad z\in T^{-k}\{1/2\}.
\end{align*}
Indeed, for $V \in \kappa_{n}^p(t)$ and $z \in T^{-k}\{1/2\}$, define 
\[V_z := \left\{x \in V \;|\; r_{B_{n,z}}(T^{m_V}x) = r_{B_n}(T^{m_V}x) < r_{[\delta_n,\delta_p]}(T^{m_V}x)\right\}.\]
Then $\kappa_{n,z}^p(t) := \left\{ V_z,\; V \in \kappa_n^p(t)\right\}$
is a partition of $D_{n,z}^p(t)$. To be more precise, for all $V = [0a_0^{m_V-2}0[p,n)] \in \kappa_n^p(t)$, we have 
\begin{align*}
    V_z = \bigsqcup_{\underline{b} \in \mathcal{I}^p_z} [0a_0^{m_V-2}0\underline{b}z_0^{k-1}0(\geq n)]
\end{align*}
with $\mathcal{I}^p_z := \{\underline{b} := (b_0^{\ell-1}) \;|\; \ell\geq 1, \,(0b_0^{\ell-1}z_0) \;\text{admissible},\, b_0 \in [p,n) \;\text{and}\; \exists!\, 0\leq j < \ell, \; b_j = p\}$.\\

  The following lemma ensures that the measure of the image $T^{m_V}V_z$ will be comparable to the one of $B_{n,z}$ for all $V_z \in \kappa_{n,z}^p(t)$.
\begin{lem}
    \label{lem:measure_images_back_to_[p,n)_wrt_Bni}
    For all $n > p$, $z \in T^{-k}\{1/2\}$ and $V_z \in \kappa_{n,z}^p(t)$,
    \begin{align*}
        \mu(T^{m_V}(V_z)) = \mu(B_{n,z}) + O(\mu(B_n)^2)
    \end{align*}
      Furthermore, $(T^{m_V}V_z)_z$ are disjoint and for $V_1 := \bigsqcup_{z\in T^{-k}\{1/2\}} V_z$, we have 
    \begin{align*}
        \mu(T^{m_V} V_1) = \mu(B_n) + O(\mu(B_n)^2).
    \end{align*}
\end{lem}

\begin{proof}[Proof (of Lemma \ref{lem:measure_images_back_to_[p,n)_wrt_Bni})]
    For all $n > p$ and $z \in T^{-k}\{1/2\}$, we have by ergodicity and conservativity 
    \begin{align*}
        \mu(B_{n,z}) & = \sum_{k = 0}^{+\infty} \mu\left( [\delta_n,\delta_p] \cap \{r_{[\delta_n,\delta_p]} > k\} \cap T^{-k}B_{n,z}\right) \\
        & = \sum_{k = 0}^{+\infty} \sum_{\ell = k+1}^{+\infty} \mu\left( [\delta_n,\delta_p] \cap \{r_{[\delta_n,\delta_p]}  = \ell\} \cap T^{-k}B_{n,z}\right)\\
        & = \sum_{\ell = 1}^{+\infty} \sum_{k = 0}^{\ell - 1} \mu\left( [\delta_n,\delta_p] \cap \{r_{[\delta_n,\delta_p]}  = \ell\} \cap T^{-k}B_{n,z}\right)\,.
    \end{align*}

  On another hand, we have 
\begin{align*}
    \mu (T^{m_V}V_z) & = \mu \left([\delta_n,\delta_p] \cap \{r_{B_n} = r_{B_{n,z}} < r_{[\delta_{n},\delta_p]}\} \right) \\
    & = \mu\left([\delta_n,\delta_p] \cap \{r_{B_{n,z}} < r_{[\delta_n,\delta_p]}\}\right) - \mu\left([\delta_p,\delta_n] \cap \{r_{B_n} < r_{B_{n,z}} < r_{[\delta_p,\delta_n]}\} \right)\\
    & = \sum_{\ell = 1}^{+\infty} \mu\left( \bigcup_{k = 0}^{\ell - 1} [\delta_{n}, \delta_p] \cap \{r_{[\delta_n,\delta_p]} = \ell\} \cap T^{-k}B_{n,z}\right) - \mu\left([\delta_p,\delta_n] \cap \{r_{B_n} < r_{B_{n,z}} < r_{[\delta_p,\delta_n]}\} \right).
\end{align*}
We have 
\begin{align*}
    \mu\left([\delta_p,\delta_n] \cap \{r_{B_n} < r_{B_{n,z}} < r_{[\delta_p,\delta_n]}\} \right) & \leq \mu\left([\delta_p,\delta_n] \cap \{r^{(2)}_{B_n}< r_{[\delta_n,\delta_p]}\} \right)\,.
\end{align*}
Consider the partition of $[\delta_p,\delta_n] \cap \{r^{(2)}_{B_n}< r_{[\delta_p,\delta_n]}\} = [\delta_p,\delta_n] \cap \{r^{(2)}_{E_n}< r_{[\delta_p,\delta_n]}\}$ defined as follows 
\begin{align*}
    G_{n}^p & := \big\{ W \subset [\delta_p,\delta_n] \;|\;\exists q\geq 0,\; T^qW = [1/2,\delta_n],\\
    &\qquad \;\exists!\, 1\leq j < q ,\; T^jW \in [1/2,\delta_n],\; \forall 1\leq \ell \leq q, \; T^{\ell} W \cap [\delta_n,\delta_p] = \emptyset\big\} \\
    & = \big\{ [0b_0^{\ell-1}0c_0^{m-1}0(\geq n)] \;|\; \ell,m\geq 1,\; b_0 \in [p,n),\\
    &\qquad \; c_0 \geq n,\; \exists!\, 0\leq j <\ell \; b_j = p,\; \exists!\, 0\leq i < m \; c_i = p\big\}\,.
\end{align*}
  Then, for every $W = [0b_0^{\ell-1}0c_0^{m-1}0(\geq n)] \in G_n^p$, we can use Corollary \ref{cor:distortion_bounds_comparison_measures} to get 
\begin{align*}
    \mu(W) & = \mu [0b_0^{\ell-1}0c_0^{m-1}0(\geq n)]  \leq C\frac{\mu[0b_0^{\ell-1}0c_0^{m-1}0]\mu[0(\geq n)]}{\mu{[0]}} \\
    & \leq C^2\mu(B_n) \frac{\mu[0b_0^{\ell-1}0c_{i + p +1}^{m-1}0]\mu[0c_{0}^{m-1}0]}{\mu[0c_{i+p+1}^{m-1}0]} \\
    & \leq C^2\mu(B_n) \frac{\mu[0b_0^{\ell-1}0c_{i+p+1}^{m-1}0]}{\mu[0c_{i+p+1}^{m-1}0]} C \frac{\mu[0c_{i+p+1}^{m-1}0]\mu[0c_0^{i+p-1}0]}{\mu[0]} \\
    & \leq C^3 \mu(B_n) \mu[0b_0^{\ell-1}0c_{i+p+1}^{m-1}0] \mu{[0c_{0}^{i+p-1}0]}.
\end{align*}
Thus, summing on $W \in G_n^p$ we get 
\begin{align*}
    \mu\left([\delta_n,\delta_p] \cap \{r^{(2)}_{B_n}< r_{[\delta_n,\delta_p]}\} \right) & = \sum_{W \in G_{n}^p} \mu(W)\\
    & \leq \sum C^3\mu(B_n)\mu[0b_0^{\ell-1}0c_{i+p+1}^{m-1}0] \mu{[0c_{0}^{i+p-1}0]}\\
    & \leq C^3\mu(B_n)^2\mu([\delta_n,\delta_p]) = O(\mu(B_n)^2)\,.
\end{align*}
  Going back to the computation of $\mu(B_{n,z}) - \mu(T^{m_V}V_z)$, we have 
\begin{align*}
    \mu(B_{n,z}) - \mu(T^{m_V}V_z) & = \sum_{\ell = 1}^{+\infty}  \mu\left( \bigcup_{0\leq j < k \leq \ell - 1} [\delta_n,\delta_p] \cap \{r_{[\delta_n,\delta_p]} = \ell \} \cap T^{-j}B_{n,z} \cap T^{-k}B_{n,z}\right) + O(\mu(B_n)^2) \\
    & = \mu\left([\delta_n,\delta_p] \cap \{r^{(2)}_{B_{n,z}} < r_{[\delta_n,\delta_p]}\}\right) + O(\mu(B_n)^2) = O(\mu(B_n)^2),
\end{align*}
using the same argument and Corollary \ref{cor:distortion_bounds_comparison_measures} as before. \\

  Now, by definition and by the Markov property, we have for all $V = [0a_0^{m-2}0[p,n)] \in \kappa_{n}^p(t)$ 
\begin{align*}
    T^{m_V}V_z = \bigsqcup_{\underline{b}\in \mathcal{I}_z^p} [0\underline{b}z_0^{k-1}0(\geq n)].
\end{align*}
  It is easy to see that $(T^{m_V}V_z)_z$ are all disjoint and thus
\begin{align*}
    \mu(T^{m_V}V_1) &= \sum_{z\in T^{-k}\{1/2\}} \mu(T^{m_V} V_z) = \sum_{z\in T^{-k}\{1/2\}} \mu(B_{n,z}) + O(\mu(B_n)^2)\\
    & = \mu(B_n) + O(\mu(B_n)^2).
\end{align*}
\end{proof}

  Now, let $M_{n,z}(t) := E_n \cap \left\{r_{B_n} \leq t/\gamma(\mu(B_n)) \right\} \cap T_{B_n}^{-1}(B_{n,z})$ be the quantity we want to control. The next lemma ensures that we can approximate it well by $D_{n,z}^p(t)$.

\begin{lem}
    \label{lem:comparison_measure_real_set_and_up_to_last_passage}
    For all $t > 0$, $p\geq 1$ large enough so that $[1/2,\delta_p] \cap \bigcup_{j=1}^k T^{-j}\{1/2\}$ and $z \in T^{-k} \{1/2\}$ we have 
    \begin{align*}
        \mu\left(D_{n,z}^p(t) \,\triangle\, M_{n,z}(t) \right) +  o(\mu(B_n))\,.
    \end{align*}
\end{lem}

\begin{proof}[Proof (of Lemma \ref{lem:comparison_measure_real_set_and_up_to_last_passage})]
    We have the following partition $\eta_{n,z}(t)$ of our set $M_{n,z}(t)$
    \begin{align*}
        \eta_{n,z}(t) & := \bigg\{ V \subset W \in \xi_{m+1} \cap E_n \; \bigg|\; m\leq t/\gamma(\mu(B_n)),\\
        & \qquad \; T^mV = B_{n,z},\; \mu(T^kV \cap [1/2,\delta_n]) =  0 \; \forall 1\leq k < m\bigg\}\\
        & = \bigg\{[0a_0^{m-2}z_0^{k-1}0(\geq n)] \;\bigg| \; m\leq t/\gamma(\mu(B_n)),\\
        &\qquad \; a_0 \geq n, \; (a_0^{m-2}z_0) \; \text{admissible}\; \text{and}\; \exists!\, 0\leq j \leq m-2,\; a_j = n\bigg\}.
    \end{align*}
    On the other side, recall that
    \begin{align*}
        \kappa_{n,z}^p(t) =  \bigg\{&[0a_0^{m-2}0[p,n)] \cap T^{-m}\{r_{B_{n,z}} = r_{B_n} < r_{[0[p,n)]} \} \;\bigg|\\
        &\; m\leq t/\gamma(\mu(B_n)),\; a_0\geq n \; \text{and}\; \exists!\, 0\leq j\leq m-2, \; a_j = n \bigg\}\,.
    \end{align*}

      We start by controlling $M_{n,z}(t) \backslash D_{n,z}^p(t)$. For $V = [0a_0^{m-2}z_0^{k-1}0(\geq n)] \in \eta_{n,z}(t)$ such that $r_{[\delta_n,\delta_p]} < r_{B_n} = r_{B_{n,z}} = m$, let $j =  \max\{i \leq m \; |\; T^ix \in [\delta_n,\delta_p]\}$. Then $V \subset W = [0a_0^{j - 2}0[p,n)\} \cap T^{-j}\{r_{B_{n,z}} = r_{B_n} < r_{[\delta_n,\delta_p]}\} \in \kappa_{n,z}^p(t)$. Thus, it only remains the intervals for which $r_{[\delta_n,\delta_p]} > r_{B_n}$, \textit{i.e.}
    \begin{align*}
        M_{n,z}(t) \backslash D_{n,z}^p(t) \subset M_{n,z}(t) \cap \left\{r_{B_n} < r_{[\delta_n, \delta_p]}\right\} \subset E_n \cap \{r_{E_n} <  r_{[\delta_n, \delta_p]} \}\,.
    \end{align*}
    We already know that $\mu_{E_n}(r_{E_n} < r_{[\delta_n, \delta_p]}) = o(1)$, so
    \begin{align*}
        \mu\left(M_{n,z}(t) \backslash D_{n,z}^p(t)\right) = o(\mu(B_n)).
    \end{align*}

      Now, we need to control $D_{n,z}^p(t) \backslash M_{n,z}(t)$. This is more difficult as we allow for more branches in $D_{n,z}^p(t)$ because we do not impose any time control for returns to $B_n$. For every $V_z = [0a_0^{m-2}0[p,n)] \cap T^{-m}\{r_{B_{n,z}}  = r_{B_n}  < r_{[0[p,n)]} \} \in \kappa_{n,z}^p(t)$, we can build a partition $\kappa_{n,z}^{p,V_z}(t)$ of $V_z$ taking cylinders up to the return time to $B_{n,z}$.  
    \begin{align*}
        \kappa_{n,z}^{p,V_z}(t) := \bigg\{& [0a_0^{m-2}0b_0^{j-1}z_0^{k-1}0(\geq n)] \; \bigg|\; m\leq \frac{t}{\gamma(\mu(B_n))},\\
        &\; b_0 \in [p,n),\; (b_0^{j-1}z_0^{k-1}) \; \text{admissible},\; \exists !\, 0\leq \ell < j, \; b_{\ell} = p\bigg\}. 
    \end{align*}
       Observe that $W \in \kappa_{n,z}^{p,V_z}(t)$ is an interval included in $E_n$. Then, it gives another partition $\kappa_{n,z}^{p,ex}(t)$ of $D_{n,z}^p(t)$ by
     \begin{align*}
         \kappa_{n,z}^{p,ex}(t) := \bigcup_{V_z \in \kappa_{n,z}^p(t)} \kappa_{n,z}^{p,V_z}(t).
     \end{align*}
       Consider the subset $\kappa_{n,z}^{p,long}(t)$ of $\kappa_{n,z}^{p,ex}(t)$ defined by
    \begin{align*}
         \kappa_{n,z}^{p,long}(t) := \left\{ W = [0a_0^{m-2}0b_0^{j-1}z_0^{k-1}0(\geq n)] \in \kappa_{n,z}^{p,ex}(t)\; \bigg|\; m +j+1 > t/ \gamma(\mu(B_n))\right\}.
     \end{align*}
    By construction, we have 
    \begin{align*}
         D_{n,z}^p(t) \backslash M_{n,z}(t) =
         \bigsqcup_{W \in \kappa_{n,z}^{p,long}(t)}  W.
    \end{align*}
      For every sequence $(y_n)_n$ going to $+\infty$ such that $y_n = o(\gamma(\mu(B_n))^{-1})$, we have 
    \begin{align*}
        \mu(D_{n,z}^p(t) \backslash M_{n,z}(t)) & = \sum_{W \in \kappa_{n,z}^{p,long}(t), \; r_{E_n} - m > y_n} \mu(W) + \sum_{W \in \kappa_{n,z}^{p,long}(t), \; r_{E_n} - m \leq y_n} \mu(W) \\
        & \leq \sum_{\underset{r_{E_n} - m > y_n}{W \in \kappa_{n,z}^{p,long}(t),}} \mu[0a_0^{m-2}0b_0^{j-1}z_0^{k-1}0(\geq n)] \\
        &\qquad \qquad + \mu\big(E_n \cap \{t/\gamma(\mu(B_n)) \leq r_{E_n} \leq t/\gamma(\mu(B_n)) + y_n\}\big) \\
        & \leq C\sum_{\underset{r_{E_n} - m > y_n}{W \in \kappa_{n,z}^{p,long}(t),}} \frac{\mu[0a_0^{m-2}0(\geq n)] \mu[0b_0^{j-1}z_0^{k-1}0(\geq n)]}{\mu[0(\geq n)]} + o(\mu(B_n))\,,
    \end{align*}
    where we used Corollary \ref{cor:distortion_bounds_comparison_measures} and Proposition \ref{prop:HTS/RTS_1/2}. Again with Corollary \ref{cor:distortion_bounds_comparison_measures}, we get 
    \begin{align*}
        \mu(D_{n,z}^p(t) \backslash M_{n,z}(t)) & \leq C^2\sum_{\underset{r_{E_n} - m > y_n}{W \in \kappa_{n,z}^{p,long}(t),}} \frac{\mu[0a_0^{m-2}0(\geq n)]}{\mu[0(\geq n)]} \frac{\mu[0b_0^{j-1}z_0^{k-1}0]\mu[0(\geq n)]}{\mu[0]} + o(\mu(B_n))\\
        & \leq C^2 \sum_{\underset{r_{E_n} - m > y_n}{W \in \kappa_{n,z}^{p,long}(t),}} \mu[0a_0^{m-2}0(\geq n)]\mu[0b_0^{j-1}z_0^{k-1}0] + o(\mu(B_n)).
    \end{align*}
      At this point, we can split the sum into two sums on the possible $a_0^{m-2}$ and the possible $b_0^{j-1}$. $(b_0^{j-1})$ is such that $b_0 \in [p,n)$, $(b_0^{j-1}z_0^{k-1})$ is admissible and $\exists! \, 0\leq \ell < j, \; b_{\ell} = p$. Since $r_{E_n} - m > y_n$, this implies 
    \begin{align*}
        \sum_{b_0^{j-1}} \mu[0b_0^{j-1}z_0^{k-1}0] \leq \mu\left([\delta_n,\delta_p] \cap \{r_{[\delta_n,\delta_p]} > y_n\}\right).
    \end{align*}
    Yet, for all $n> p$, $\mu[\delta_n, \delta_p] > \mu([\delta_{p+1},\delta_p]) > 0$ and thus \cite[Proposition 2.1 b)]{RZ20} ensures that $\mu([\delta_n,\delta_p] \cap \{r_{[\delta_n,\delta_p]} > y_n\}) = o(1)$ as $y_n \xrightarrow[n\to +\infty]{} +\infty$.\\

      On the other side, we have $(a_0^{m-2})$ that must be such that $(a_0^{m-2}0)$ is admissible, $m\leq t/\gamma(\mu(B_n))$, $a_0 \geq n$ and $\exists!\, 0\leq \ell < m-1$, $a_{\ell} = n$. Thus,
    \begin{align*}
        \sum_{a_0^{m-2}} \mu[0a_0^{m-2}0(\geq n)] \leq \mu\left(E_n \cap \left\{r_{E_n} \leq \frac{t}{\gamma(\mu(B_n))}\right\} \right).
    \end{align*}
    By Proposition \ref{prop:HTS/RTS_1/2}, since $t > 0$, we know that
    \begin{align*}
        \mu \left(E_n \cap \left\{r_{E_n} \leq \frac{t}{\gamma(\mu(B_n))}\right\} \right) \lesssim \mu(E_n) = \mu(B_n).
    \end{align*}
    Hence, we have
    \begin{align*}
        \mu(D_{n,z}^p(t) \backslash M_{n,z}(t)) &\lesssim C^2 \mu([\delta_n,\delta_p] \cap \{r_{[\delta_n,\delta_p]} > y_n\}\mu(B_n) + o(\mu(B_n))\\
        & = o(\mu(B_n))
    \end{align*}
    and it concludes the proof of the lemma.
\end{proof}

  Let $M_n(t) := E_n \cap \{\gamma(\mu(B_n))\,r_{B_n} \leq t\} = \bigsqcup_{z\in T^{-k}\{1/2\}} M_{n,z}(t)$.

\begin{prop}
    \label{prop:first_return_B_n_preimages_of_1/2}
    For $z \in T^{-k}\{1/2\}$ and all $t > 0$, we have 
    \begin{align*}
        \mu_{M_n(t)}\left( T_{B_n}^{-1}B_{n,z}\right) = \frac{\mu(M_{n,z}(t))}{\mu(M_n(t))} \xrightarrow[n\to +\infty]{} \mathbb{Q}_k(z)\,.
    \end{align*}
\end{prop}

\begin{proof}[Proof (of Proposition \ref{prop:first_return_B_n_preimages_of_1/2})]
    For every $V\in \kappa_n^p(t)$, let $V_1 := \bigsqcup_{z\in T^{-k}\{1/2\}} V_z$. By Lemma \ref{lem:comparison_measure_real_set_and_up_to_last_passage}, for all $p \geq 1$  large enough, $n\geq p$ and $z \in T^{-k}\{1/2\}$, we have 
    \begin{align*}
        \mu(M_{n,z}(t)) + o(\mu(B_n)) & = \mu(D_{n,z}^p(t)) = \sum_{V \in \kappa_{n}^p(t)} \mu(V_z) \\
        & = \left(1 \pm C \diam([\delta_{n},\delta_p]) \right) \sum_{V \in \kappa_{n}^p(t)} \mu(V_{1})\frac{\mu(T^{m_V}V_z)}{\mu(T^{m_V}V_{1})},
    \end{align*}
    where we use the bounded distortion Corollary \ref{cor:distortion_bounds_comparison_measures} for the last inequality. Hence, using Lemma \ref{lem:measure_images_back_to_[p,n)_wrt_Bni}, for all $z \in T^{-k}\{1/2\}$ 
    \begin{align*}
        \mu(M_{n,z}(t)) + o(\mu(B_n)) & = \left(1 \pm C \diam([\delta_{n},\delta_p]) \right) \sum_{V \in \kappa_{n}^p(t)} \mu(V_1) \frac{\mu(B_{n,z}) + O(\mu(B_n)^2)}{\mu(B_{n}) + O(\mu(B_n)^2)}
    \end{align*}
    and 
    \begin{align*}
        \sum_{V \in \kappa_{n}^p(t)} \mu(V_1) & = \sum_{z'\in T^{-k}\{1/2\}}\sum_{V_{z'}\in \kappa_{n,z'}^p(t)}\mu(V_{z'}) \\
        & = \sum_{z'\in T^{-k}\{1/2\}}\mu(D_{n,z'}^p(t)) \\
        & = \sum_{z'\in T^{-k}\{1/2\}}\left(\mu(M_{n,z'}(t)) + o(\mu(B_n))\right)\\
        & = \left(\mu(M_n(t)) + o(\mu(B_n))\right)\,.
    \end{align*}
      Thus, with Lemma \ref{lem:quotient_measures_B_nz}, for all $p$ large enough, we have 
    \begin{align*}
        \mu(M_{n,z}(t))/\mu(M_n(t)) + o(1) = \left(1 \pm C \diam([\delta_{n},\delta_p])\right) \frac{\mu(B_{n,z}) + O(\mu(B_n)^2)}{\mu(B_{n}) + O(\mu(B_n)^2)} (1 + o(1))
    \end{align*}
    using Lemma \ref{lem:convergence_r_B_n_starting_from_E_n} to ensure that $o(\mu(B_n)) = o(\mu(M_n(t)))$. \\
    Hence, for all $t> 0$, we have 
    \begin{align*}
        \limsup_{n\to +\infty} \frac{M_{n,z}(t)}{M_n(t)} \leq (1 + C\diam([1/2,\delta_p])) \lim_{n\to +\infty}\frac{\mu(B_{n,z})}{\mu(B_n)} \leq (1 + C\diam([1/2,\delta_p])) \mathbb{Q}_k(z).
    \end{align*}
    We have a similar control for the $\liminf$. Since it was true for $p$ large enough, taking $p\to +\infty$ concludes the proof of Proposition \ref{prop:first_return_B_n_preimages_of_1/2}.
\end{proof}

Building up on all these lemmas, we are now able to prove Theorem \ref{thm:Marking_RREPP}.

\begin{proof}[Proof (of Theorem \ref{thm:Marking_RREPP})]
    To get the convergence of the marking process, it is enough to get the convergence of the finite dimensional marginal of the stochastic process
    \begin{align*}
        \left((\gamma(\mu(B_n)))\,r_{B_n} \circ T^i_{B_n}, \phi \circ T^{i+1}_{B_n})\right)_{i\geq 0}.
    \end{align*}
    Let $d\geq 1$. We want to show
    \begin{align}   \label{eq:convergence_first_d_returns_from_mu_E_n}
        \left((\gamma(\mu(B_n)))\,r_{B_n} \circ T^i_{B_n}, \phi \circ T^{i+1}_{B_n})\right)_{0\leq i \leq d} \xRightarrow[n\to +\infty]{\mu_{E_n}} \left((\widetilde{\mathfrak{J}}_{\alpha}^{(i)}, Y^{(i)})\right)_{0\leq i\leq d},
    \end{align}
    where $\left((\widetilde{\mathfrak{J}}_{\alpha}^{(i)}, Y^{(i)})\right)_{0\leq i\leq d}$ are i.i.d, $\widetilde{\mathfrak{J}}_{\alpha}^{(1)} \eqlaw\widetilde{\mathfrak{J}}_{\alpha}$,  $Y^{(1)} \eqlaw \mathbb{Q}_k$ and $\widetilde{\mathfrak{J}}_{\alpha}^{(1)}$ and $Y^{(1)}$ are independent.

      As we did for the point $1/2$ in Proposition \ref{prop:result_REPP_point_1/2} and building up again on \cite[(xii)-(xiii)]{PSZ13}, we prove the result by induction.  Proposition \ref{prop:first_return_B_n_preimages_of_1/2} together with Lemma \ref{lem:convergence_r_B_n_starting_from_E_n} gives 
\begin{align}
    \label{eq:first_return_marking_process}
    (\gamma(\mu(B_n))\,r_{B_n}, \phi \circ T_{B_n}) \xRightarrow[n \to +\infty]{\mu_{E_n}} (\widetilde{\mathfrak{J}}_{\alpha}, Y)\,,
\end{align}
where $Y \eqlaw \mathbb{Q}$ and $\widetilde{\mathfrak{J}}_{\alpha}$ and $Y$ are independent, proving the result for $d = 0$. 

  Now, we do the inductive step from $d$ to $d+1$. By invariance of the measure, for all $k \geq 1$, we have 
\begin{align*}
    (\gamma(\mu(B_n))\,r_{B_n} \circ T_{E_n}^k, \phi \circ T_{B_n}^{k+1}) \xRightarrow[n \to +\infty]{\mu_{E_n}} (\widetilde{\mathfrak{J}}_{\alpha}, Y), 
\end{align*}
  using that $T_{B_n} \circ T_{E_n}^i = T_{B_n}^{i+1}$ on $E_n$ for all $i\geq 0$. Let $t_0, \dots, t_d > 0 $ and $z^{(0)},\dots,z^{(d)} \in T^{-k}\{1/2\}$. Let 
    \begin{align*}
        M_n := E_n \cap \bigcap_{i = 0}^d \{ \gamma(\mu(B_n))\,r_{B_n}\circ T_{E_n}^i \leq t_i\} \cap \{\phi \circ T_{B_n}^{i+1} = z^{(i)}\}.
    \end{align*}
      Let $t > 0$ and $z^{(d+1)} \in T^{-k}\{1/2\}$. We show 
    \begin{align*}
        \label{eq:recurrence_T_E_n}
        \mu_{M_n}\left( \{\gamma(\mu(B_n))\,r_{B_n} \circ T_{E_n}^{d+1} \leq t\} \cap \{\phi \circ T_{B_n}^{d+2} = z^{(d+1)}\}\right) \xrightarrow[n\to +\infty]{} \mathbb{P}(\widetilde{\mathfrak{J}}_{\alpha} \leq t) \mathbb{Q}_{k}(z^{(d+1)}). 
    \end{align*}
    Because $E_n$ is an union of $1$-cylinders, we can give a partition $ \eta_n$ in cylinders of $M_n$ and a partition $\eta'_n$ of $N_n := M_n \cap  \{\gamma(\mu(B_n))\,r_{B_n} \circ T_{E_n}^{d+1} \leq t\} \cap \{\phi \circ T_{B_n}^{d+2} = z^{(d+1)}\}$. Furthermore, associated to the partition $\eta_n$, we set $\mathcal{I}_n$ such that
    \begin{align*}
        \eta_n = \{V := [0a_0^{m-1}0(\geq n)] \; |\; (a_0^{m-1}) \in \mathcal{I}_n \}. 
    \end{align*}
    This yields
    \begin{align*}
        \eta'_n = \bigg\{& [0a_0^{m-1}0b_0^{j-1}0(\geq n)]\; \bigg|\; a_0^{m-1} \in \mathcal{I}_n,\; b_0 \geq n,\\
        & \exists! i\; b_i = n,\; j\leq t/\gamma(\mu(B_n)),\; b_{j - 1 - k}^{j-1} = (z^{(d+1)})_0^{k-1}\bigg\}  \,.
    \end{align*}
    Thus,
    \begin{align*}
        \mu(N_n) &= \sum_{[0a_0^{m-1}0b_0^{j-1}0(\geq n)] \in \eta'_n} \mu[0a_0^{m-1}0b_0^{j-1}0(\geq n)] \\
        & = (1 \pm C \diam(E_n)) \sum_{[0a_0^{m-1}0b_0^{j-1}0(\geq n)] \in \eta'_n} \mu[0a_0^{m-1}0(\geq n)] \frac{\mu[0b_0^{j-1}0(\geq n)]}{\mu(E_n)} \\
        & = (1 \pm C\diam(E_n)) \mu_{E_n}( \{\gamma(\mu(B_n))\,r_{B_n} \leq t\} \cap \{\phi \circ T_{B_n} = z^{(d+1)}\}) \mu(M_n)\,,
    \end{align*}
    where we used Corollary \ref{cor:distortion_bounds_comparison_measures}, meaning that 
    \begin{align*}
        \mu_{M_n}(N_n) \xrightarrow[n\to +\infty]{} \mathbb{P}(\widetilde{\mathfrak{J}}_{\alpha} \leq t) \mathbb{Q}_{k}(z^{(d+1)}).
    \end{align*}
    Since $r_{B_n} \circ T_{E_n}^i = r_{B_n} \circ T_{B_n}^i - k$ for $i\geq 1$ on $E_n$ and $\gamma(\mu(B_n))k \xrightarrow[n\to +\infty]{} 0$, it is sufficient to get \eqref{eq:convergence_first_d_returns_from_mu_E_n} and it concludes the proof of Theorem \ref{thm:Marking_RREPP}.
\end{proof}

\subsubsection{Proof of Theorems \ref{thm:REPP_LSV_preimages_of_0} and \ref{thm:REPP_preimages_of_0_barely_infinite_case}}

\begin{proof}[Proof (of Theorem \ref{thm:REPP_LSV_preimages_of_0})]
    By Theorem \ref{thm:Marking_RREPP}, we have 
    \begin{align*}
        \psi_n^{\gamma, k} = \sum_{j\geq 1} \delta_{(\gamma(\mu(B_n))\,r_{B_n}^{(j)}, \phi \circ T_{B_n}^{j+1})} \xRightarrow[n\to +\infty]{\mu_{E_n}} P_k.
    \end{align*}
    Furthermore, for all $z \in T^{-k}\{1/2\}$, we have 
    \begin{align*}
        \left\|\widehat{T}^k\left(\frac{\mathbf{1}_{B_{n,z}}}{\mu(B_{n,z})}\right) - \frac{1}{\mu(E_n)}\mathbf{1}_{E_n} \right\|_{L^{\infty}(\mu_{E_n})} \xrightarrow[n\to +\infty]{} 0.
    \end{align*}
    Together with $k\gamma(\mu(B_n)) \xrightarrow[n\to +\infty]{} 0$ and $r_{B_n}\circ T^k = r_{B_n} - k$ on $B_n$, this is enough to ensure 
    \begin{align*}
        \Psi_{n}^{\gamma,k} \xRightarrow[n\to +\infty]{\mu_{B_{n,z}}} P_k.
    \end{align*}

      Thus, we get
    \begin{align*}
        \sum_{j\geq 1} \delta_{\gamma(\mu(B_n))\,r_{B_{n,z}}^{(j)}} = \sum_{j\geq 1} \delta_{(\gamma(\mu(B_n))\,r_{B_n}^{(j)}, \phi \circ T^{j+1}_{B_n})} \mathbf{1}_{\{\phi \circ T_{B_n}^{j+1} = z\}} \xRightarrow[n\to +\infty]{\mu_{B_{n,z}}} P_k(\cdot \times \{z\}) \eqlaw \RPP(\widetilde{\mathfrak{J}}_{\alpha})^{(\mathbb{Q}_k(z),1)}
    \end{align*}
    Finally, using that 
    \begin{align*}
        \lim_{n\to+\infty} \gamma(\mu(B_{n,z}))/\gamma(\mu(B_n)) = \lim_{n\to +\infty} \left(\mu(B_{n,z})/\mu(B_n)\right)^{1/\alpha} = \mathbb{Q}_k(z)^{1/\alpha}, 
    \end{align*}
    we get, for every $k\geq 1$ and $z\in T^{-k}\{1/2\}$, 
    \begin{align*}
        N_{B_{n,z}}^{\gamma} \xRightarrow[n\to +\infty]{\mu_{B_{n,z}}} \RPP(\widetilde{\mathfrak{J}}_{\alpha})^{(\mathbb{Q}_{k}(z), \mathbb{Q}_{k}(z)^{1/\alpha})}.
    \end{align*}
Therefore, by Corollary \ref{cor:equivalence_HTS_RTS_true_Point_Processes}, it also gives
    \begin{align*}
        N_{B_{n,z}}^{\gamma} \xRightarrow[n\to +\infty]{\mathcal{L}(\mu)} \DRPP(\mathfrak{J}_{\alpha}, \widetilde{\mathfrak{J}}_{\alpha})^{(\mathbb{Q}_{k}(z), \mathbb{Q}_{k}(z)^{1/\alpha})}.
    \end{align*}
\end{proof}

\begin{proof}[Proof (of Theorem \ref{thm:REPP_preimages_of_0_barely_infinite_case})]
    The proof of Theorem \ref{thm:REPP_preimages_of_0_barely_infinite_case} is similar to the proof of Theorem \ref{thm:REPP_LSV_preimages_of_0}. The laws $\mathfrak{J}_{\alpha}$ and $\widetilde{\mathfrak{J}}_{\alpha}$ must be changed to the exponential law $\mathcal{E}$ but, apart from that, the proof stays identical and especially within the key Theorem \ref{thm:Marking_RREPP}. Indeed, from Theorem \ref{thm:Collet_Galves_p_equals_1} and Lemma \ref{lem:comparison_renormalizations_for_0_p=1}, the same proof as for Proposition \ref{prop:result_REPP_point_1/2}, gives this time a renewal point process of exponential waiting times, \textit{i.e.} the homogeneous Poisson process. Building up on it, an equivalent of Theorem \ref{thm:Marking_RREPP} can be proven the same way. Finally, since the Poisson point process is stable through thinning and rescaling by the same parameter, \textit{i.e.} if $N$ is a Poisson point process and $\tau > 0$, $N^{(\tau, \tau)}$ is again a Poisson process, and since $\alpha = 1$, the limit for the preimages of $1/2$ are again Poisson point processes. 
\end{proof}

\subsubsection{Proof of Proposition \ref{prop:convergence_to_FPP_when_going_further_from_0}}

\begin{proof}[Proof (of Proposition \ref{prop:convergence_to_FPP_when_going_further_from_0})]
    Recall the formula \eqref{eq:laplace_transform_J_alpha_tilde} for the Laplace transform of the law $\widetilde{\mathfrak{J}}_{\alpha}$.
    \begin{align*}
        \mathbb{E}\left[e^{-s\widetilde{\mathfrak{J}}_{\alpha}}\right] = 1 - \frac{s^{\alpha}}{\Gamma(1+\alpha)} \left(e^{-sd_{\alpha}} + sd_{\alpha}\int_0^{1} y^{-\alpha}e^{-d_{\alpha}sy}\,\dd y\right)^{-1}, \quad s\geq 0.
    \end{align*}
    In particular, we have
    \begin{align*}
        \mathbb{E}\left[e^{-s\widetilde{\mathfrak{J}}_{\alpha}}\right] = 1 - \frac{s^{\alpha}}{\Gamma(1+\alpha)} + o(s^{\alpha}).
    \end{align*}
    However, for a renewal process $\RPP(X)$ with $\mathbb{E}[e^{-sX}] = 1 - \lambda s^{\alpha} + o(s^{\alpha})$, we have 
    \begin{align*}
        \RPP(X)^{(\tau,\tau^{\alpha})} \xRightarrow[\tau \to 0]{} \fPp_{\alpha}(\lambda^{-1}).
    \end{align*}
    This can be found for example in \cite[\S 10.4.4 p.355]{Gorenflo20_book_Mittag-Leffler}.
      Hence, it remains to show that for every sequence $(x_k)_{k\geq 0}$ with $x_k \in T^{-k}\{1/2\}$, $\mathbb{Q}_k(x_k) \xrightarrow[k\to +\infty]{} 0$ which is equivalent to $\rho(x_k)/ (T^k)'(x_k) \xrightarrow[k\to +\infty]{} 0$. Since on every interval on the form $[\varepsilon, 1]$, $\rho$ is bounded away from $+\infty$ and $T' \geq c(\varepsilon) > 1$, this is enough to show it when $x_k \xrightarrow[k\to +\infty]{} 0$. We consider the case $x_k = T_1^{-k}1/2 = c_{k+1}$, the other cases can be considered similarly. Since the return map $Y$ is full-branched Gibbs-Markov, for each branch $T^{k+1} : [\delta_{k+1}, \delta_k]$ is such that there exists $y \in [\delta_{k+1}, \delta_k]$ with $(T^{k+1})'(y) = (2(\delta_k -\delta_{k+1}))^{-1} = (c_k - c_{k+1})^{-1}$ by the mean value theorem. By bounded distortion, it yields $(T^{k+1})'(\delta_{k+1}) \asymp (c_k - c_{k+1})^{-1}$ and since $T'(\delta_{k+1}) = 1/2$, we obtain $(T^{k})'(c_{k+1}) \asymp (c_k - c_{k+1})^{-1}$. On the other side, by \eqref{eq:formula_density_LSV_map}, we have $\rho(c_k) \asymp c_{k+1} (c_{k+1} - c_{k+2})^{-1}$ and implying $\rho(c_k)/(T^k)'(c_k) \xrightarrow[k\to +\infty]{} 0$.  
    \end{proof}

\section{Fractional Poisson processes in infinite ergodic theory: a concluding perspective}
\label{section:discussion_Pene_Saussol}

  In the recent article \cite{PS23} (\cite{Yas18} for the first return), the authors showed a convergence for a recurrence REPP in the infinite measure preserving context. They consider a $\mathbb{Z}$-extension over a one-sided subshift of finite type, defined as follows: let $(\Omega, \sigma, \nu)$ be a topologically mixing subshift of finite type endowed with a Gibbs measure $\nu$ and let $h : \Omega \to \mathbb{Z}$ be a centered integrable Hölder observable. On the phase space $X := \Omega \times \mathbb{Z}$, we define the dynamics $T$ by
\begin{align*}
    T : (x, a) \mapsto (\sigma(x), a + h(x)), \quad x\in \Omega,\; a\in \mathbb{Z}.
\end{align*}
The $\sigma$-finite measure $\mu := \nu \otimes \left(\sum_{z\in \mathbb{Z}} \delta_z \right)$ is preserved by $T$. 

\begin{thm}\textup{\cite[Theorem 3.10]{PS23}}
    \label{thm:proof_Z_ext_PS23}
    For all $r> 0$, let $N_r$ be the point process defined as follows
    \begin{align*}
        N_r(x) := \sum_{k \geq 1} \delta_{\nu(B(x,r))^2r^{(k)}_{B(x,r)\times \{0\}}} \quad \forall x\in \mathcal{X}.
    \end{align*}
    Then, for the vague convergence, 
    \begin{equation*}
        N_r \xRightarrow[r\to 0]{\mu} N \circ L_0 \quad \text{and} \quad N_r \xRightarrow[r\to 0]{\mu_{B(x,r) \times \{0\}}} \widetilde{N} \circ \widetilde{L}_0,
    \end{equation*}
    where $N, \widetilde{N}$ are standard Poisson process and $L_0, \widetilde{L}_0$ is are local times at $0$ of the standard Brownian motion, independent of $N, \widetilde{N}$.
\end{thm}

  Here, the target set depends on the point $x$ considered. However, the techniques developed could have also given the convergence of the hitting REPP for generic points $x$ and shrinking balls $B_n = B(x,\eta_n)$ for some $\eta_n \xrightarrow[n\to+\infty]{} 0$. We justify that both results are linked and we think combining the two approaches could provide fruitful insights for future developments. \\

  For simplicity, consider the subshift of finite type on $\Omega := \{-b,\dots, b\}^{\mathbb{N}}$ for some $b > 0$, $\nu = m^{\otimes \mathbb{N}}$ a Bernoulli measure on $\Omega$ with $m$ a probability of mean $0$ and variance $\sigma^2 = 1$ and $h := \text{pr}_0$ the projection onto the first coordinate. Then the $\mathbb{Z}$-extension $(X,T)$ of $(\Omega, \sigma)$ models the usual random walk on $\mathbb{Z}$ and we assume furthermore that it is aperiodic. By inducing on $Y := \Omega \times \{0\}$ and using \cite[Lemma 3.7.4]{Aar97}, it is possible to show that the system is PDE. Furthermore, using classical estimates for random walks, we easily an equivalent of the wandering rate $w_n(Y) := \mu(Y)\mathbb{E}_{\mu_Y}[r_Y \wedge n] = \sum_{k=0}^{n-1}\mu(Y \cap \{r_Y > n\})$ when $n \to +\infty$. 

\begin{lem}
    \label{lem:wandering_rate_Z_extension_RW}
    For $Y := \Omega \times \{0\}$, we have 
    \begin{equation*}
       w_n(Y) \uset{\widesim}{n\to +\infty} \frac{2\sqrt{2}}{\sqrt{\pi}}\sqrt{n}. 
    \end{equation*}
\end{lem}

\begin{proof}[Proof (of Lemma \ref{lem:wandering_rate_Z_extension_RW})]
    Using classical local limit theorems for random walks, we know that $\mu_Y(r_Y > k) = \mathbb{P}_0(r_0 > k) \uset{\widesim}{n\to +\infty} \sqrt{\frac{2}{\pi k}}$. Since it is not summable, we get
    \begin{align*}
        w_n(Y) = \mu(Y) \sum_{k=0}^{n-1} \mu_Y(r_Y > k) \uset{\widesim}{n\to +\infty} \sum_{k=0}^{n-1} \sqrt{\frac{2}{\pi k}} \uset{\widesim}{n\to +\infty} \frac{2\sqrt{2}}{\sqrt{\pi}} \sqrt{n}.
    \end{align*}
\end{proof}

  Then, by  \cite[Proposition 3.8.7 p.137]{Aar97}, we can identify the normalizing sequence $(a_n)_{n\geq 0}$ as we have 
    \begin{align*}
        a_n \uset{\widesim}{n\to +\infty} \frac{1}{\Gamma(2 - \alpha)\Gamma(1+\alpha)}\frac{n}{w_n(Y)}.
    \end{align*}
    Since $w_n(Y) \in \RV(1/2)$ we have $\alpha = 1/2$ and thus, using Lemma \ref{lem:wandering_rate_Z_extension_RW}, we have  
    \begin{align*}
        a_n \uset{\widesim}{n\to +\infty} \frac{\sqrt{\pi}}{2\sqrt{2}\Gamma(3/2)^2}\sqrt{n} = \frac{4}{\pi}\frac{\sqrt{\pi}}{2\sqrt{2}}\sqrt{n} = \sqrt{\frac{2}{\pi n}},
    \end{align*}
    using that $\Gamma(3/2) = \Gamma(1/2)/2$ and $\Gamma(1/2)^2 = \pi$.
      Now, by definition \eqref{eq:definition_of_gamma} of $\gamma$, $\gamma : s \mapsto 2s^2/\pi$ is a suitable scaling. Thus, Theorem \ref{thm:proof_Z_ext_PS23} implies the following, for $x$ generic in $\Omega \times \{0\}$ and $B_n := B(x,r_n) \times \{0\}$ with $r_n \xrightarrow[n\to+\infty]{} 0$,

\begin{align*}
    \sum_{k \geq 1} \delta_{\gamma(\mu(B_n))\,\mr_{B_n}^{(k)}} & = \sum_{k\geq 1} \delta_{\frac{2}{\pi}\mu(B_n)^2\,\mr_{B_n}^{(k)}} \xRightarrow[n\to +\infty]{\mu} N \circ L_0\left(\frac{2}{\pi} \;\cdot\right). 
\end{align*}

  It turns out this point process is in fact equal to the limit process we expect with the methods developed in the previous sections.

\begin{lem}
    \label{lem:equivalence_PS23_article}
    Let $N$ be a standard Poisson point process and $L_0$ be the local time at $0$ of a standard Brownian motion. Then,
    \begin{align*}
        N \circ L_0\left(\frac{2}{\pi}\;\cdot \right) \eqlaw \fPp_{1/2}(\Gamma(1+1/2)).
    \end{align*}
\end{lem}
  
In Definition \ref{defn:Fractional_Poisson_Process}, the fractional Poisson process is introduced as a specific type of renewal point process. However, this is not the only perspective; alternative definitions have been shown to be equivalent. One such approach characterizes it as a classical Poisson process where the time parameter is randomly rescaled, as we now explain.
Let $\alpha \in (0,1]$ and let $D_{\alpha}$ be the standard $\alpha$-stable subordinator, that is to say an increasing Lévy process such that for all $s,t \geq 
0$, $\mathbb{E}[e^{-sD_{\alpha}(t)}] = \exp(-ts^{\alpha})$ (see 
\cite{bertoin96_LevyProcesses} for more details on Lévy processes and subordinators). 
Let $E^{\alpha}$ be the inverse stable subordinator defined as the generalized 
inverse of $D_{\alpha}$, \textit{i.e.}
\begin{align*}
E^{\alpha}(t) = D_{\alpha}^{\leftarrow}(t) := \inf \{ u > 0 \;|\; D_{\alpha}(u) > t\} \quad \text{for}\; t\geq 0.
\end{align*}
This process $E^{\alpha}$ is also non decreasing and is the correct random time scaling to define the fractional Poisson process.

\begin{rem}
    In infinite ergodic theory, this process is also defined as the Mittag-Leffler process and appears as a functional limit for averages of integrable observables for infinite CEMPT \cite{OwadaSamorodnitsky15,Sera20}.
\end{rem}

\begin{prop}{\textup{\cite[Theorem 2.2]{MNV11_FFPAndTheInverseStableSubordinator}}}
    Let $N$ be a standard Poisson point process of parameter $\lambda > 0$ and $E^{\alpha}$ an inverse stable subordinator independent of $N$. Then, $N \circ E^{\alpha}$ is a fractional Poisson process of parameter $\alpha$ and $\lambda$.  
\end{prop}

  In particular, for $\alpha = 1/2$, $E^{1/2}$ has the same law as $\sqrt{2}\, \bar{B}$ where $\bar{B}$ is the supremum of the Brownian motion \cite[Theorem 2.2.9 p.95]{App09_LevyProcessesAndStochasticCalculus}. We are now able to show Lemma \ref{lem:equivalence_PS23_article} showing that the point process obtained in \cite[Theorem 3.10]{PS23} is the fractional Poisson process of parameter $1/2$.

\begin{proof}[Proof (of Lemma \ref{lem:equivalence_PS23_article})]
    We have 
    \begin{align*}
    N \circ L_0\left(\frac{2}{\pi} \cdot\right) & \eqlaw N \circ \bar{B} \left(\frac{2}{\pi} \cdot\right) \eqlaw N \circ \sqrt{\frac{\pi}{2}} \bar{B} \eqlaw N_{\frac{\sqrt{\pi}}{2}} \circ \sqrt{2}\,\bar{B}\\
    & \eqlaw N_{\Gamma(3/2)}\circ E^{1/2} \eqlaw \fPp_{1/2}(\Gamma(1+ 1/2))\,,
    \end{align*}
    where $N_{\lambda}$ is a Poisson point process of parameter $\lambda$ and every processes and point processes are assumed to be independent.
\end{proof}

  This closes the gap between the two approaches when $\alpha = 1/2$, the right scale when considering a $\mathbb{Z}$-extension. However, as we have seen, the construction of the fractional Poisson point process $N\circ E^{\alpha}$ could provide a useful approach for $\alpha \neq 1/2$ in future works.

\section*{Funding}

DBT is partially supported by CMUP, which is financed by national funds through FCT – Fundação para a Ciência e a Tecnologia, I.P., under the project with reference UIDB/00144/2020. He is also financed by CNRS and École Polytechnique.

\bibliographystyle{amsalpha}
	
\bibliography{main.bib}

\newcommand{\etalchar}[1]{$^{#1}$}
\providecommand{\bysame}{\leavevmode\hbox to3em{\hrulefill}\thinspace}
\providecommand{\MR}{\relax\ifhmode\unskip\space\fi MR }
\providecommand{\MRhref}[2]{%
  \href{http://www.ams.org/mathscinet-getitem?mr=#1}{#2}
}
\providecommand{\href}[2]{#2}
\begin{thebibliography}{GKMR20}

\bibitem[Aar97]{Aar97}
Jon Aaronson, \emph{An introduction to infinite ergodic theory}, Mathematical Surveys and Monographs, vol.~50, American Mathematical Society, Providence, RI, 1997. \MR{1450400}

\bibitem[Alv20]{Alv20}
Jos\'{e}~F. Alves, \emph{Nonuniformly hyperbolic attractors---geometric and probabilistic aspects}, Springer Monographs in Mathematics, Springer, Cham, [2020] \copyright 2020. \MR{4226156}

\bibitem[App09]{App09_LevyProcessesAndStochasticCalculus}
David Applebaum, \emph{L\'{e}vy processes and stochastic calculus}, second ed., Cambridge Studies in Advanced Mathematics, vol. 116, Cambridge University Press, Cambridge, 2009. \MR{2512800}

\bibitem[Ber96]{bertoin96_LevyProcesses}
Jean Bertoin, \emph{L{\'e}vy processes}, vol. 121, Cambridge university press Cambridge, 1996.

\bibitem[BGT89]{Bingham89_RegularVariation}
Nicholas~H. Bingham, Charles~M. Goldie, and Jef~L. Teugels, \emph{Regular variation}, no.~27, Cambridge university press, 1989.

\bibitem[BTF23]{BF23}
Dylan Bansard-Tresse and Jorge~Milhazes Freitas, \emph{Inducing techniques for quantitative recurrence and applications to misiurewicz maps and doubly intermittent maps}, to appear in Annales de l'Institut Henri Poincaré Probabilités Statistiques (2023).

\bibitem[BZ01]{BZ01}
Xavier Bressaud and Roland Zweim\"{u}ller, \emph{Non exponential law of entrance times in asymptotically rare events for intermittent maps with infinite invariant measure}, Ann. Henri Poincar\'{e} \textbf{2} (2001), no.~3, 501--512. \MR{1846853}

\bibitem[CC13]{ChazottesCollet13}
J.-R. Chazottes and P.~Collet, \emph{Poisson approximation for the number of visits to balls in non-uniformly hyperbolic dynamical systems}, Ergodic Theory Dynam. Systems \textbf{33} (2013), no.~1, 49--80. \MR{3009103}

\bibitem[CG93]{CG93_Statisticsofclosevisitstotheindifferentfixedpointofanintervalmap}
Pierre Collet and Antonio Galves, \emph{Statistics of close visits to the indifferent fixed point of an interval map}, J. Statist. Phys. \textbf{72} (1993), no.~3-4, 459--478. \MR{1239564}

\bibitem[CI95]{CampaninoIsola95}
Massimo Campanino and Stefano Isola, \emph{Statistical properties of long return times in type {I} intermittency}, Forum Math. \textbf{7} (1995), no.~3, 331--348. \MR{1325560}

\bibitem[DT23]{DT23}
Mark Demers and Mike Todd, \emph{A trichotomy for hitting times and escape rates for a class of unimodal maps}, arXiv preprint arXiv:2309.09624 (2023).

\bibitem[FFTV16]{FFTV16}
Ana Cristina~Moreira Freitas, Jorge~Milhazes Freitas, Mike Todd, and Sandro Vaienti, \emph{Rare events for the {M}anneville-{P}omeau map}, Stochastic Process. Appl. \textbf{126} (2016), no.~11, 3463--3479. \MR{3549714}

\bibitem[GKMR20]{Gorenflo20_book_Mittag-Leffler}
Rudolf Gorenflo, Anatoly~A. Kilbas, Francesco Mainardi, and Sergei Rogosin, \emph{Mittag-{L}effler functions, related topics and applications}, Springer Monographs in Mathematics, Springer, Berlin, 2020, Second edition. \MR{4179587}

\bibitem[Gou04]{Gouezel04_PhD}
S{\'e}bastien Gou{\"e}zel, \emph{Vitesse de d{\'e}corr{\'e}lation et th{\'e}oremes limites pour les applications non uniform{\'e}ment dilatantes}, Ph.D. thesis, Ph. D. Thesis, Ecole Normale Sup{\'e}rieure, 2004.

\bibitem[Hay13]{Haydn13_EntryAndReturnTimesDistribution}
N.~T.~A. Haydn, \emph{Entry and return times distribution}, Dyn. Syst. \textbf{28} (2013), no.~3, 333--353. \MR{3170620}

\bibitem[HLV05]{HLV05}
N.~Haydn, Y.~Lacroix, and S.~Vaienti, \emph{Hitting and return times in ergodic dynamical systems}, Ann. Probab. \textbf{33} (2005), no.~5, 2043--2050. \MR{2165587}

\bibitem[Kal02]{Kal02_SecondEdition}
Olav Kallenberg, \emph{Foundations of modern probability}, second ed., Probability and its Applications (New York), Springer-Verlag, New York, 2002. \MR{1876169}

\bibitem[Las03]{Las03}
Nick Laskin, \emph{Fractional {P}oisson process}, vol.~8, 2003, Chaotic transport and complexity in classical and quantum dynamics, pp.~201--213. \MR{2007003}

\bibitem[LFF{\etalchar{+}}16]{LFFF16}
Valerio Lucarini, Davide Faranda, Ana~Cristina Freitas, Jorge Miguel~Milhazes Freitas, Mark Holland, Tobias Kuna, Matthew Nicol, Mike Todd, and Sandro Vaienti, \emph{Extremes and recurrence in dynamical systems}, Pure and Applied Mathematics (Hoboken), John Wiley \& Sons, Inc., Hoboken, NJ, 2016. \MR{3558780}

\bibitem[LP18]{LastPenrose18_LecturesOnPoissonProcess}
G\"{u}nter Last and Mathew Penrose, \emph{Lectures on the {P}oisson process}, Institute of Mathematical Statistics Textbooks, vol.~7, Cambridge University Press, Cambridge, 2018. \MR{3791470}

\bibitem[LSV99]{LSV97}
Carlangelo Liverani, Beno\^{\i}t Saussol, and Sandro Vaienti, \emph{A probabilistic approach to intermittency}, Ergodic Theory Dynam. Systems \textbf{19} (1999), no.~3, 671--685. \MR{1695915}

\bibitem[Mar17]{Mar17}
Jens Marklof, \emph{Entry and return times for semi-flows}, Nonlinearity \textbf{30} (2017), no.~2, 810--824. \MR{3604363}

\bibitem[MNV11]{MNV11_FFPAndTheInverseStableSubordinator}
Mark~M. Meerschaert, Erkan Nane, and P.~Vellaisamy, \emph{The fractional {P}oisson process and the inverse stable subordinator}, Electron. J. Probab. \textbf{16} (2011), no. 59, 1600--1620. \MR{2835248}

\bibitem[MS19]{MS19_Stochastic_models_for_fractional_calculus}
Mark~M. Meerschaert and Alla Sikorskii, \emph{Stochastic models for fractional calculus}, second ed., De Gruyter Studies in Mathematics, vol.~43, De Gruyter, Berlin, 2019. \MR{3971272}

\bibitem[OS15]{OwadaSamorodnitsky15}
Takashi Owada and Gennady Samorodnitsky, \emph{Functional central limit theorem for heavy tailed stationary infinitely divisible processes generated by conservative flows}, Ann. Probab. \textbf{43} (2015), no.~1, 240--285. \MR{3298473}

\bibitem[PS10]{PeneSaussol10_BackToBallsInBilliards}
Fran\c{c}oise P\`ene and Beno\^{i}t Saussol, \emph{Back to balls in billiards}, Comm. Math. Phys. \textbf{293} (2010), no.~3, 837--866. \MR{2566164}

\bibitem[PS24]{PS23}
\bysame, \emph{Quantitative recurrence for $t,t^{-1}$ tranformation}, Probability Theory and Related Fields (2024).

\bibitem[PSZ13]{PSZ11}
Fran\c{c}oise P\`ene, Beno\^{i}t Saussol, and Roland Zweim\"{u}ller, \emph{Recurrence rates and hitting-time distributions for random walks on the line}, Ann. Probab. \textbf{41} (2013), no.~2, 619--635. \MR{3077520}

\bibitem[PSZ17]{PSZ13}
\bysame, \emph{Return- and hitting-time limits for rare events of null-recurrent {M}arkov maps}, Ergodic Theory Dynam. Systems \textbf{37} (2017), no.~1, 244--276. \MR{3590502}

\bibitem[RZ20]{RZ20}
Simon Rechberger and Roland Zweim\"{u}ller, \emph{Return- and hitting-time distributions of small sets in infinite measure preserving systems}, Ergodic Theory Dynam. Systems \textbf{40} (2020), no.~8, 2239--2273. \MR{4120779}

\bibitem[Sar01]{Sar01_NullRecurrentPotentials}
Omri~M. Sarig, \emph{Thermodynamic formalism for null recurrent potentials}, Israel J. Math. \textbf{121} (2001), 285--311. \MR{1818392}

\bibitem[Sau09]{Sau09_Survey_AnIntroductionToQuantitativeRecurrenceInDynamicalSystems}
Beno{\^{\i}}t Saussol, \emph{An introduction to quantitative {P}oincar\'{e} recurrence in dynamical systems}, Rev. Math. Phys. \textbf{21} (2009), no.~8, 949--979. \MR{2568049}

\bibitem[SB22]{SuBunimovitch22}
Yaofeng Su and Leonid~A. Bunimovich, \emph{Poisson approximations and convergence rates for hyperbolic dynamical systems}, Comm. Math. Phys. \textbf{390} (2022), no.~1, 113--168. \MR{4381186}

\bibitem[Ser20]{Sera20}
Toru Sera, \emph{Functional limit theorem for occupation time processes of intermittent maps}, Nonlinearity \textbf{33} (2020), no.~3, 1183--1217. \MR{4063962}

\bibitem[Tha80]{Tha80_EstimatesInvariantDensities}
Maximilian Thaler, \emph{Estimates of the invariant densities of endomorphisms with indifferent fixed points}, Israel J. Math. \textbf{37} (1980), no.~4, 303--314. \MR{599464}

\bibitem[Tha83]{Tha83}
\bysame, \emph{Transformations on {$[0,\,1]$} with infinite invariant measures}, Israel J. Math. \textbf{46} (1983), no.~1-2, 67--96. \MR{727023}

\bibitem[Yas18]{Yas18}
Nasab Yassine, \emph{Quantitative recurrence of some dynamical systems preserving an infinite measure in dimension one}, Discrete and Continuous Dynamical Systems \textbf{38} (2018), 343--361.

\bibitem[Yas24]{Yas24_quantitativerecurrencezextensionthreedimensional}
\bysame, \emph{Quantitative recurrence for $\mathbb{Z}$-extension of three-dimensional axiom a flows}, preprint, 2024.

\bibitem[You99]{You99}
Lai-Sang Young, \emph{Recurrence times and rates of mixing}, Israel J. Math. \textbf{110} (1999), 153--188. \MR{1750438}

\bibitem[Zwe00]{Zwe98}
Roland Zweim\"{u}ller, \emph{Ergodic properties of infinite measure-preserving interval maps with indifferent fixed points}, Ergodic Theory Dynam. Systems \textbf{20} (2000), no.~5, 1519--1549. \MR{1786727}

\bibitem[Zwe07a]{Zwe07_InfiniteMeasurePreservingTransformationsWithCompactFirstRegeneration}
\bysame, \emph{Infinite measure preserving transformations with compact first regeneration}, J. Anal. Math. \textbf{103} (2007), 93--131. \MR{2373265}

\bibitem[Zwe07b]{Zwei07}
\bysame, \emph{Mixing limit theorems for ergodic transformations}, J. Theoret. Probab. \textbf{20} (2007), no.~4, 1059--1071. \MR{2359068}

\bibitem[Zwe08]{Zwe08}
\bysame, \emph{Waiting for long excursions and close visits to neutral fixed points of null-recurrent ergodic maps}, Fund. Math. \textbf{198} (2008), no.~2, 125--138. \MR{2369126}

\bibitem[Zwe16]{Zwe16}
\bysame, \emph{The general asymptotic return-time process}, Israel J. Math. \textbf{212} (2016), no.~1, 1--36. \MR{3504316}

\bibitem[Zwe19]{Zwe18}
\bysame, \emph{Hitting-time limits for some exceptional rare events of ergodic maps}, Stochastic Process. Appl. \textbf{129} (2019), no.~5, 1556--1567. \MR{3944776}

\bibitem[Zwe22]{Zwe22}
\bysame, \emph{Hitting times and positions in rare events}, Ann. H. Lebesgue \textbf{5} (2022), 1361--1415. \MR{4526257}

\end{thebibliography}


\appendix

\section{Additional proofs}

\subsection{Proof of Corollary \ref{cor:HTS_neighborhood_0_normalization_gamma}}

  Corollary \ref{cor:HTS_neighborhood_0_normalization_gamma} can be easily deduced from Theorem \ref{thm:Hitting_to_0_infinite} and the following Lemma \ref{lem:Comparison_renormalization_0}.
\begin{lem}
    \label{lem:Comparison_renormalization_0}
    Let $(B_n)_{n\in \mathbb{N}}$ and $(\eta_n)_{n\in \mathbb{N}}$ be defined as in Theorem \ref{thm:Hitting_to_0_infinite}. Let $Q(B_n) := B_n \cap T^{-1}(B_n^c)$. Then we have 
    \begin{align}
        \label{eq:equivalence_renormalization_0}
        \gamma(\mu(T_2^{-1}B_n)) = \gamma\left(\mu(Q(B_n))\right) \uset{\widesim}{n\to +\infty} \left(\frac{1}{\Gamma(1+\alpha)\Gamma(1 - \alpha)}\right)^{1/\alpha} \frac{1}{I(\eta_n)}.
    \end{align}
\end{lem}

\begin{proof}[Proof (of Lemma \ref{lem:Comparison_renormalization_0})]
      The first equality follows directly from the $T$-invariance of $\mu$. Recall \eqref{eq:formula_density_LSV_map} stating that the density $\rho := \dd\mu/\dd\leb$ is of the special form
    \begin{align*}
        \rho(x) = h_0(x)\times\frac{x}{x - T_1^{-1} x} \,,\; \forall x\in (0,1],
    \end{align*}
    where $h_0 : [0,1] \to \mathbb{R}$ is a positive continuous function (see for example \cite[Lemma 4]{Tha83}).
      Furthermore, recall that $\gamma : s \mapsto 1/b(1/s)$ where $b$ is the asymptotic inverse of $a$. Since we are studying a special AFN-map we know that (see \cite[Theorem 4]{Zwe98}):
    \begin{align*}
        a_n \uset{\widesim}{n\to +\infty} \frac{2(1-\alpha)}{h_0(0)\Gamma(1+\alpha)\Gamma(2 - \alpha)\alpha^{\alpha}} n^{\alpha}.
    \end{align*}
    Hence, we can choose
    \begin{align*}
        \gamma(s) =  \left(\frac{2(1-\alpha)}{h_0(0)\Gamma(1+\alpha)\Gamma(2 - \alpha)\alpha^{\alpha}}\right)^{1/\alpha} s^{1/\alpha}\,,\; \forall s\geq 0.
    \end{align*}
      On the other part, we have 
    \begin{align*}
        \mu(Q(B_n)) &= \int_{T_1^{-1}\eta_n}^{\eta_n} \rho(x)\dd x \uset{\widesim}{n\to +\infty} h_0(0) \int_{T_1^{-1}\eta_n}^{\eta_n} \frac{x}{x - T_1^{-1}x}\dd x  = h_0(0) \xi_n \frac{\eta_n - T_1^{-1}\eta_n}{\xi_n - T_1^{-1}\xi_n}\quad \text{for some} \; \xi_n \in [T_1^{-1}\eta_n, \eta_n] \\
        & \uset{\widesim}{n\to +\infty} h_0(0) \xi_n \uset{\widesim}{n\to +\infty} h_0(0) \eta_n.
    \end{align*}
    Hence,
    \begin{align}
        \label{eq:equivalence_gamma_mu_Bn}
        \gamma(\mu(Q(B_n)) \uset{\widesim}{n\to +\infty} \left(\frac{2(1-\alpha)}{\Gamma(1+\alpha)\Gamma(2 - \alpha)\alpha^{\alpha}}\right)^{1/\alpha} \eta_n^{1/\alpha}\,.
    \end{align}
    On the other side, we have
    \begin{align}
        \label{eq_equivalence_I_Bn}
        I(\eta_n) = \int_{\eta_n}^1 \frac{\dd x}{x - T_1^{-1}x} \uset{\widesim}{n\to +\infty} \int_{\eta_n}^1 \frac{\dd x}{2^{1/\alpha}x^{1/\alpha + 1}} = \frac{\alpha}{2^{1/\alpha}}\left(\frac{1}{\eta_n^{1/\alpha}} - 1\right)\uset{\widesim}{n\to +\infty} \frac{\alpha}{2^{1/\alpha}\eta_n^{1/\alpha}}, 
    \end{align}
    where we use the fact that $x-T_1^{-1}x \uset{\widesim}{x\to 0^+} 2^{1/\alpha}x^{1/\alpha + 1}$ since $T_1^{-1}x = x - 2^{1/\alpha}x^{1/\alpha+1} + o(x^{1/\alpha})$. Putting together \eqref{eq:equivalence_gamma_mu_Bn} and \eqref{eq_equivalence_I_Bn} gives \eqref{eq:equivalence_renormalization_0} using that $\Gamma(2 - \alpha) = (1 - \alpha)\Gamma(1 - \alpha)$.
\end{proof}

\subsection{Proof of Lemma \ref{lem:comparison_renormalizations_for_0_p=1}}

  Recall that $B_n = [0,c_n]$ and let $E_n := T_2^{-1}B_n = Y \cap \{r_Y > n\} = [1/2,\delta_n]$. Recall that we have chosen the renormalization of $\mu$ so that $\mu(Y) = 1$. We know by \cite[Theorem 4]{Zwe98} that, in the barely infinite case, we have $a_n \sim \frac{2n}{h_0(0)\log(n)}$ and thus $b_n := \frac{h_0(0)}{2}n\log(n)$ is an asymptotic inverse of $(a_n)_{n\geq 0}$ (recall that $\gamma(s) := b(s^{-1})^{-1}$). \\

  Fix some $p\geq 0$ and assume $n \geq p$ (this is no issue as we are interested by the limit on $n$). Let $\widetilde{r}^{(p)}_{B_n}$ be the random variable defined as follows:
\begin{align*}
    \widetilde{r}^{(p)}_{B_n} := h_{E_n}\circ T_{E_p} \mathbf{1}_{E_n} \,,
\end{align*}
where $h_A(x) := \inf\{n\geq 0\;|\; T^n(x)\in A\}$.

\begin{lem}
    \label{lem:bounded_distortion_to_1_again}
    There exist a constant $C_p$, with $C_p \xrightarrow[p\to +\infty]{} 1$ such that for all $n\geq p$, 
    \begin{align*}
        C_p^{-1} \mu(E_n)\leq \frac{\mathbb{E}_{\mu_Y}[\widetilde{r}^{(p)}_{B_n}]}{\mathbb{E}_{\mu_{E_p}}[r_{B_n}]} \leq C_p \mu(E_n)\,.
    \end{align*}
\end{lem}

\begin{proof}[Proof (of Lemma \ref{lem:bounded_distortion_to_1_again})]
    As usual, we use the bounded distortion result Corollary \ref{cor:distortion_bounds_comparison_measures}  and work branch by branch. Let 
    \begin{align*}
        A_n := \{[0a_0^{j-1}b_0^{k}],\; a_0 > n, \exists!i \; a_i = p, \; b_0 =0, \; b_1 > p,\; b_{k-1} = 0, \; b_k > n, \; \forall 0\leq \ell \leq k-1 \; b_k\leq n\}
    \end{align*}
    be a partition in branches of $\mathbf{1}_{E_n}$ for $\widetilde{r}_{B_n}^{(p)}$. Then, we have
    
    \begin{align*}
        \mathbb{E}_{\mu_Y}[\widetilde{r}^{(p)}_{B_n}] &= \sum_{[0a_0^{j-1}b_0^{k}] \in A_n} (k-1)\mu[0a_0^{j-1}b_0^{k}]\\
        & \leq (1+C\diam(E_p))\sum_{[0a_0^{j-1}b_0^{k}] \in A_n} (k-1) \frac{\mu[0a_0^{j-1}0(>p)]\mu[b_0^k]}{\mu[0(>p)]} \\
        & \leq C'_p \frac{\mu(E_n)}{\mu(E_p)} \mathbb{E}_{\mu_Y}[h_{E_n} \mathbf{1}_{E_p}]\\
        & \leq C'_p \frac{\mu(E_n)}{\mu(E_p)} \mathbb{E}_{\mu_Y}[(r_{B_n} - 1) \mathbf{1}_{E_p}]\\
        & \leq C_p \frac{\mu(E_n)}{\mu(E_p)} \mathbb{E}_{\mu_Y}[r_{B_n}\mathbf{1}_{E_p}].
    \end{align*}
      The lower bound is obtained similarly. Since $\diam(E_p) \xrightarrow[p\to +\infty]{}$, we have $C_p \xrightarrow[p\to +\infty]{} 1$.
\end{proof}

\begin{lem}
    \label{lem:equivalence_wandering_rate_tilde_r_B_n}
    For a fixed $p\geq 0$, we have the following asymptotic result:
    \begin{align*}
        \mathbb{E}_{\mu_Y}[\widetilde{r}_{B_n}^{(p)}] \sim \int (r_Y \wedge n)\,\dd\mu_Y = w_n(Y).
    \end{align*}
\end{lem}

\begin{proof}[Proof (of Lemma \ref{lem:equivalence_wandering_rate_tilde_r_B_n})]
    We look at the Birkhoff's sums for the induced map on $Y$ and consider the map $\ell_{n} := r_Y\mathbf{1}_{E_n^c} = r_Y \mathbf{1}_{\{r_Y \leq n\}}$. Let $j\geq 1$, $x\in Y$ and assume for simplicity that $T_Y^jx \in E_n$. Write $k(j) := \sup \{i \geq 0\;|\; r_{E_n}^{Y,(i)} < j\}$ (in particular, with the hypothesis we made, we have $j = r_{E_n}^{Y,(k(i)+1)}$). We have 
    \begin{align*}
        \frac{1}{j}S^Y_j\ell_n(x)  & = \frac{1}{j}\sum_{k = 0}^{j-1} r_Y(T_Y^{k}x)\mathbf{1}_{E_n^c}(T_Y^kx)\\
        & = \frac{1}{j} \sum_{k = 1}^{k(j)} \sum_{i = 1}^{r^Y_{E_n}\circ T_{E_n}^k(x) - 1} r_Y(T_Y^i \circ T_{E_n}^k x) + \frac{1}{j}h_{E_n}\\
        & = \frac{1}{j} \sum_{k = 1}^{k(j)} \left(\sum_{i = r^Y_{E_p}\circ T_{E_n}^k(x)}^{r^Y_{E_n}\circ T_{E_n}^k(x) - 1} r_Y(T_Y^i \circ T_{E_n}^k x) + \sum_{i = 1}^{r^Y_{E_p}\circ T_{E_n}^k(x) - 1} r_Y(T_Y^i \circ T_{E_n}^k x)\right) + \frac{1}{j}h_{E_n}.
    \end{align*}

      On one side, we have 
    \begin{align*}
        \frac{1}{j} \sum_{k = 1}^{k(j)}\sum_{i = 1}^{r^Y_{E_p}\circ T_{E_n}^k(x) - 1} r_Y(T_Y^i \circ T_{E_n}^k x) &= \frac{1}{j} \sum_{k = 1}^{k(j)} h_{E_p} \circ T_Y \circ T_{E_n}^k(x)\\
        & = \frac{1}{j} \sum_{k = 1}^{j-1} \left(\mathbf{1}_{E_n}h_{E_p} \circ T_Y\right) \circ T_Y^k(x)\\
        &\sim \frac{1}{j} S_j^{Y}\left( \mathbf{1}_{E_n}h_{E_p} \circ T_Y\right) \xrightarrow[j\to +\infty]{\mu_Y \; a.e.} \mathbb{E}_{\mu_Y}[\mathbf{1}_{E_n}h_{E_p} \circ T_Y]
    \end{align*}
      On the other side, we have 
    \begin{align*}
        \frac{1}{j} \sum_{k = 1}^{k(j)} \sum_{i = r^Y_{E_p}\circ T_{E_n}^k(x)}^{r^Y_{E_n}\circ T_{E_n}^k(x) - 1} r_Y\circ T_Y^i \circ T_{E_n}^k (x) & = \frac{1}{j} \sum_{k=1}^{k(j)} h_{E_n} \circ T_{E_p} \circ T_{E_n}^k(x) \quad \text{using $E_n \subset E_p$}
    \end{align*}
    
      The $h_{E_n}$ instead of $r_{E_n}$ comes from the possibility that the sum is empty if $r_{E_p}^Y = r_{E_n}^Y$ and the $0$ we get is exactly $h_{E_n}$ in that case. \\

      For $\widetilde{r}_{B_n}^{(p)}$, we have 
    \begin{align*}
        \frac{1}{j}S^Y_j \widetilde{r}_{B_n}^{(p)}(x) &= \frac{1}{j} \sum_{k=0}^{j -1} \widetilde{r}_{B_n}^{(p)} \circ T_Y^k(x)\\
        & = \frac{1}{j} \sum_{k = 1}^{k(j)} h_{E_n} \circ T_{E_p} \circ T_{E_n}^k(x) \qquad \text{by definition of $\widetilde{r}_{B_n}^{(p)}$.}
    \end{align*}

      Putting everything together, we get by Birkhoff's theorem,
    \begin{align*}
        \mathbb{E}_{\mu_Y}[\widetilde{r}_{B_n}^{(p)}] = \mathbb{E}_{\mu_Y}[\ell_n] - \mathbb{E}_{\mu_Y}[\mathbf{1}_{E_n}h_{E_p}\circ T_Y].
    \end{align*}
    In particular, when $n$ goes to $+\infty$, we have the asymptotic equivalence
    \begin{align*}
        \mathbb{E}_{\mu_Y}\big[\widetilde{r}_{B_n}^{(p)}\big] \uset{\widesim}{n\to +\infty} \mathbb{E}_{\mu_Y}[\ell_n].
    \end{align*}

      Finally, we have 
    \begin{align*}
        \mathbb{E}_{\mu_Y}[\ell_n] = \int_Y r_Y \mathbf{1}_{\{r_Y \leq n\}}\,\dd\mu_Y = w_n(Y) - n\mu_Y(r_Y > n).
    \end{align*}
      However, in the particular barely infinite case, we have $\mu_Y(r_Y > n) \asymp 1/n$ whereas $w_n(Y) \asymp \log(n)$. Thus, 
    \begin{align*}
        \mathbb{E}_{\mu_Y}[\ell_n] \uset{\widesim}{n\to +\infty} w_n(Y).
    \end{align*}
    This yields the wanted result $\mathbb{E}_{\mu_Y}\big[\widetilde{r}_{B_n}^{(p)}\big] \uset{\widesim}{n\to +\infty} w_n(Y)$.
\end{proof}

\begin{proof}[Proof (of Lemma \ref{lem:comparison_renormalizations_for_0_p=1})]
  With Lemma \ref{lem:bounded_distortion_to_1_again} and \ref{lem:equivalence_wandering_rate_tilde_r_B_n}, we obtain
\begin{align*}
    C_p^{-1}\mu(E_n) &\lesssim \frac{w_n(Y)}{\mathbb{E}_{\mu_{E_p}}[r_{B_n}]} \lesssim C_p\mu(E_n)\\
     C_p^{-1}&\lesssim \frac{\mathbb{E}_{\mu_{Y}}[r_{B_n}]^{-1} w_n(Y)/\mu(E_n)}{\mathbb{E}_{\mu_{Y}}[r_{B_n}]^{-1}\mathbb{E}_{\mu_{E_p}}[r_{B_n}]} \lesssim C_p,
\end{align*}
where $u_n \lesssim v_n$ means that $\limsup_{n\to+\infty} u_n/v_n \leq 1$.\\

  Taking the limit in $n$ ($\limsup$ and $\liminf$) and since the convergence in Theorem \ref{thm:Collet_Galves_p_equals_1} is true $\mu_{E_p} \in \mathcal{L}(\mu)$, we get, taking finally the limit $p\to +\infty$,  
\begin{align*}
    C_p^{-1} \lesssim \frac{w_n(Y)}{\mu(E_n)}\mathbb{E}_{\mu_Y}[r_{B_n}]^{-1} \lesssim C_p.
\end{align*}
Taking afterwards $p\to +\infty$, it yields
\begin{align*}
     \frac{w_n(Y)}{\mu(E_n)} \uset{\widesim}{n\to +\infty} \mathbb{E}_{\mu_Y}[r_{B_n}].
\end{align*}
However, 
\begin{align*}
    \gamma(\mu(E_n))^{-1} &= b(\mu(E_n)^{-1}) \uset{\widesim}{n\to +\infty} \frac{h_0(0)}{2\mu(E_n)}\log(\mu(E_n)^{-1}) \\
    & \uset{\widesim}{n\to +\infty} \frac{h_0(0)}{2\mu(E_n)}\log(n) \uset{\widesim}{n\to +\infty} w_n(Y)/\mu(E_n).
\end{align*}
\end{proof}

\begin{rem}
    In fact, in Lemma \ref{lem:equivalence_wandering_rate_tilde_r_B_n}, we actually showed that $\mathbb{E}_{\mu_Y}[\widetilde{r}_{B_n}^{(p)}] \uset{\widesim}{n\to +\infty} w_n(Y) - n\mu_Y(Y\cap \{r_Y > n\})$, whenever $p > 1$. In the case $p> 1$, both terms are of same order. Since we also have 
    \begin{align*}
        w_n(Y) = \sum_{k=0}^{n-1} \mu(Y \cap \{r_Y > k\})
    \end{align*}
    and we know that $\mu(Y \cap \{r_Y > n\}) \uset{\widesim}{n\to +\infty} Cn^{-\alpha}$ for some constant $\alpha$, we obtain
    \begin{align*}
        \mathbb{E}_{\mu_Y}[r_{B_n}] \uset{\widesim}{n\to +\infty} \frac{\alpha}{1 -\alpha}n,
    \end{align*}
    which can also be proven with Theorem \ref{thm:Hitting_to_0_infinite} which implies
    \begin{equation}
        \frac{1}{n}r_{B_n} \xRightarrow[n\to+\infty]{\mu_Y} \mathcal{J}_{\alpha}
    \end{equation}
    In particular, since $\mathbb{E}[\mathcal{J}_{\alpha}] = \alpha/(1 - \alpha)$, it also gives $\mathbb{E}_{\mu_Y}[r_{B_n}] \uset{\widesim}{n\to +\infty} \alpha n/(1- \alpha)$.
\end{rem}

\end{document}